\documentclass[12pt,twoside]{preprint}
\usepackage[margin=1.3in]{geometry}
\usepackage{amssymb}
\usepackage{amsmath}
\usepackage{hyperref}
\usepackage{breakurl}
\usepackage{mhenvs}
\usepackage{mhequ}
\usepackage{booktabs}
\usepackage{tikz} 
\usepackage{mathrsfs}
\usepackage{times}
\usepackage{microtype}
\usepackage{comment}
\usepackage{wasysym}
\usepackage{centernot}
\usepackage{appendix}

\usepackage{color}



\newtheorem{assumption}[lemma]{Assumption}

\newcommand{\eqdef}{\stackrel{\mbox{\tiny def}}{=}}

\newcommand{\C}{\mathbb{C}}
\newcommand{\E}{\mathbb{E}}
\newcommand{\N}{\mathbb{N}}
\renewcommand{\P}{{\mathbb P}}

\newcommand{\R}{\mathbb{R}}
\newcommand{\T}{\mathbb{T}}
\newcommand{\Z}{\mathbb{Z}}

\newcommand{\nn}{\mathfrak{m}}

\newcommand{\KK}{\mathfrak{K}}

\newcommand{\Cc}{\mathcal{C}}

\newcommand{\ga}{\gamma}

\newcommand{\eps}{\varepsilon}
\newcommand{\ka}{\kappa}

\newcommand{\si}{\sigma}

\DeclareMathOperator{\dist}{dist}

\renewcommand{\subset}{\subseteq}

\newcommand{\msf}{\mathsf}

\newcommand{\virg}[1]{``{#1}''}
\newcommand{\ang}[1]{\left\langle #1 \right\rangle} 
\newcommand{\nug}{\nu_{\gamma}} 							
\newcommand{\nuN}{\nu_{\gamma,N}} 						
\newcommand{\kb}{\mathbf{k}} 								
\newcommand{\norm}[2]{\left\lvert #2 \right\rvert_{#1}}
\newcommand{\Norm}[2]{\left\Vert #2 \right\Vert_{#1}}

\newcommand{\supp}{supp}

\newcommand{\aste}{\ast_{\epsilon}}
\newcommand{\betareg}{\beta}								
\newcommand{\pptg}{\tilde{\mathfrak{p}}_{\gamma}}			
\newcommand{\ppg}{\mathfrak{p}_{\gamma}}					
\newcommand{\pp}{\mathfrak{p}}							
\newcommand{\D}{\mathcal{D}}								
\newcommand{\XgH}{X_{\gamma}^{high}}

\newcommand{\Err}{Err}
\newcommand{\Oo}{\mathcal{O}}								
\newcommand{\Qm}{Q_{\nn}}								
\renewcommand{\theequation}{\arabic{section}.\arabic{subsection}.\arabic{equation}}
\def\aatg#1{\tilde{\mathfrak{a}}^{\gamma}_{#1}}			
\def\aag#1{\mathfrak{a}^{\gamma}_{#1}}					
\def\aa#1{\mathfrak{a}_{#1}}								
\def\HXXj#1{H(X^{(j)}|X|^{#1},\mathfrak{c})}				
\def\HXXgj#1#2{H(\Xg^{(j)} |\Xg|^{#1},#2)}				
\def\pg#1{p^{\hg(#1)}} 									
\def\pgm#1{p_{\nn}(#1,\sigma)} 							

\newcommand{\hg}{h_{\ga}} 								
\newcommand{\kg}{\ka_\ga} 								
\newcommand{\Kg}{K_\ga} 								
\newcommand{\LN}{\Lambda_N} 							
\newcommand{\Le}{\Lambda_\eps}		 					
\newcommand{\Hg}{\mathscr{H}_{\ga}} 						
\newcommand{\lbg}{\lambda_{\ga}} 							
\newcommand{\cg}{c_{\ga}}		 	 					
\newcommand{\mg}{m_\ga} 								
\newcommand{\Mg}{M_\ga}  								
\newcommand{\Dg}{\Delta_\ga}		 					
\newcommand{\Eg}{E_\ga} 								
\newcommand{\Xg}{X_{\ga}}		  						
\newcommand{\Xng}{X_\ga^0} 							
\newcommand{\Xn}{X^0} 							

\newcommand{\Ce}{\mathfrak{c}_\eps} 						

\def\Pg#1{P_{#1}^{\gamma}} 								
\newcommand{\Zg}{Z_{\ga}}  								
\newcommand{\Rg}{R_{\ga,t}}		 						
\newcommand{\hRg}{\hat{R}_{\ga,t}}		 				
\newcommand{\ZgH}{\Zg^{\mathrm{high}}} 					

\newcommand{\taun}{\tau_{\ga,\nn}}							 

\newcommand{\CG}{\mathfrak{c}_\ga} 						
\newcommand{\Ex}{\msf{Ext}} 								

\newcommand{\hKg}{\hat{K}_{\ga}} 							

\begin{document}

\title{Convergence of Glauber dynamic on Ising-like models with Kac interaction to $\Phi^{2n}_2$}
\author{Massimo Iberti}
\institute{University of Warwick, UK, \email{M.iberti@warwick.ac.uk}}

\maketitle

\begin{abstract}
It has been recently shown by H.Weber and J.C. Mourrat, for the two-dimensional Ising-Kac model at critical temperature, that the fluctuation field of the magnetization, under the Glauber dynamic, converges in distribution to the solution of a non linear ill-posed SPDE: the dynamical $\Phi^4_2$ equation.\\
In this article we consider the case of the multivatiate stochastic quantization equation $\Phi^{2n}_2$ on the two-dimensional torus, and we answer to a conjecture of H.Weber and H.Shen. We show that it is possible to find a state space for a spin system on the two-dimensional discrete torus undergoing Glauber dynamic with ferromagnetic Kac potential, such that the fluctuation field converges in distribution to $\Phi^{2n}_2$.
\end{abstract}



\section{Introduction}
\label{sec:intro}
During the last few years there has been a huge development in the theory of SPDE, especially for what concerns the construction of solution to ill-posed SPDE introduced in the physical literature in the last decades. The main source for the difficulties was the presence in the SPDE of both a nonlinearity and a rough noise term that forces the solution to live in a space of distributions.\\
In the study of the stochastic quantization equation $\Phi^{2n}_2$, see for instance \cite{JonaMitterquantization}, a breakthrough was represented by the work \cite{dPD}, where for the first time a pathwise solution theory has been proven. This approach together with ideas from the theory of rough paths, ultimately led to the creation of a theory of regularity structures by Hairer in \cite{Regularity}, that provides a general and abstract framework for the renormalization of the equations and the definition of a pathwise solution for the so-called subcritical SPDE.
One natural question is whether or not the same strategy can be applied to discrete models arising form the statistical mechanic and stochastic lattice gasses literature. 
This question is motivated by the fact that, very recently, some of the aforementioned works have been already extended to the discrete setting \cite{hairer2015discretisations,erhard2017discretisation}.

One of the first result in this direction, has been obtained in \cite{MR1317994} in the case of a one-dimensional Ising-Kac model at criticality (we recall that the Kac-Ising model has a phase transition even in one dimension, see \cite{1966JMP, presutti2008scaling}). In this case the solution of the process is a function not a distribution, hence there is no need for a renormalization of the nonlinearity. In the article, a coupling argument with a well studied discrete process, the voter model, is employed to show the convergence in distribution of the fluctuation field to the solution of $\Phi^4_1$.\\
With a similar spirit, the authors of \cite{MourratWeber} were able to prove that the fluctuation field for the Kac-Ising model at critical temperature converges to the solution of the dynamical stochastic $\Phi^4_2$. In order to do so, not only the microscopic model has to be rescaled in a diffusive way, but also the critical inverse temperature of the Kac-Ising model had to be tuned in a precise way, as a function of the lattice size. This discrepancy in dimension two was already known \cite{cassandro1995corrections,bovier1997low}, and it played a crucial role in the renormalization of the nonlinear terms arising in the dynamic.\\
In a subsequent article, Shen and Weber \cite{ShenWeber} proved that a similar dynamical lattice model converges to the solution of the dynamical $\Phi^6_2$ equation. The model they considered is the Kac-Blume-Capel model (or \virg{site diluited} Ising model) around its critical temperature. Also in this case the parameters of the model (inverse temperature and chemical potential) have to converge, as a function of the Kac parameter, in a precise way to their respective critical values.
In this paper we answer a question that has been posed in \cite{ShenWeber}, that is the existence of models that rescale to $\Phi^{2n}_2$ for any positive $n > 1$.
\begin{equation}
\label{e:Phi2n2bad}
\partial_t X = \Delta X + a_1 X + a_2 X^{2} +  \dots + a_{2n-1} X^{2n-1} +  \xi 
\end{equation}
In order to do so one has the feeling that it would be necessary to provide the model with enough parameters, in addition to the \virg{scaling} parameter, each of them converging to a critical value. It turn out that it is possible to do so indirectly charging an \virg{a priori} reference measure $\nug$ with all the \virg{model} parameters, where $\gamma$ is the main parameter and it is orchestrating the convergence of the model.\\
One of the reason for the introduction of so many parameters is that all the monomials in \eqref{e:Phi2n2bad} need to be renormalized in dimension two. In \cite{dPD} the authors showed the existence and uniqueness of strong solutions to \eqref{e:Phi2n2notttobad}
\begin{equation}
\label{e:Phi2n2notttobad}
\partial_t X = \Delta X + a_1 X + a_2 :X^{2}: + \dots + a_{2n-1} :X^{2n-1}: +  \xi 
\end{equation}
where $a_{2n-1}<0$ for any initial condition in a Besov space of negative regularity. In the above equation $:p(X):$ denotes the Wick renormalization of the polynomial.
The $a_{2n-1}<0$ is needed to guarantee the existence of the solution to \eqref{e:Phi2n2notttobad} for all times.\\
Recently in \cite{tsatsoulis2016spectral} the authors showed that the process converges exponentially to its stationary distribution.\\

Consider an odd polynomial $\aa{1} + \aa{2} x  + \dots + \aa{2n-1}x^{2n-1}$ with negative leading coefficient and let $m \geq 1$ be a positive integer. Extending the result of \cite{MourratWeber}, in the present article we describe how to produce a spin systems on a periodic lattice together with a Gibbs measure and a (spin flip) dynamic on it, such that its fluctuation field is converging, to the solution to the following SPDE
\begin{equation}
\label{e:Phi2n2}
\partial_t X = \Delta X + \aa{1} X + \aa{3} :X |X|^{2}: + \dots + \aa{2n-1} :X |X|^{2n -2}: +  \xi 
\end{equation} 
where $X = (X^{(1)},\dots,X^{(m)})$ is a vector-valued distribution from the $2$-dimensional torus and $\xi = (\xi^{(1)},\dots,\xi^{(m)})$ is a space-time multivalued two dimensional white noise.\\
As an application of our main result, we deduce in Corollary~\ref{cor:mvectormodel} that the fluctuation field of the Glauber dynamic on the $m$-vector model converges in distribution to
\begin{equation}
\label{e:mvectorlimit}
\partial_t X = \Delta X - \frac{m}{m+2}: |X|^{2}X: + \frac{1}{\sqrt{m}}\xi
\end{equation}
In Section~\ref{sec:model-theorems} we will introduce the model and the reference measure on the state space of the spins and we define the dynamic. In Proposition~\ref{prop:existencereferencemeasure} we prove that, for $\aa{2n-1}$ in a given interval, there exists a discrete model converging to \eqref{e:Phi2n2}. In Section~\ref{subsec:limitingspde} we recall the solution theory of \eqref{e:Phi2n2} and we introduce some ingredients for the following sections. In Sections \ref{sec:linear} and \ref{sec:tightness-linear} the linear part of the process is shown to converge to the solution of the stochastic heat equation. The remaining Section~\ref{sec:nonlinear}, completes the analysis with the study of the nonlinear part of the dynamic.
\section{Models and main theorem}
\label{sec:model-theorems}
Let $N$ be a positive integer and define $\LN = [-N,N]^2 \cap \Z^2$ to be a two dimensional torus. Let $m\geq 1$ be a positive integer and let $S = \R^m$ be the state space for the spins. We will consider a reference measure $\nug$ on $S$, having the following characteristics:
\begin{enumerate}
\item $\nug$ is isotropic.
\item For all $\theta > 0$, $\nug$ has exponential moment of order $\theta$, i.e. 
\[\int \exp(\theta |\eta| ) \nug(d\eta) < \infty
\]
\item $\int_S |\eta|^2 \nug(d\eta) = m$.
\end{enumerate}
This definition takes into account the possibility to have unbounded spins (for instance Gaussian) and contains the framework of the previous works \cite{MourratWeber,ShenWeber}.\\
In addition to the above requirements, $\nug$ will have to satisfy constraints related to the form of the limiting polynomial. In order to understand the form of the further assumptions, it is necessary to introduce the model and the dynamic.
In the following pages we are going to define the dynamic and in Subsection~\ref{subsec:assumptionsnug} we will complete the list of assumptions on $\nug$.\\

Denote by $\Sigma_N = S^{\LN}$  the space of all configuration. Given a configuration $\sigma \in \Sigma_N$, for $\Lambda' \subseteq \Lambda_N$ we denote with $\sigma_{\Lambda'}$ the configuration $\sigma$ restricted on $\Lambda'$. For the singletons $\sigma_{\{x\}} = \sigma_x$.\\ On this set, we define a product measure $\nuN \eqdef \prod_{i \in \LN} \nug^{(i)}$, where each $\nug^{(i)}$ is a copy of $\nug$ at the position $i$ in the lattice. The Gibbs measure will be defined by prescribing its density with respect to $\nuN$ which we call the reference product measure.\\
\begin{remark}
It seems strange to allow the measure $\nug$ to depend on $\gamma$. The reason is that in order to obtain a generic polynomial as in \eqref{e:Phi2n2}, we need the moments of the a priory measure $\nug$ to satisfy some relations as $\gamma$ tends to $0$. The rate of convergence to this relations will be responsible for the form of the polynomial. In \cite{ShenWeber}, choosing the parameters $(\beta,\theta) = (\beta(\gamma),\theta(\gamma))$ close to a critical curve, it is shown that the Glauber dynamic converges to the solution of the dynamical stochastic quantization equation $\Phi^4_2$, while for $(\beta(\gamma),\theta(\gamma))$ close to a critical point, one obtains the convergence to $\Phi^6_2$. The reason is basically that some algebraic relations among the parameters have to be satisfied in order to annihilate more coefficients. This is not the only constraint that the parameters have to satisfy: since the solutions of the limiting equation are distribution valued processes, the divergences created by the powers of the variables have to be compensated. This means that also the speed and the direction at which the parameters approach the critical hypersurfaces needs to compensate such divergences. 
We remark furthermore that the parameter $\beta(\gamma)$ itself could have been absorbed into the measure $\nug$, but condition 3 rules out this possibility allowing a clear definition for the model and for the inverse temperature as well. It is now clear that the choice of the constant on the right-hand-side of condition 3 is arbitrary and it is just made out of convenience.\\
In order to keep the notation light, when there is no possibility of confusion, we will drop from $\beta(\gamma)$ the dependence on $\gamma$.
\end{remark}

We will now going to define our Gibbs measure.\\
Let $\KK$ be a $\Cc^2(\R^2;[0,1])$ function with support contained in $B(0,3)$, the ball of radius $3$, satisfying
\begin{equation}
\label{e:norm-kk}
\int_{\R^2} \KK(x) \, dx = 1, \qquad \int_{\R^2} \KK(x) \, |x|^2 \, dx = 4 \;. 
\end{equation}
Define the interaction kernel $\kg : \LN \to [0,\infty)$ as  $\kg(0)=0$ and
\begin{equation}\label{e:samplek}
\kg(k) = \frac{ \gamma^2 \, \KK(\ga k)  }{ \sum_{k \in \LN \setminus \{0\} }\ga^2 \, \KK(\ga k)}
 \qquad k \neq 0\;.
\end{equation}
We are interested in a family of Ising-Kac-like models of the following form: the Hamiltonian is defined as
\begin{equ}\label{e:Hamiltonian}
\Hg(\si) \eqdef -  \frac12 \sum_{k, j \in \LN}  \kg(k-j) \ang{\si_j ,  \si_k }
\end{equ}
where $\ang{\cdot,\cdot}$ is the scalar product in $\R^m$. The Gibbs  measure $\lbg^N$ on $\Sigma_N$ is defined, using the reference measure introduced above, as
\begin{equation}
\label{e:Gibbsmeasure}
 \lbg^N (\sigma) \eqdef \frac{1}{\mathcal{Z}_N}\exp{\left( - \beta \Hg(\sigma) \right)}  \nuN(\sigma)
\end{equation}
where $\mathcal{Z}_N$ is the normalization constant. We define, for $x \in \LN$
\begin{equation}
\label{e:hgdef}
\hg(x,\sigma) \eqdef \sum_{i \in \LN} \kg(x-i) \sigma_i 
\end{equation}
And we will soon abuse the notation writing $\hg(x,t)$ in the place of $\hg(x,\sigma_t)$, or $\hg(x)$ for $\hg(x,\sigma)$ when there is no time involved.\\
We will now define a dynamic over the state space $\Sigma_N$ as follows: each site of the discrete lattice is given an independent Poisson clock with rate $1$. When the clock rings at site $x \in \LN$, the spin at $x$ which is in the state $\sigma_x$, changes its value by picking one randomly distributed according to a distribution which makes the Gibbs measure \eqref{e:Gibbsmeasure} reversible for the Markov process described. \\

When a jump occurs at $x \in \LN$, the configuration changes
\[
\sigma_{\LN \setminus \{x\}} \sqcup \sigma_x \mapsto \sigma_{\LN \setminus \{x\}} \sqcup \eta_x
\]
and the new value $\eta_x$ is chosen according to  $\pg{x}$, a probability distribution on $\R^m$ depending on the energy difference between the configuration before and after the jump
\[
\beta \Delta^{x}H(\sigma) \eqdef \beta H(\sigma_{\LN	\setminus \{x\}} \sqcup \eta_{x}) -\beta  H(\sigma) =  \ang{\beta \hg(x),(\eta_x- \sigma_x)}\;.
\]
We then define $p^{\lambda}$ prescribing its density with respect to $\nug$
\begin{equation}
\label{e:pdef}
\frac{dp^{\lambda}}{d\nug}(d\eta_x) \sim  \exp(\beta\ang{\lambda,\eta_x})
\end{equation}
$\pg{x}(d\eta_x)$ will be the distribution of the spin at $x$ after a jump at $x$ has occurred.\\
\begin{remark}
The distribution $\pg{x}$ and the normalizing constant $Z_x$ only depend on $\sigma$ via ${\hg}(x) = \sum_{i \neq x} \kg(x-i)\sigma_i$, since there are some cancellations between the terms containing $\sigma_x$. In particular $\pg{x}$ is a \virg{tilted} version of the reference measure $\nug$.
\end{remark}

\begin{remark}
\label{remark:workingunderstoppingtime}
It turns out that it will be more convenient to work with a modified version of $\pg{x}$, that will be introduced in Subsection~\ref{subsec:stopping time}. The process considered this way, will coincide with the Glauber dynamic defined here \virg{up to a stopping time}.
\end{remark}

For any $\lambda \in \R^m$ we define $\Phi : \R^m \to \R^m$ by
\begin{equation}
\label{e:Phidef}
\Phi(\lambda) \eqdef \int_{S}  \eta_x  p^{\lambda}(d \eta_x)\;.
\end{equation}
The generator of the dynamic on local functions $f$ is given by
\[
\mathcal{L} f = \sum_{x \in \supp(f)}  \int_{S} \pg{x}(d\eta_x)\left( f(\eta_x \sqcup \sigma_{\{x\}^c}) - f(\sigma) \right)\;,
\]
where we denoted by $\eta_x \sqcup \sigma_{\{x\}^c}$ the concatenation of two configurations, and we set $\sigma_{\{x\}^c} = \{\sigma_j\}_{j \in \Lambda_{\epsilon} \setminus x}$.  If $f(\sigma) = {\hg}(z,\sigma)$ then
\begin{equs}
\label{e:generator}
\mathcal{L} {\hg}(z) &= \sum_{j \neq z} \kg(j,z) \mathcal{L} \sigma_j = \sum_{j \neq z} \kg(j,z)  \int_{S} \eta_j \pg{j}(d\eta_j) - {\hg}(z)  \\
&= \kg \ast \Phi({\hg}(\cdot)) (z)  - {\hg}(z),
\end{equs}
where we used the fact that $\kg(0)=0$.\\

Our study will develop around the evolution in time of the local mean field $\hg(x,t)$.  Here $(x,t)$ are \virg{microscopic coordinates} and we have
\begin{equation}
\label{e:microdynamic}
\hg(x,t) =\int_0^t \mathcal{L} \hg(x,s) ds + \mg(x,t),
\end{equation}
where $\mg$ is an $\R^m$-valued martingale with predictable quadratic variation given by the matrix, for $1 \leq i,j \leq m$
\begin{multline}
\label{e:microquadvar}
\ang{\mg^{(i)}(x,\cdot),\mg^{(j)}(z,\cdot)}_t \\
= \int_0^t \sum_{l \in \LN} \kg(l-x)\kg(l-z) \int_{S} (\eta^{(i)} - \sigma_{l}^{(i)}(s^-))(\eta^{(j)} - \sigma_l^{(j)}(s^-)) \pg{x, s^- }(d \eta)
\end{multline}
\begin{proposition}
 The Gibbs measure $\lbg^N$ is reversible with respect to the above dynamic.
\end{proposition} 
\begin{proof}
In order to do this we are using the fact that ${\hg}(x)$ remains unchanged by the jump at $x$, by the definition \eqref{e:hgdef} and \eqref{e:samplek}, hence the measure $\pg{x}(d\eta_x)$ as well. We have
\begin{gather*}
\int f(\sigma) \mathcal{L}g(\sigma) \lbg^N(d\sigma) \\
= \int f \left( \sum_{x \in \Lambda_N}  \int_{\R^m} \pg{x}(d\eta_x) g(\eta_x \sqcup \sigma_{\{x\}^c}) - g(\sigma)\right) \lbg^N(\sigma) \\
= \sum_{x \in \Lambda_N} \int f(\sigma)g(\eta_x \sqcup \sigma_{ x^c}) \frac{d\lbg^N}{d\nuN}(\sigma) \nuN(d \sigma) \frac{d \pg{x}}{d\nug}(\eta_x) \nug(d \eta_x) - \int f(\sigma) g(\sigma) \lbg^N(d\sigma)
\end{gather*}
where we used the fact that $\frac{dp}{d\nug}$ is defined in (\ref{e:pdef}). Since $\pg{x}$ doesn't contain the variable $\sigma_x$ and 
\[
\frac{d\lbg^N}{\nuN}(\sigma) \frac{d \pg{x}}{d\nug}(\eta_x) = \frac{d\lbg^N}{\nuN}(\eta_x \sqcup \sigma_{\{x\}^c}) \frac{d \pg{x}}{d\nug}(\sigma_x) \qquad \nuN \otimes \nug-a.s.\;,
\]
we conclude that $\lbg$ is reversible with respect to the generator $\mathcal{L}$
\[
\int f(\sigma) \mathcal{L}g(\sigma) \lbg^N(d\sigma) = \int g(\sigma) \mathcal{L}f(\sigma) \lbg^N(d\sigma).
\]
\end{proof}

\begin{remark}
\label{Remark:mourratwebershen}
We would like to point out that this framework covers the cases in \cite{MourratWeber} and \cite{ShenWeber}, having respectively
\begin{align*}
\text{(Kac-Ising)}\qquad\nug^{I} &= \frac12 \delta_{-1} + \frac12 \delta_{1}\\
\text{(Kac-Blume-Capel)}\qquad\nug^{BC} &= \frac{e^{\theta}}{1 + 2e^{\theta}} \delta_{-1} +\frac{1}{1 + 2e^{\theta}}\delta_0 + \frac{e^{\theta}}{1 + 2e^{\theta}} \delta_{1}
\end{align*}
where $\delta_x$ is the Dirac measure at $x$. Here $m=1$ (recall that $m$ is the dimension of the state space). Moreover the rate function for the Kac-Blume-Capel model (formula 2.5 in \cite{ShenWeber}), can be written using \eqref{e:pdef}
\[
\pg{x}(d \eta_x) = \frac{\exp\left( \beta \hg(x) \eta_x\right) \nug^{BC}(d \eta_x)}{\int_{\R}\exp\{\beta \hg(x) \zeta\}\nug^{BC}(d\zeta)} = \begin{cases}
\frac{e^{-\beta \hg(x) + \theta}}{1 + e^{-\beta \hg(x) + \theta} + e^{\beta \hg(x) + \theta}} & \text{ if } \eta_x = -1\\
\frac{1}{1 + e^{-\beta \hg(x) + \theta} + e^{\beta \hg(x) + \theta}} & \text{ if } \eta_x = 0\\
\frac{e^{\beta \hg(x) + \theta}}{1 + e^{-\beta \hg(x) + \theta} + e^{\beta \hg(x) + \theta}} & \text{ if } \eta_x = 1
\end{cases}
\]
We observe that
\begin{equation}
\label{e:shenwebervariance}
\int_{\R} x^2 \nug^{BC}(dx) = \frac{2e^{\theta}}{1 + 2e^{\theta}} \neq m\;,
\end{equation}
hence $\nug^{BC}$ doesn't satisfy the third assumption for a reference measure (this is not a problem as it is always possible to rescale the value of spins). This minor difference is responsible for the fact that the critical temperature of the Kac-Blume-Capel model considered in \cite{ShenWeber} is different from the critical temperatures in our models. In a similar way, if $m>1$ the $m$-vector model is a generalization of the Ising model defined for
\[
\text{(m-vector model)}\qquad\nug(\eta) = \omega_m^{-1}\delta_1(|\eta|)
\]
where $\omega_m$ is the surface of the $m-1$-dimensional sphere. Also in this case $\int_S |\eta|^2 \nug(d\eta) = 1\neq m$.
\end{remark}

\subsection{Rescaling}
\label{subsec:rescaling}
We are interested in the fluctuation of the local magnetization $\hg$ when the parameter $\beta$ is close to its mean field critical value $\beta_c$.
In order for the fluctuations to survive the rescaling process, we must ensure that the quantity $\mathcal{L}\hg$ doesn't dominate the fluctuations. First recall
\begin{align*}
\mathcal{L}\hg(z) = \kg \ast \Phi({\hg}(\cdot)) (z)  - {\hg}(z)  = \kg \ast \left( \Phi({\hg}(\cdot)) (z)  - {\hg}(z) \right) + (\kg - \delta_0 )\ast {\hg}(z) 
\end{align*}
then rewrite \eqref{e:Phidef} as
\begin{equation}
\label{e:momgenfun1}
\Phi({\hg}(\cdot)) (z) - \hg(z)= \nabla_{\lambda}|_{\lambda = \beta \hg(z)} \left(\log   \int_S e^{\ang{\lambda,\eta}} \nug(d\eta) - \frac{|\lambda|^2}{2 \beta} \right)
\end{equation}
the critical $\beta_c$ is chosen in order to make $\hg(z)$ the Taylor expansion at fist order of $\Phi({\hg}(\cdot)) (z)$ and it is given in this case by $\beta_c \eqdef 1$.\\
The value of $\beta = \beta(\gamma)$ will be taken suitably close to $1$, and its precise value will be given in \eqref{e:betagamma}. This discrepancy will be used to compensate the divergences in the renormalization of the of the field $\Xg$.\\
In order to see a non trivial polynomial in the limit, heuristically, one has to look at \eqref{e:momgenfun1} and write it as
\begin{equation}
\Phi({\hg}(\cdot)) (z) - \hg(z) = \nabla_{\lambda}|_{\lambda = \beta \hg(z)}  \log   \int_S e^{\ang{\lambda,\eta} - \frac{|\lambda|^2}{2 \beta}} \nug(d\eta)\;.
\end{equation}

At this point one recognizes the fact that, if $\nug$ were a $m$-dimensional Gaussian with mean zero and covariance matrix $\beta^{-1} I$, the quantity would vanish. In order to make some of the powers of $\lambda$ vanish, we would need $\nug$ to have some moments in common with a Gaussian measure, when $\gamma \to 0$. The more moments of $\nug$ are Gaussian, the higher the degree of the polynomial.\\
In our case the critical hypersurfaces are referred to the moments of $\nug$. This will be further explained in Subsection~\ref{subsec:beta-nug}. 

\[
\Phi({\hg}(\cdot)) (z) - \hg(z) = \frac{\int_{S} \eta e^{\beta\ang{\eta, \hg(z)}} \nug(d\eta)}{\int_{S} e^{\beta\ang{\eta,\hg(z)}} \nug(d\eta)} - \hg(z)\;.
\]
Since $\nu$ is rotation invariant, $\Phi({\hg}(\cdot)) (z)$ is a scalar multiple of $\hg(z)$ hence
\begin{multline}
\Phi({\hg}(\cdot)) (z) - \hg(z)= \frac{\hg(z)}{|\hg(z)|^2} \left( \frac{\sum_{j \text{ odd}} \frac{1}{j!} \beta^j \int |\ang{\hg(z),\eta}|^{j+1} \nug(d\eta) }{\sum_{j \text{ even}} \frac{1}{j!} \beta^j \int |\ang{\hg(z),\eta}|^{j} \nug(d\eta)  } \right) - \hg(z) \\
= \hg(z) \left( a_1^{\gamma} + a_3^{\gamma} |\hg(z)|^2 + \dots + a_{2n-1}^{\gamma} |\hg(z)|^{2n-2} + \Oo(|\hg(z)|^{2n}) \right)\;.
\label{e:TaylorPhimicro}
\end{multline}
The above coefficients only depend on the even moments of the measure $\nug$ and on $\beta(\gamma)$, which is however a bounded, fixed function of $\gamma$. We will discuss about the existence of a suitable measure $\nug$ and its properties in Subsection~\ref{subsec:beta-nug}.\\
\begin{remark}
\label{remark:betac}
For the calculation above, the first coefficient $a_1^{\gamma}$ is
\[
a_1^{\gamma} = \frac{\beta(\gamma)}{|\hg(z)|^2}  \int_S \left(\sum_{j=1}^m  \hg(z)^{(j)} \eta^{(j)}\right)^2 \nug(d\eta) - 1 = \frac{\beta(\gamma)}{m}\int_S |\eta|^2 \nug(d\eta) - 1  = \beta(\gamma) - 1\;.
\]
This shows the motivation behind the assumptions on the reference measure $\nug$ and the choice for $\int_S |\eta|^2 \nug(d\eta)$. This exact calculation, together with \eqref{e:shenwebervariance}, provides the form of the critical line in \cite[Fig.~1]{ShenWeber}.
\end{remark}

We now introduce the parameters that will provide our scaling: as in \cite{MourratWeber} we will work in a suitable discretization of the $2$-dimensional torus $\T^2 = [-1,1]^2$. Let $\epsilon = \frac{2}{2N+1}$ and consider the discretization $\Le = ([-1,1] \cap \epsilon \Z)^2$. Define the rescaled field

\begin{equation}
\label{e:Xdef}
\Xg(x,t) \eqdef \delta^{-1} \hg(x/\epsilon,t/\alpha)
\end{equation}
where $(x,t) \in \Le \times [0,\infty)$ are macroscopic coordinates and $\alpha$ and $\delta$ are the scaling parameter of the time and the height of the field respectively. The relations between $(\epsilon,\alpha,\delta)$ have to be chosen in a specific way that we will describe below.\\

In macroscopic coordinates the effect of the Glauber dynamic on $\Xg$ is given for $x \in \Le$, $t \in [0,T]$ by the multidimensional SPDE
\begin{equation}
\label{e:Xequation1}
\begin{split}
d \Xg(x,t) &= \frac{\epsilon^2}{\gamma^2 \alpha} \Delta_{\gamma} \Xg(x,t)dt \\
&+ \Kg \aste \Xg \left( \frac{1}{\alpha} a_1^{\gamma} + \frac{\delta^2}{\alpha} a_3^{\gamma} |\Xg|^2 + \dots + \frac{\delta^{2n-2}}{\alpha} a_{2n-1}^{\gamma} |\Xg|^{2n-2} \right)(x,t)dt\\
&+ \Kg \aste \Eg(x,t) dt + d\Mg(x,t)\;.
\end{split}
\end{equation}

Where $\Kg(x) \eqdef \epsilon^{-2} \kg(\epsilon x)$ is approximating a Dirac distribution, the convolution is defined $F \aste G (x)\eqdef \sum_{y \in \Le} \epsilon^2 F(x-y)G(y)$ and $\Delta_{\gamma} \Xg = \frac{\gamma^2}{\epsilon^2}\left( \Kg \aste  \Xg -\Xg \right)$ is approximating a continuous Laplacian.\\
The form of the error term $\Eg$ can be deduced from \eqref{e:TaylorPhimicro}, and the value of the coefficients of the polynomial as well.\\

From \eqref{e:microquadvar} the martingale $\Mg(x,t)=\delta^{-1} \mg(\epsilon^{-1}x,\alpha^{-1}t)$ has predictable quadratic variation given by the matrix
\begin{equation}
\label{e:quadvarMg}
\begin{split}
&\ang{{\Mg}^{(i)} (x,\cdot) , {\Mg}^{(j)}(y,\cdot)}_t \\
&=\frac{\epsilon^2}{\delta^2 \alpha}\int_0^t \sum_{z \in \Le} \epsilon^2 \Kg(x-z)\Kg(y-z) \times \\
&\times \int_{S} (\eta^{(i)} - \sigma^{(i)}_{\alpha^{-1} s^-}(\epsilon^{-1}z))(\eta^{(j)} - \sigma^{(j)}_{\alpha^{-1} s^-}(\epsilon^{-1}z)) \pg{\epsilon^{-1}z, \alpha^{-1}s^- }(d \eta) ds
\end{split}
\end{equation}
where the superscript $(i)$ indicates the $i$-th component of a vector in $\R^m$. We then set
\begin{equation}
\label{e:ppowersdef}
Q^{i,j}(s,z) \eqdef \int_{S} (\eta^{(i)} - \sigma^{(i)}_{\alpha^{-1} s}(\epsilon^{-1}z))(\eta^{(j)} - \sigma^{(j)}_{\alpha^{-1} s}(\epsilon^{-1}z)) \pg{\epsilon^{-1}z, \alpha^{-1}s}(d \eta)\;.
\end{equation}
The conditions we have to impose to have in the limit, at least heuristically, the noise, the Laplacian and the $2n-1$ power of the field is to have
\[
\frac{\epsilon^2}{\delta^2 \alpha} \sim \frac{\epsilon^2}{\gamma^2 \alpha} \sim \frac{\delta^{2n-2}}{\alpha} \sim 1\;,
\]
which yields the scaling
\begin{equation}
\label{e:scaling}
\epsilon \simeq \gamma^{n}\;,\qquad \alpha = \gamma^{2n-2}\;, \qquad \delta = \gamma\;.
\end{equation}
It readily follows that the coefficients $a_1,a_3,\dots,a_{2n-3}$ have to vanish in $\gamma$ with a certain order. On top of that, the powers of the field $\Xg$ need to be substituted by their Wick powers. From \eqref{e:scaling} and \eqref{e:Xequation1}, define the coefficients $\aatg{1},\dots,\aatg{2n-1}$ and the polynomial $\pptg$ as
\begin{equation}
\label{e:coeffdef1}
a_1 = \delta^{2n-2} \aatg{1}\;, \qquad a_3 = \delta^{2n-4} \aatg{3}  \qquad \cdots \qquad a_{2n-1} =\aatg{2n-1} \;.
\end{equation}

\begin{equation}
\label{e:polydef}
\pptg(\Xg(z,t))= \Xg \left( \aatg{1} + \aatg{3} |\Xg|^2 + \dots + \aatg{2n-1} |\Xg|^{2n-2} \right)(z,t)
\end{equation}
where we recall that $\Xg$ is a vector-valued field $\Xg : \Le \mapsto \R^m$. In particular $\pptg(\Xg(\cdot,t)) : \Le \mapsto \R^m$ and we shall refer to $\pptg^{(j)}(\Xg)$ when we are considering its $j$-th component.\\

With this scaling the error in \eqref{e:Xequation1} can be bounded by
\begin{equation}
\label{e:errorterm}
|\Eg(x,t)| \leq C \gamma^2|\Xg(x,t)|^{2n+1} \int_S e^{\gamma \beta |\Xg(x,t)| |\eta|} \nug(d\eta)\;.
\end{equation}

Recall that in \eqref{e:Xequation1} all powers in the brackets are even powers and they can be rewritten as symmetric functions of $\Xg^{(1)}, \dots,\Xg^{(m)}$ the components of the vector $\Xg$. It is then clear that the renormalized polynomial would be a symmetric function of $\Xg^{(1)2}, \dots,\Xg^{(m)2}$, it is not clear a priori whether it is possible to renormalize \eqref{e:polydef} just making use of the dependency on $\gamma$ of the coefficients $\aatg{i}$ for $i = 1,\dots,2n-1$. We will describe how to do so in a systematic way in Remark \ref{rmk:Hermitepoly}, defining $\aag{2k-1}$ in \eqref{e:coeffdef2}, and we will describe it more carefully in Section~\ref{sec:nonlinear}, where we will perform such renormalization in detail.\\
The polynomial arising after the renormalization can be written as
\begin{equation}
\label{e:polyrinormilizeddef}
\ppg^{(j)}(\Xg(z,t))  = \aag{1}\Xg^{(j)} (z,t) + \aag{3} :\Xg^{(j)}|\Xg|^{ 2 }:(z,t) + \dots + \aag{2n-1} :\Xg^{(j)}|\Xg|^{2n-2}:(z,t)
\end{equation}
and $\pp$ the polynomial that is expected to be found in the limit
\begin{equation}
\label{e:polylimitdef}
\pp^{(j)}(X(z,t))  = \aa{1}X^{(j)} (z,t) + \aa{3} :X^{(j)}|X|^{ 2 }:(z,t) + \dots + \aa{2n-1} :X^{(j)}|X|^{2n-2}:(z,t)\;.
\end{equation}

In the above equations $:\Xg^{(j)}|\Xg|^{ 2j }:$ and $:X^{(j)}|X|^{ 2j }:$ have to be understood as the renormalized counterparts of $\Xg^{(j)}|\Xg|^{ 2j }$ and $X^{(j)}|X|^{ 2j }$ respectively. In order to go from \eqref{e:polyrinormilizeddef} to \eqref{e:polylimitdef} we need an assumption on the measure $\nug$, where the coefficients are ultimately coming from. All the assumptions on $\nug$ are collected in Section~\ref{subsec:assumptionsnug}.

\subsection{Assumptions over $\nug$ and the initial condition}
\label{subsec:assumptionsnug}
We now state together the assumption on the reference measure $\nug$ and on the initial distribution of the spins $\{\sigma_x(0)\}_{x \in \LN}$ in the same section for convenience.

\begin{assumption}
The measure $\nug$ on $S \subset \R^m$ is isotropic, has exponential moments of any order, i.e. for any $\theta \in \R^m$
\begin{equation}
\label{e:hpexponentialmoment}
\tag{M0}
\int_S e^{\ang{\theta,\eta}} \nug(d \eta) < \infty
\end{equation}
uniformly in $\gamma$ in a neighborhood of the origin, and moreover $\int_S |\eta|^2 \nug(d \eta) = m$.
\end{assumption}
Recall the coefficient used in \eqref{e:polyrinormilizeddef}, defined in \eqref{e:coeffdef2}. The following is a condition, given in a very implicit form, on the form of the moments of the measure $\nug$. 
\begin{assumption}
There exists $\aa{1},  \aa{3},\dots, \aa{2n-1} \in \R$,  $c_0 >0$ and $\lambda_0 > 0$ such that
\begin{equation}
\label{e:hpmeasure}
\tag{M1}
\sup_{k = 1, \dots, n} \norm{}{\aag{2k-1} - \aa{2k-1}} \leq c_0 \gamma^{\lambda_0}\;.
\end{equation}
\end{assumption}

The next assumption provides the control for large times

\begin{assumption}
The leading coefficient of the limiting polynomials $\pp^{(j)}$ (see \eqref{e:hpmeasure}) of degree $2n-1$ satisfy
\begin{equation}
\label{e:hppoly}\tag{M2}
\aa{2n-1} < 0\;.
\end{equation}
\end{assumption}

In order to prove the convergence of the linear and non linear dynamic, we now state the hypothesis on the initial distribution of the spins $\sigma(0)$. 
Two hypothesis are needed in order to prove the result: they are mainly needed to control the processes uniformly in $\gamma$. The first one concerns the regularity of the initial profile:
\begin{assumption} 
Let $\Xn \in \Cc^{-\nu}(\T^2)$ for any $\nu > 0$.
\begin{equation}
\label{e:hp1}\tag{I1}
\lim_{\gamma \to 0} \E \Norm{\Cc^{-\nu}}{ \delta^{-1}\hg(\cdot,0) - \Xn} = 0\;.
\end{equation}
\end{assumption}
This will be used to control the contribution of the initial condition to the $\Cc^{-\nu}$ norm of the process. The second assumption is used to get a uniform control over the quantity \eqref{e:ppowersdef}. It is a control over the starting measure.
\begin{assumption}
For all $p \geq 1$ there exists a $\gamma_0>0$ such that
\begin{equation}
\label{e:hp2}\tag{I2}
\sup_{\gamma < \gamma_0}\E \left[ \Norm{L^{\infty}}{\sigma(0)}^p \right] < \infty\;.
\end{equation}
\end{assumption}
In the work \cite{MourratWeber} the \eqref{e:hp2} assumption is not needed since the state space of the spins is a compact set.\\
Condition \eqref{e:hp2} can be relaxed to the assumption that the initial condition has all moments pointwise, i.e.
\begin{equation}
\label{e:hp2bis}\tag{I2'}
\sup_{\gamma < \gamma_0}\sup_{z \in \Le}\E \left[ |\sigma_{z}(0)|^p \right] < \infty\;.
\end{equation}
With assumption \eqref{e:hp2bis} and the monotonicity of $L^{p}$ norms one can prove \eqref{e:hp2} up to any small negative power of $\gamma$:
\begin{equation}
\label{e:hpbound}
\E \left[ \Norm{L^{\infty}}{\sigma(0)}^p \right]  \lesssim \gamma^{-\kappa}\;.
\end{equation}

The role of the above assumption is related to the control on the quadratic variation of the process

\subsection{Hermite Polynomials and renormalization}
\label{subsec:hermite}

The aim of this subsection is to clarify the way we renormalize the polynomial arising from the discrete model and to recall some general fact for later reference.\\ 
Recall that for a multiindex $\kb = (k_1,\cdots,k_m) \in \N^n$ and vector $\bar{X} = (X_1,\dots , X_m)$ and matrix $\bar{T} = (T_{i,j})_{i,j=1}^m$ the multivariate Hermite polynomial $H_{\kb}$ are defined as the coefficients of the Taylor expansion 
\begin{equation}
\label{e:Hermiteexponential}
\exp \left\lbrace \sum_{j=1}^m \lambda_j X_j - \frac{1}{2}\sum_{i,j =1}^m \lambda_i \lambda_j T_{i,j}\right\rbrace = \sum_{k_1,\dots,k_m \geq 0} \frac{\lambda^{\kb}}{\kb!} H_{\kb}(\bar{X},\bar{T})
\end{equation}
where we used the convention $\lambda^{\kb}=\lambda^{k_1} \cdots \lambda^{k_m} $ and $\kb! = k_1! \cdots k_m!$
From the above definition we obtain the identities 
\begin{align*}
\partial_{X_j} H_{(k_1,\dots,k_m)}(\bar{X},\bar{T}) &= k_j H_{(k_1,\dots,k_j -1\dots,k_m)}(\bar{X},\bar{T})\\
\partial_{T_{i,j}} H_{(k_1,\dots,k_m)}(\bar{X},\bar{T}) & = - \frac{1}{2} \begin{cases} k_i k_j H_{(k_1,\dots,k_i -1\dots,k_j -1\dots,k_m)}(\bar{X},\bar{T}) & \text{ for } i \neq j\\
(k_j - 1) k_j H_{(k_1,\dots,k_j -2\dots,k_m)}(\bar{X},\bar{T}) & \text{ for } i = j \end{cases}
\end{align*}
and for $\bar{V} \in \R^m$
\begin{equation}
\label{e:hermitesum}
H_{\kb}(\bar{X} + \bar{V},\bar{T}) = \sum_{\substack{\mathbf{a} \in \N^m\\ \mathbf{a} \leq \kb}} \bar{V}^{\mathbf{a}} H_{\kb - \mathbf{a}}(\bar{X},\bar{T})\;.
\end{equation}
Assume now that $\bar{T} = \mathfrak{c} I_m$ where $I_m$ is the identity matrix in $\R^{m \times m}$ and $\mathfrak{c} > 0$ we abuse the notation writing instead of $H_{\kb}(\bar{X},\mathfrak{c}I_m)$
\begin{equation}
\label{cap2multdimorthermite}
H(X_1^{k_1} \cdots X_m^{k_m},\mathfrak{c}I_m) = \prod_{j=1}^m H_{k_j}(X_j,\mathfrak{c})
\end{equation}
and we extend by linearity the above expression to any polynomial in $m$ variables setting
\begin{equation*}
H\left(\sum_{\mathbf{a} \in \N^m} b_{\mathbf{a}} X_1^{a_1} \cdots X_m^{a_m},\mathfrak{c}I_m\right) = \sum_{\mathbf{a} \in \N^m} b_{\mathbf{a}} H(X_1^{a_1} \cdots X_m^{a_m},\mathfrak{c}I_m)\;.
\end{equation*}

The notation above can be justified with the following expression for the $n$-th Hermite polynomial
\begin{equation}
\label{e:Hermiterenormalization}
H_{\kb}(\bar{X},\mathfrak{c} I_m) = e^{- \frac{\mathfrak{c}}{2}\Delta} \bar{X}^{\kb} = \left(1 - \frac{\mathfrak{c}}{2}\Delta + \frac{\mathfrak{c}^2}{8}\Delta^2 + \dots \right) \bar{X}^{\kb} \;.
\end{equation}

\begin{remark}
\label{rmk:Hermitepoly}
We will use in the following sections the fact that the renormalization for the polynomial $X^{(i)}|\bar{X}|^{2n}$ is given by a similar polynomial
\[
H(X^{(i)}|\bar{X}|^{2n},\mathfrak{c}I_m) = e^{-\frac{\mathfrak{c}}{2} \Delta} X^{(i)}|\bar{X}|^{2n} = \sum_{k = 1}^{n} b_k(\mathfrak{c}) X^{(i)}|\bar{X}|^{2k}\;,
\]
this can be seen using the fact that
\begin{equation}
\label{e:Deltatopoly}
\begin{split}
\Delta X^{(i)}|\bar{X}|^{2n} &= \left(\sum_{j=1}^m \partial_j^2 \right) X^{(i)}|\bar{X}|^{2n} \\
&= \partial_i|\bar{X}|^{2n}  + 2n \partial_i X^{(i)2}|\bar{X}|^{2n-2}+ \sum_{j \neq i} 2n \partial_j X^{(i)}X^{(j)}|\bar{X}|^{2n-2}  \\
&= 2n X^{(i)}|\bar{X}|^{2n-2} + 4n X^{(i)}|\bar{X}|^{2n-2} + 2n(2n-2) X^{(i)3} |\bar{X}|^{2n-4} \\
&+ \sum_{j \neq i} 2n X^{(i)}|\bar{X}|^{2n-2} + 2n (2n-2)X^{(i)}X^{(j)2}|\bar{X}|^{2n-4} \\
&=  2n (m+2n)  X^{(i)} |\bar{X}|^{2n-2} \;.
\end{split} 
\end{equation}
\end{remark}

The issue that we are going to discuss now concerns the renormalization procedure to highlight the Wick powers of the field $\Xg$.
\begin{align}
\nonumber
\pptg^{(j)}(X) &= \sum_{k=0}^{n-1} \aatg{2j+1} e^{\frac{\mathfrak{c}}{2}\Delta_X}e^{-\frac{\mathfrak{c}}{2}\Delta_X}  X^{(j)} |X|^{2k} = \sum_{k=0}^{n-1} \aatg{2j+1} e^{\frac{\mathfrak{c}}{2}\Delta_X}\HXXj{2k} \\
&= \sum_{k=0}^{n-1}\left(e^{\frac{\mathfrak{c}}{2}\Delta_X^*} \aatg{}\right)_{2j+1}\HXXj{2k} \;.
\label{e:polyrenorm1}
\end{align}
Hence we define the coefficients in \eqref{e:polyrinormilizeddef} as
\begin{equation}
\label{e:coeffdef2}
\aag{2k+1} \eqdef \left(e^{\frac{\CG}{2}\Delta_X^*} \aatg{}\right)_{2k+1}\;.
\end{equation}
It is immediate to see from its definition that the exponential of $\Delta_X^*$ is a well defined operation on the space $l^{\infty}_0$ of the sequences which are eventually zero.\\
The above calculation shows how is possible to write the polynomial $\ppg^{(j)}(X)$ in \eqref{e:polydef} as
\begin{multline}
\label{e:polyrenorm}
\pptg^{(j)}(X)  =\aag{1}X^{(j)} + \aag{3} H(X^{(j)}|X|^{2},\CG)+ \dots + \aag{2n-1} H(X^{(j)}|X|^{2m-2},\CG)\;.
\end{multline}

\begin{remark}
As a consequence of this definition and the assumption of the Subsection~\ref{subsec:assumptionsnug}, we can provide a formula for the value of $\beta(\gamma)$ and its discrepancy from its critical value $1$. From Remark~\ref{remark:betac}, the expression for $a_1^{\gamma}$ in \eqref{e:coeffdef1}, \eqref{e:scaling} and \eqref{e:coeffdef2} we have
\[
\beta(\gamma) = 1 + a_1^{\gamma} = 1 + \alpha \aatg{1} = 1 + \alpha \left(e^{-\frac{\CG}{2}\Delta_X^*} \aag{}\right)_{1}
\]
by the fact that $\CG$ is diverging as $\log(\gamma^{-1})$ and assumption \eqref{e:hpmeasure} we have
\begin{equation}
\label{e:betagamma}
\beta(\gamma) = 1 + \alpha \left(e^{-\frac{\CG}{2}\Delta_X^*} \aa{}\right)_{1} + \Oo(\alpha \gamma^{\lambda_0} \CG^{n})
\end{equation}
where $\aa{} = (\aa{1},\dots,\aa{2n-1},0,\dots)$ are the coefficients of the limiting polynomial. It is immediate to see that this value coincide with the choice of the critical temperature in \cite{MourratWeber} and in \cite{ShenWeber} (see Remark~\ref{Remark:mourratwebershen}).
\end{remark}

\subsection{Existence of a reference measure $\nug$}
\label{subsec:beta-nug}

We saw in Subsection~\ref{subsec:hermite} that the limiting polynomial comes from a combination of the moments (or cumulants) of the reference measure $\nug$. We will now start from a given renormalized polynomial and find a reference measure producing such a polynomial. It is clear that the limiting process depends only on finitely many moments of the measure $\nug$, and therefore it is clear that the solution, if exists, is not unique.\\
The rest of the article is independent from this subsection.\\
We will now deal with the existence of a reference measure producing a polynomial of form \eqref{e:polyrinormilizeddef}. The problem of find a measure with a given sets of moments is known in the literature as the \virg{moment problem} (see \cite{aheizer1965classical}). For the equation to make sense in the limit, the polynomial has to be renormalized as in \cite{dPD} with a renormalization constant $\CG$ diverging as $\gamma \to 0$. The precise value of $\CG$ will be given in \eqref{e:cgammadef} and it is not important at this stage. The only fact that we will use is that the divergence is logarithmic as $\gamma \to 0$, hence slower that any negative power of $\gamma$. Recall the expansion
\begin{multline*}
\frac{1}{\alpha \delta} \left( \Phi({\hg}(z,t))  - \hg(z,t) \right)\\
=\Xg \left( \frac{1}{\alpha} a_1^{\gamma} + \frac{\delta^2}{\alpha} a_3^{\gamma} |\Xg|^2 + \dots + \frac{\delta^{2n-2}}{\alpha} a_{2n-1}^{\gamma} |\Xg|^{2n-2} \right)(z,t) + \frac{\delta^{2n}}{\alpha} \Oo(|\Xg(z,t)|^{2n+1})\:.
\end{multline*}

The scaling \eqref{e:scaling} entails the fact that the coefficients $a_1^{\gamma},\dots,a_{2n-3}^{\gamma}$ are vanishing at a suitable rate. We know, however, from the form \eqref{e:momgenfun1} that they are identically zero as soon as $\nug$ shares the first $2n-2$ moments with a $m$ dimensional Gaussian random variable with a suitable covariance matrix. In fact \eqref{e:momgenfun1} tells that the coefficient depends on the difference between the cumulants of the measure $\nug$ minus the cumulants of a multivariate Gaussian random variable. \\

\begin{proposition}
\label{prop:existencereferencemeasure}
Let $C_{\gamma}$ any positive sequence of renormalization constants diverging logarithmically as $\gamma \to 0$.\\
Let $b_{\gamma} \to b \in \R^+$ a sequence of positive real numbers.\\
Let now $n \geq 2$ and $\aa{1},\dots,\aa{2n-1} \in \R$ with $\aa{j} = 0$ if $j$ is even.\\
For all $m \in \N$ and $\gamma < \gamma_0$ small enough for any $\aa{2n-1}<0$ small enough in absolute value, there exists a family of rotational invariant measures $\{\mu_{\gamma}\}_{\gamma < \gamma_0}$ over $\R^m$ such that 
\begin{itemize}
\item $\forall \gamma < \gamma_0$, $\mu_{\gamma}$ have all exponential moments
\item The sequence $\mu_{\gamma}$ is weakly convergent as $\gamma \to 0$
\item For $j = 1,\dots,m$ the polynomial obtained in \eqref{e:polyrenorm}, using the measure $\mu_{\gamma}$ and the inverse temperature $b_{\gamma}$ is
\[
\aa{1}X^{(j)} + \aa{3} H(X^{(j)}|X|^{2},C_{\gamma})+ \dots + \aa{2n-1} H(X^{(j)}|X|^{2m-2},C_{\gamma})\;.
\]
\end{itemize}
Moreover it is possible to choose the family $\mu_{\gamma}$ to be supported in the same compact set of $\R^m$.
\end{proposition}

The proof is straightforward, but we will provide it for convenience of the reader.

\begin{proof}
It is clear that it suffices to prove the theorem above in the case $b_{\gamma} = 1$ for all $\gamma$.\\
We first describe conditions over the moments (or the cumulants) of the measure.
From \eqref{e:coeffdef1}, \eqref{e:coeffdef2} and \eqref{e:momgenfun1} we have that $\mu_{\gamma}$ has to satisfy, for $\lambda \in \R^m$
\[
\log\left( \int_S e^{\ang{\lambda,\eta}} \mu_{\gamma}(d\eta)\right) = \frac{1}{2}|\lambda|^2 + \sum_{j=1}^n \frac{\alpha}{ \delta^{2j - 2} } \left(e^{-\frac{C_{\gamma}}{2}\Delta_X^*} \aa{}\right)_{2j-1} \frac{1}{2j} |\lambda |^{2j} + \Oo(|\lambda|^{2n+2})
\]
where $e^{-\frac{C_{\gamma}}{2}\Delta_X^*}$ is the inverse of the operator described in Subsection~\ref{subsec:hermite} and the sequence $\aa{}$ is extended to be zero after the $2n-1$-th place. For $\gamma$ small the above polynomial is a perturbation of
\[
\frac{1}{2}|\lambda|^2 + \frac{\aa{2n-1}}{2n} |\lambda |^{2n} + \Oo(|\lambda|^{2n+2})
\]
taking the exponential we see that the $\mu_{\gamma}$ can be seen as a small perturbation of a multivariate Gaussian random variable, up to the $2n-2$-th moment (recall that $\alpha = \delta^{2n-1}$ by \eqref{e:scaling}). Here we are using the fact that $C_{\gamma}$ has a logarithmic divergence in $\gamma$. For $\lambda$ with modulus $1$, the $2n$-th moment is given by
\[
\int_S |\ang{\lambda,\eta}|^{2n} \mu_{\gamma}(d \eta) = \frac{2n!}{n!2^n} + (2n-1)! \aa{2n-1} 
\]
since we aim to produce an isotropic measure, it is sufficient then to prove the existence of a univariate distribution having the first $2n$ even moments equal to
\begin{equation}
\label{e:momentcollection}
m_{2j}^{\gamma} = \begin{cases}
\frac{2j!}{j!2^j} + o(1) & \text{ if } j < n\\
\frac{2n!}{n!2^n}+(2n-1)! \aa{2n-1} + o(1) & \text{ if } j = n\\
\end{cases}
\end{equation}
for $\gamma \to 0$. Since we have the freedom to chose the higher moments, we shall do that later. Such a measure is known to exists (see \cite{aheizer1965classical}) if and only if the moment matrix is positive definite i.e. for all choice complex numbers $\{z_0,z_1,\dots,z_p\}$
\begin{equation}
\label{e:momentcondtion}
\sum_{i_1,i_2 = 0}^p  m_{i_1 + i_2}^{\gamma} \bar{z}_{i_1} z_{i_2} \geq 0	\;.
\end{equation}
It is easy to see the necessity of such condition, since \eqref{e:momentcondtion} is the expectation of the square of a polynomial in the random variable. Condition \eqref{e:momentcondtion} is satisfied for $ m_{j}^{\gamma}$ if it is satisfied with a strict inequality for $m_{j} \eqdef \lim_{\gamma \to 0} m_{j}^{\gamma}$. It is easy to see that for a standard Normal $U \sim \mathcal{N}(0,1)$ \eqref{e:momentcondtion} holds for a strict inequality if $\sum_{i=0}^p |z_i|^2 > 0$
\[
\E \left[\left| z_0 + z_1 U + z_2 U^2 + \dots + z_{p}U^p\right|^2 \right] > 0\;.
\]
It is immediate to conclude that there exists a negative value of $\aa{2n-1}$ and $\gamma$ small enough such that the above inequality is satisfied for the collection of moments given by \eqref{e:momentcollection}.\\
The values of $\aa{2n-1}$ that guarantee condition \eqref{e:momentcondtion} are given by the inequality
\begin{equation}
\label{e:momentsineq}
\aa{2n-1}  > - \frac{D_{2n}}{(2n-1)! D_{2n-1} }\;.
\end{equation}
If we denote with $D_p \eqdef \det\left( m^G_{i_1 + i_2} \right)_{i_1,i_2 =0}^p > 0$ the determinant of the moment matrix of the Gaussian random variable $U$.\\
In order to complete the proof it is sufficient to complete the list of the moments with arbitrary values satisfying \ref{e:momentcondtion}. It is always possible to find such a sequence of moments because each moment is asked to satisfy an inequality similar to \ref{e:momentsineq} which admits trivially a solution.\\
This implies the existence of a measure with the prescribed moments.\\
We might chose, in particular, at some $p > n$ to have the left hand side of \eqref{e:momentcondtion} equal to $0$ for all $\gamma$. This implies that the random variable annihilates a certain nonnegative polynomial and therefore it is supported on the set of the real zeros of such polynomial, which is a finite set, hence a bounded set. This proves the last claim.
\end{proof}

\subsection{Stopping time for the dynamic}
\label{subsec:stopping time}

As announced in Remark~\ref{remark:workingunderstoppingtime}, we are now ready to define a stopping time $\taun$ for the macroscopic dynamic defined above. Fix a positive $\nu > 0$ and a value $\nn > 1$, and let
\begin{equation}
\label{e:stoppingtime}
\taun = \inf \left\lbrace t \geq 0| \Norm{\Cc^{-\nu}}{\Xg(t,\cdot)} \geq \nn \right\rbrace\:.
\end{equation}
Following \cite{MourratWeber} we will not work directly with the process introduced in Subsection~\ref{subsec:rescaling}, but with a process whose jump distribution is given by (for macroscopic coordinates $(x,s) \in \Le \times [0,T]$)
\begin{equation}
\label{e:p_mdefin}
\pgm{x, s^- }(d\eta_x)  = \begin{cases} \pg{\epsilon^{-1} x, \alpha^{-1} s^- }(d\eta_x) & \text{ if } s \leq \alpha^{-1} \taun \\
\nug(d\eta_x) & \text{ if } s > \alpha^{-1}\taun
\end{cases}
\end{equation}
In particular, $\pgm{x, s^- }(d\eta_x)$ doesn't depend on the current configuration when $s > \alpha^{-1}\taun$ and for general $\nn \geq 0$, $k > 0$ and $s > 0$, from assumption \eqref{e:hpexponentialmoment}
\begin{equation}
\label{e:boundonmomentsofnug}
\int_S |\eta|^k \pgm{z,s^-}(d \eta) \leq C(k,\nn)\;.
\end{equation}
The process with jump distribution \eqref{e:p_mdefin} coincides with the process defined in Section \ref{sec:model-theorems} up to the stopping time, after which it still follows a Glauber dynamic, but with respect to the Gibbs measure at infinite temperature ($\beta = 0$).

\begin{remark}
\label{remark:stoopingtime}
The reason behind the introduction of the stopping time is to have an \virg{a priori} control of the norm of the fluctuation field because the quadratic variation of the linear part $\Zg$ depends on the whole fluctuation field $\Xg$. It can be proven, following the same proof as \cite [Theorem~2.1]{MourratWeber}, that for all arbitrarily small $\zeta > 0$, there exists $\nn > 1$ such that
\[
\limsup_{\gamma \to 0} \P[\taun \leq T] \leq \zeta
\]
that allows us \virg{a posteriori} to use the dynamic defined in \eqref{e:p_mdefin}.\\ 
A consequence of this fact is that in order to prove the convergence in distribution for $\Xg$ it is sufficient to show that, $\forall \nn > 1$, and all continuous bounded $F:\D([0,T],\Cc^{-\nu})$
\[
\limsup_{\gamma \to 0}\E\left[|F(\Xg) - F(X)|\mathbf{1}_{\{\taun > T\}} \right] = 0\;.
\]
This means that we can always assume to work with the stopped dynamic, and we will do so.\\

\end{remark}

For the new process, define $\Qm$ as in \eqref{e:ppowersdef}
\begin{multline}
\label{e:ppowersdefm}
\Qm^{i,j}(s^-,z) = \int_{S} (\eta^{(i)} - \sigma_{\epsilon^{-1} z}^{(i)}(\alpha^{-1} s^-))(\eta^{(j)} - \sigma_{\epsilon^{-1} z}^{(j)}(\alpha^{-1} s^-)) \pgm{z,s^-}(d \eta) \\
= \delta_{i,j} + \sigma_{\epsilon^{-1} z}^{(i)}(\alpha^{-1} s^-)\sigma_{\epsilon^{-1} z}^{(j)}(\alpha^{-1} s^-) + err_{\gamma}(\nn,\sigma,z,s^-)
\end{multline}
where
\[
|err_{\gamma}(\nn,\sigma,z,s^-)| \leq C(\nn)\gamma^{1-\nu}\begin{cases}
 |\sigma_{\epsilon^{-1} z}^{(i)}(\alpha^{-1} s^-)| + |\sigma_{\epsilon^{-1} z}^{(j)}(\alpha^{-1} s^-)| &\text{ if } s < \alpha^{-1}\taun\\
0 &\text{ otherwise}
\end{cases}
\]
Since we will always use the stopped dynamic, in order to keep the notation cleaner, for the rest of the paper we will abuse the notation omitting from $\sigma$, $\hg$ and all the other fields the dependence on $\nn$.
We have the following

\begin{proposition}
\label{prop:boundonrate}
For all $\lambda > 0$ and $q > 1$ there exists $C= C(q,\nn,\lambda,T)$, depending on the constant in \eqref{e:hp2bis}, such that, for some $\gamma_0>0$
\[
\sup_{0 < \gamma \leq \gamma_0}\sup_{0 \leq s \leq T}\sup_{z \in \Le} s^{\lambda} \E\left[\left| err_{\gamma}(\nn,\sigma,z,s^-)\right|^{q}\right]\leq C \left(\gamma^{q(1-\nu)} + \alpha^{\lambda} \right)
\]
where $s,T$ are macroscopic times. Moreover, there exists  $C = C(\nn,q,T)$ such that
\[
\sup_{0 < \gamma \leq \gamma_0}\sup_{0 \leq s \leq T}\sup_{z \in \Le}\E\left[|\Qm^{i,j}( s, z)|^q\right] \leq C\;.
\]
\end{proposition}

\begin{proof}
It is sufficient to notice that the Radon-Nikodym derivative 
\[
e^{- 2 \gamma^{1-\nu} \nn |\eta|} \leq d\pgm{z, s^- }/d\nug(d\eta) \leq e^{2 \gamma^{1-\nu} \nn |\eta|}
\]
and that the measure $\nug$ has exponential moments by \eqref{e:hpexponentialmoment}. Then
\begin{multline*}
r^{\lambda}\E\left[ \left|\sigma_{\epsilon^{-1} y}^{(i)}(\alpha^{-1} r)\right|^q \right] \\
\leq C r^{\lambda}\P(T_0 \geq r) + r^{\lambda}\P(T_0 < r) \int_S \left|\eta^{(i)}\right|^q \pgm{y,  r^- }(d\eta) \leq C r^{\lambda} e^{-\alpha^{-1}r} + C
\end{multline*}
for $0 \leq r \leq T$, $y \in \Lambda_N$ and where $T_0$ it the first (macroscopic) time  that the spin in $y$ jumps. To go from the first line to the second one we used the assumption \eqref{e:hp2bis} on the initial condition.
\end{proof}

\subsection{Limiting SPDE}

\label{subsec:limitingspde}

In this section we define the solution to the limiting equation \eqref{e:Phi2n2}, which is a multidimensional version of the $\Phi^{2n}_2$ equation, at which the discrete process introduced in Section~\ref{sec:model-theorems} will converge to.\\
For $j = 1,\dots,m$, each $X^{(j)}$ turns out to be a process with values in the Besov space of negative regularity $\Cc^{-\nu}(\T^2)$ for any $\nu > 0$, where $\Cc^{\alpha}$ is a separable version of the usual Besov space $\mathcal{B}^{\alpha}_{\infty,\infty}$, defined as the closure of the space of smooth functions under the norm
\[
\Norm{\Cc^{\alpha}}{g} \eqdef \sup_{k \geq -1} 2^{\alpha k }\Norm{L^{\infty}(\T^2)}{\delta_{k} g}
\]
where $\delta_{k}$ is the $k$-th Paley-Littlewood projection and $g: \T^2 \to \R$ is a smooth function (see \cite{MourratWeberGlobal} for a collection of results about Besov spaces and \cite[App.~A]{MourratWeber} for the detail of the construction in our case). The multivalued stochastic quantization equation in two dimension is given by
\begin{equation}
\label{e:cap2limitingspde}
d X^{(j)}(\cdot,t) = \Delta X^{(j)}(\cdot,t) dt + \pp^{(j)}(X)(\cdot,t) dt + \sqrt{2} dW^{(j)}(t)
\end{equation}
with initial conditions $\Xn \in \Cc^{-\nu}(\T^2;\R^m)$. The processes $W^{(j)}$ are $m$ independent white noises on $\T^2$ and $\pp^{(j)}$ are odd renormalized polynomials of degree $2n-1$ of the form \ref{e:polylimitdef} satisfying assumption \eqref{e:hppoly}.\\ 
The existence and uniqueness theory behind \eqref{e:cap2limitingspde} follows from the work of \cite{dPD} and \cite{tsatsoulis2016spectral}.
The analysis of equation has already been performed in our context by \cite{MourratWeber} for an odd polynomial of degree $3$ and $m=1$ and in \cite{ShenWeber} in case of an odd polynomial of any degree and $m=1$, the extension to the multidimensional case it's straightforward.\\
In order to fix some notations and definitions useful in Section~\ref{sec:nonlinear} we will briefly summarize the results for \eqref{e:cap2limitingspde}, following the treatment in \cite{ShenWeber}.\\
In this context we take the definition of the Fourier transform:  for a function $Y:\Le \to \C$ define for $\omega \in \{-N,\dots,N\}^2$
\begin{equation}
\label{e:Fouriertransform}
\hat{Y}(\omega)  \eqdef \sum_{x \in \Le} \epsilon^2 Y(x) e^{-\pi i x \cdot \omega}\;.
\end{equation}
With this definition we have the following inversion formula, for $x \in \Le$ 
\begin{equation}
\label{e:Fourierinversion}
Y(x) = \frac{1}{4} \sum_{\omega \in \{-N,\dots,N\}^2} \hat{Y}(\omega) e^{\pi i x \cdot \omega}\;.
\end{equation}
The discrete Fourier transform \eqref{e:Fouriertransform} is also defined for a general field $Y:\T^2 \to \C$ over the whole torus and in particular one can define the approximation to $Y$ denoted with $\Ex(Y)$ as
\begin{equation}
\label{e:extentionoperator}
\Ex(Y)(x) \eqdef \frac{1}{4} \sum_{\omega \in \{-N,\dots,N\}^2} \hat{Y}(\omega) e^{\pi i x \cdot \omega}
\end{equation}
for $x \in \T^2$. In particular $Y(x)$ coincides with $\Ex(Y)(x)$ for $x \in \Le$. Moreover, for $Y: \Le \to \R$ we define
\begin{equation}
\label{e:highlowestensiondef}
Y^{high} \eqdef \sum_{2^k \geq \frac{3}{8} \frac{\epsilon^{-1}}{2n-1}} \delta_k Y \qquad Y^{low} \eqdef \sum_{2^k < \frac{3}{8} \frac{\epsilon^{-1}}{2n-1}} \delta_k Y
\end{equation}
as processes over the continuous torus $\T^2 \to \R$, and analogous definitions can be given in the vector valued processes. The threshold $\frac{3}{8}\frac{1}{2n-1}\epsilon^{-1}$ has been chosen in such a way that the operation of taking the $2n-1$ power of $Y^{low}$ commutes with the extension operator.\\
Let $W_{\epsilon}^{(j)}$ be a smooth approximation of the white noise, given by truncating the Fourier modes with frequencies $|\omega| > \epsilon^{-1}$.
Assume $Z_{\epsilon}^{(j)}$ for $j=1...m$ to be the smooth solution of the heat equation on the torus $\T^2$ with Gaussian noise $W_{\epsilon}^{(j)}$
\begin{equation}
\label{e:limitlinearapproxepsilon}
\begin{cases}
\partial_t Z_{\epsilon}^{(j)}(t) &= \Delta Z_{\epsilon}^{(j)}(t) + \sqrt{2}d W_{\epsilon}^{(j)}(t)\\
Z_{\epsilon}^{(j)}(t) &= 0
\end{cases}
\end{equation}
In particular $Z_{\epsilon}$ is a martingale. For a positive integer $k \in \N$, define the Wick power $Z^{(j):k:}_{\epsilon}(t,x)\eqdef H_{k}(Z^{(j)}_{\epsilon}(t,x),\mathfrak{c}_{\epsilon}(t))$ where
\[
\mathfrak{c}_{\epsilon}(t) = \E[ (Z(t,0)^{(j)})^2 ] = \frac{t}{2} +  \sum_{0 < |\omega| \leq \epsilon^{-1}:\omega \in \Z^2} \frac{1}{4 \pi^2 |\omega|^2} \left(1 - e^{- 2 t \pi^2 |\omega|^2} \right)
\]
and extend the above definition to a general multiindex $\kb \in \N^m$ using the definition of multidimensional Hermite polynomial in the case of independent components (\ref{cap2multdimorthermite})
\[
Z^{:\kb:}_{\epsilon}(t,x) = H_{\kb}\left( ( Z^{(1)}_{\epsilon}(t,x) , \dots , Z^{(m)}_{\epsilon}(t,x)),\mathfrak{c}_{\epsilon}(t)I_m ) \right) = \prod_{j = 1}^m H_{k_j}(Z^{(j)}_{\epsilon}(t,x),\mathfrak{c}_{\epsilon}(t))\;.
\]
An important remark that will be the reason for some different definitions in the following sections is that the Wick powers $Z^{:\kb:}_{\epsilon}$ do have Fourier modes of frequencies of order $|\omega| > \epsilon^{-1}$. This will be important when dealing with the nonlinear process.\\

We will now state a result which is a multidimensional dynamical version of \cite[Lemma~3.2]{dPD} and \cite[Prop~3.1]{MourratWeber}. The proof of the result follows essentially from \cite{MourratWeber}, with the use of the independence of the components.
\begin{proposition}
\label{cap2linearconvcontinuous}
For $T > 0$ , $\nu > 0$ and $\kb \in \N^m$, the stochastic processes $Z_{\epsilon}^{:\kb:}$ converges a.s. and in any stochastic $L^p$ space in the metric of $C \left([0,T],\Cc^{-\nu}\right)$.\\
We will refer to this limit with $Z^{:\kb:}(t,\cdot)$.
\end{proposition}

We will need moreover \cite[Prop.~2.1]{dPD} in the form of the Besov inequality, which we will restate below for the sake of convenience.
\begin{proposition}
\label{prop:productbesov}
Let $a,b > 0$ with $b-a > 0$. Assume $A$ to be in $\Cc^{-a}(\T^d)$ and $B$ to be in $\Cc^b(\T^d)$. Then the pointwise product $A B $ (defined on a dense subspace of $\Cc^{-a}(\T^d)$) can be extended to a bilinear map random variable $\Cc^{-a}(\T^d) \times \Cc^{b}(\T^d) \to \Cc^{-a}(\T^d)$ and
\begin{equation}
\Norm{\Cc^{-a}(\T^d)}{ AB } \lesssim \Norm{\Cc^{-a}(\T^d)}{ A }  \Norm{\Cc^{b}(\T^d)}{ B } \;.
\end{equation}
\end{proposition}

The solution of the linear equation \eqref{e:limitlinearapproxepsilon} in $\T^2$ started with $\Xn$ initial conditions can be written as
\begin{equation}
\tilde{Z}_{\epsilon}^{(j)}(t) \eqdef Y^{(j)}(t) + Z_{\epsilon}^{(j)}(t)\;.
\end{equation}
The process $\tilde{Z}^{(j)}(t,\cdot)$ enjoys, by Proposition~\ref{prop:productbesov} and the properties of the heat semigroup \eqref{e:heatsgregimpr}
\begin{equation}
\sup_{0 \leq t \leq T} t^{(\betareg + \nu)\frac{|\kb|}{2}}\Norm{\Cc^{-\nu}}{\tilde{Z}^{:\kb:}(t,\cdot)} \leq C^* 
\label{e:cap2linprocregular}
\end{equation}

where $C^* = C^*(T,\Norm{\Cc^{-\nu}}{X_0},\betareg,\nu,n,\Norm{\Cc^{-\nu}}{Z^{:\kb:}} \text{ for } |\kb|\leq 2n-1)$. Let
\[
\mathfrak{c}_{\epsilon} = \frac{1}{4}\sum_{0 < |\omega| \leq \epsilon^{-1}} \frac{1}{ \pi^2 |\omega|^2} 
\]
and define the difference
\[
A_{\epsilon}(t) \eqdef \Ce - \Ce(t) = - \frac{t}{2 } + \frac{1}{4}\sum_{0 < |\omega| \leq \epsilon^{-1}} \frac{e^{- 2 t \pi^2 |\omega|^2}}{ \pi^2 |\omega|^2}  \qquad A(t) \eqdef \lim_{\epsilon \to 0} A_{\epsilon}(t)
\]
With $A(t) \sim \log(t^{-1})$ for $t \to 0$ and $|A(t)| \sim t$ as $t \to \infty$.\\
We are now ready to describe the notion of solution to equation \eqref{e:cap2limitingspde}, first defined in \cite{dPD}. We say that $X$ solves \eqref{e:cap2limitingspde} if $X(t,\cdot) =  \tilde{Z}(t,\cdot) + V(t,\cdot)$ and the process $V$, solves the PDE
\begin{equation}
\label{e:cap2PDEforV}
\begin{cases}
\partial_t V^{(j)}(\cdot,t) &= \Delta V^{(j)}(\cdot,t) + \overline{\Psi}^{(j)}\left(t, (\tilde{Z}^{:\kb:})_{|\kb| \leq 2n-1} \right)\left(V_{\gamma} (\cdot,t)\right) \\
V^{(j)}(\cdot,t) &= 0
\end{cases}
\end{equation}
For
\begin{equation}
\label{e:barPsidef}
\overline{\Psi}^{(j)}\left(t, (\tilde{Z}^{:\kb:})_{|\kb| \leq 2n-1} \right)\left(V_{\gamma} (\cdot,t)\right)  = \pp^{(j)}(\tilde{Z} (\cdot,t) + V(\cdot,t))
\end{equation}
where
\begin{equation}
\pp^{(j)}(\tilde{Z} (\cdot,t) + V(\cdot,t)) \eqdef \sum_{ |\mathbf{b}| + |\mathbf{a}|\leq 2n-1} b_{\mathbf{a},\mathbf{b}}^{(j)}(t) V^{\mathbf{a}}(\cdot,t)\tilde{Z}^{:\mathbf{b}:}(\cdot,t)
\end{equation}
for some coefficients with
\begin{equation}
\label{e:epseqboundoncoeff}
|b_{\mathbf{a},\mathbf{b}}^{(j)}(t)| \lesssim |A(t)|^{\frac{2n-1 - |\mathbf{a} + \mathbf{b}|}{2}}\;.
\end{equation}
 We recall that the products between $\tilde{Z}^{:\mathbf{b}:}(\cdot,t)$ and $V^{\mathbf{a}}(\cdot,t)$ are well defined thanks to Proposition~\ref{prop:productbesov} and the fact that $V(t,\cdot) \in \Cc^{\alpha}(\T^2,\R^m)$ for any $\alpha < 2$. In particular \eqref{e:cap2PDEforV} is a PDE that depends on a given realization of the linear process $\tilde{Z}$ and its Wick powers.\\
The next theorem completes the existence and uniqueness theory behind the limiting equation. 
\begin{theorem}
\label{thm:existuniqlimit}
For $ 0 < \nu < \frac{2}{2n-1}$ small enough and initial data $\Xn \in \Cc^{-\nu}(\T^2,\R^m)$. For a realization of
\[
\mathbf{z}_{\kb} \in L^{\infty} ([0,T];\Cc^{-\nu}(\T^2)) \qquad\text{ for } |\kb| \leq 2n-1\;,
\]
let $\mathcal{S}_T$ the solution map that associates to $( \mathbf{z}_{\kb})_{|\kb|\leq 2n-1}$ the solution $V$ to the PDE \eqref{e:cap2PDEforV}.\\
The solution map exists, it is unique and it is Lipschitz continuous for all $\nu,\kappa > 0$ with $\kappa > (2n-1)\nu$ sufficiently small as
\begin{align*}
\mathcal{S}_T :  [L^{\infty} ([0,T];\Cc^{-\nu}(\T^2))]^{n^*} \mapsto &\Cc \left( [0,T] , \Cc^{2 - \nu - \kappa}(\T^2,\R^m)\right)\\
\left\lbrace \mathbf{z}_{\kb} \right\rbrace_{|\kb|\leq 2n-1} \to & V
\end{align*}
\end{theorem}
\begin{proof}
The same proof in \cite{MourratWeber,ShenWeber,MourratWeberGlobal} applies to the vector valued problem, see also \cite[Sec.~3]{tsatsoulis2016spectral} for some bounds which are independent on the initial conditions. 

\end{proof}
We now spend few words about the existence of the solution for all times. With a general polynomial, the process is expected to have a blowup in finite time, but in fact, the Assumption~\ref{e:hppoly} guarantees the well posedness of the solution for all times. The proof of this fact for $m=1$ is presented in \cite[Sec. 6]{MourratWeberGlobal}, and it consists in deriving $L^p$-bounds for the process $V^{(j)}$ testing $V^{(j)p-1}$ with \ref{e:cap2PDEforV} via the assumption on the leading coefficient of the polynomial given in \ref{e:hppoly}. The application to our case is straightforward.

We are now ready to state the main result of this article, which will be proved in Section~\ref{sec:nonlinear}.

\begin{theorem}
\label{thm:maintheorem}
Let $\Xg$ the multidimensional process defined from the Glauber dynamic as in Section~\ref{sec:model-theorems}.\\
For $\nu>0$ small enough, let the reference measure $\nug$ and the initial condition satisfy the assumptions \eqref{e:hp1}, \eqref{e:hp2bis}, \eqref{e:hpexponentialmoment}, \eqref{e:hpmeasure}, \eqref{e:hppoly} in Section~\ref{subsec:assumptionsnug}.\\
Then the process $\Xg$ converges in distribution in $\mathcal{D}\left([0,T];\Cc^{-\nu} \right)$ to the solution, in the sense of Section~\ref{subsec:limitingspde}, $X$ of the SPDE in \eqref{e:cap2limitingspde}.
\end{theorem}

Recall that by Remark~\ref{remark:stoopingtime}, it is sufficient to work under the condition $\taun > T$.\\
The proof of the main theorem follows exactly as in \cite[Theorem~2.1]{MourratWeber}, where the only bound needed is provided by Proposition~\ref{prop:propvdiffstat}. 

Theorem~\ref{thm:maintheorem} implies, for instance the following corollary
\begin{corollary}
\label{cor:mvectormodel}
Consider the $m$-vector model defined in Remark~\ref{Remark:mourratwebershen}. Suppose that the law at time zero satisfies $\E\Norm{\Cc^{-\nu}}{\Xng} < \infty$. Then the Glauber dynamic converges to the solution of \eqref{e:mvectorlimit}.
\end{corollary}

\begin{proof}
It is easy to see that the assumptions in Subsection~\ref{subsec:assumptionsnug} are satisfied (except the condition $\int_S |\eta|^2 \nug(d\eta) = 1$). We can then apply the theorem to $\sigma_x'(t) = \sqrt{m}\sigma_x(t)$ with the new invariant measure $\beta' = \frac{1}{m}\beta$. It is easy to see that the calculation in \eqref{e:momgenfun1} yields
\[
(\beta' m) \hg'(x,t) - \frac{1}{m+2} (\beta' m)^3 |\hg'(x,t)|^2\hg'(x,t) + \Oo(|\hg'(x,t)|^5)\:.
\]
If $\beta' m = 1 + \gamma^2 \cg + o(\gamma^2)$, then Theorem~\ref{thm:maintheorem} implies that $\Xg' := \delta^{-1}\hg' = \delta^{-1} \sqrt{m} \hg = \sqrt{m} \Xg$ converges to
\[
\partial_t X' = \Delta X' - \frac{1}{m+2} :|X'|^2 X': + \xi
\]
and therefore the original field converges to the solution of \eqref{e:mvectorlimit}.
\end{proof}

\section{The linearized process}
\label{sec:linear}

In order to prove convergence in law for the Glauber dynamic defined in Section~\ref{sec:model-theorems}, we will introduce the linearized dynamic and start proving convergence in law of the linearized dynamic to the solution of the multivariate heat equation.
The strategy that we will be using is the same as \cite{MourratWeber}. We will first show tightness of the linear process and then characterizing the law with the martingale problem satisfied by the heat equation.\\
In this section the linearized dynamic $\Zg$ is presented: in order to prove the tightness of the laws of the different processes as $\gamma \to 0$ we introduce the approximation $\Rg$ to $\Zg$. In order to prove that the limiting law satisfies the martingale problem, instead of introducing a stopping time, as in \cite{MourratWeber}, we use some a priori bounds on the nonlinear dynamic.\\

\subsection{Wick powers of the rescaled field}
\label{subsec:Wick-powers}

We will write the solution $\Xg$ on $\Le$ as
\begin{equation}
\label{e:Xduhamel}
\begin{split}
\Xg(\cdot,t) &= \Pg{t} \Xng + \int_0^t \Pg{t-s} \Kg \aste \left( \ppg(\Xg)(\cdot,s) + \Eg(\cdot,s)\right) ds \\
&+\int_0^t \Pg{t-s} d\Mg(\cdot,s)
\end{split}
\end{equation}
where $\ppg$ has been defined in \eqref{e:polydef} and $\Xng = \delta^{-1}\hg(\cdot,0)$ is the initial condition. We will now denote by $\Zg(x,t)$ the mild solution
\begin{equation}
\label{e:Zdef}
\Zg(x,t) \eqdef \int_{r=0}^t \Pg{t-r}d\Mg(x,s)
\end{equation}
of the approximation of the stochastic heat equation on $\Le$
\begin{equation}
\label{e:SHE:zeroinit}
\begin{cases}
d\Zg(x,t) &= \Dg \Zg(x,t) + d\Mg(x,t)\\
\Zg(x,0) &= 0
\end{cases}
\end{equation}
And we will extend $\Zg$ to the whole torus $\T^2$, by considering the trigonometric polynomial of degree $N$ that coincides with $\Zg$ on $\Le$.
Following \cite{MourratWeber,ShenWeber} we introduce a martingale approximation of $\Zg$, defined for $s\leq t$ as
\begin{equation}
\label{e:Rdef}
\Rg(x,s) = \int_{[0,s)} \Pg{t-r} d\Mg(x,r)\;.
\end{equation}

From its definition $\Rg$ is a martingale for $0 \leq s \leq t$ and $\lim_{s \to t} \Rg(x,s) =\Zg(x,t)$ in $\Cc^{-\kappa}$ for any $\kappa>0$.  We will now define recursively the higher renormalized powers of $\Rg$. Such a definition might not seem intuitive, but it has the advantage of producing automatically a martingale.\\
We recall that $\Rg = (\Rg^{(1)}, \dots, \Rg^{(m)})$ and every renormalized power is indexed by $\kb = (k_1,\dots,k_m) \in \N^m$. We then call the degree of the multiindex $\kb$ the quantity $|\kb| = \sum_{i=1}^m k_i$. 
To be consistent with the notations, if $|\kb| = 1$ we simply consider $\Rg^{\kb} \eqdef \Rg^{(i)}$ if $\kb$ is nonzero only in the $i$-th position.

We then define, for $x \in \Le$ and $0 \leq s \leq t$
\begin{equation}
\label{e:Rpowersdef}
\Rg^{:\kb:}(x,s) = \Rg^{:k_1,k_2,\dots,k_m:}(x,s) \eqdef \sum_{i=1}^m k_i \int_{[0,s)} \Rg^{:k_1,\dots,k_i-1,\dots,k_m:}(x,r^{-}) d \Rg^{(i)}(x,r)\;.
\end{equation}
Where the left limit $\Rg^{:k_1,\dots,k_i-1,\dots,k_n}(x,r^{-})$ guarantees that the above quantity is a martingale for all $\kb$.\\
The above definition has the drawback that it is defined only on $\Le$. In order to extend it to the whole torus $\T^2$, it turns out to be more convenient to work with another definition of $\Rg^{:\kb:}$ via the Fourier series
\begin{equation}
\label{e:Rpowersextension}
\hRg^{:\kb:}(\omega,s) \eqdef  \sum_{i=1}^m k_i \int_{[0,s)} \frac{1}{4}\sum_{\omega' \in \Z^2} \hRg^{:k_1,\dots,k_i-1,\dots,k_m:}(\omega- \omega',r^{-}) d \hRg^{(i)}(\omega',r)\;.
\end{equation}
It is immediate to verify that \eqref{e:Rpowersextension} defines an extension to $\T^2$ of \eqref{e:Rpowersdef} and it is a Fourier polynomial of degree $4|\kb|\epsilon^{-2}$.\\
For multiindex $\kb$ and $x \in \T^2$, $0 \leq t$ define
\begin{equation}
\label{e:Zpowerdef}
\Zg^{:\kb:}(x,t) \eqdef \lim_{s \nearrow t}\Rg^{:\kb:}(x,s)\;.
\end{equation}

As the notation suggests, the quantities $\Zg^{:\kb:}(\cdot,t)$ are going to be an approximation for the Wick powers of the solution of the linearized process. This relation will be made more precise in Proposition~\ref{prop:cap2boundonHermiteerror} in the next section.\\

The rest of the section is devoted to showing that the processes $\Rg^{:\kb:}(\cdot,t)$ belongs to $\Cc([0,T],\Cc^{-\nu})$ for any small $\nu > 0$, which is the content of Proposition~\ref{prop:cap2boundsonR}.\\
Using \eqref{e:quadvarMg} the quadratic covariation of  \eqref{e:Rdef}  is given by
\begin{gather*}
\ang{R_{\gamma,t}^{(i)}(z_1,\cdot) , R_{\gamma,t}^{(j)}( z_2,\cdot)}_s \\
= \int_{[0,s)}\sum_{y_1,y_2 \in \Le} \epsilon^{2d} P^{\gamma}_{t-r}(z_1 - y_1) P^{\gamma}_{t-r}(z_2 - y_2) d \ang{M^{(i)}_{\gamma,\cdot}(y_1) , M^{(j)}_{\gamma,\cdot}(y_2)}_r \\
 \label{e:cap2quadrvarlin2}
= \sum_{z \in \Le } \epsilon^d \int_{[0,s)} P^{\gamma}_{t-r} \Kg(z-y)^2 \Qm^{i,j}(r,z) dr\;.
\end{gather*}

By Proposition~\ref{prop:boundonrate}, the expectation of \eqref{e:cap2quadrvarlin2} is bounded by
\begin{equation}
\label{e:Rquadexpbound}
\begin{split}
\left(\E \left|\ang{R_{\gamma,t}^{(i)}(y,\cdot) , R_{\gamma,t}^{(j)}(y,\cdot)}_s \right|^q\right)^{1/q} \leq C(\nn,q) \sum_{z \in \Le } \epsilon^2 \int_{[0,s)} P^{\gamma}_{t-r} \Kg(z-y)^2 dr 
\end{split}
\end{equation}
and therefore, for $s < t$, $R_{\gamma,t}(y,s)$ is a true martingale.\\

We expect to get the orthogonality of the martingales for $i \neq j$ in the limit as $\gamma \to 0$.

The next estimate is needed in Proposition~\ref{prop:cap2boundsonR} to control the norm of the iterated integrals of the process $\Rg^{:\kb:}(\cdot,s)$. This is essentially lemma 4.1 of \cite{MourratWeber} for a particular choice of the kernels. We will provide a proof of it since it is a key estimate, even though the proof follows closely the one in \cite{MourratWeber}, with the only difference that in the proof an H\"older inequality has been used to deal with the fact that the spins in our model are not bounded uniformly by $1$. 
Furthermore, the result is not stated in its more general form in order to avoid the introduction of notations that are not going to be used in the rest of the paper.\\
For the next proposition we will use the notation $\Rg^{:\kb:}(\varphi,s)$ to denote the $L^2(\Le)$ scalar product between $\Rg^{:\kb:}(\cdot,s)$ and a test function $\varphi$.
\begin{proposition}
\label{prop:RroughBounds}
Let $\varphi : \T^2 \to \R$ be a smooth test function, let $p > 2$ and $ \kappa > 0$, then there exists a constant $C = C(\kb,p,\nn,\kappa)$, such that
\begin{equation*}
\begin{split}
 &\left(\E \sup_{0 \leq s \leq t} |\Rg^{:\kb:}(\varphi,s) |^p \right)^{\frac{2}{p}} \\
& \leq C \sum_{i=1}^m k_i \int_{r=0}^t \sum_{y \in \Le} \epsilon^2 \E \left[\left|\Rg^{:k_1,\dots,k_i-1,\dots,k_m:}(\Pg{t-r}\Kg(\cdot-y)\varphi,r^{-})  \right|^{p+\kappa} \right]^{\frac{2}{p+\kappa}} dr \\
&+ C (\delta^{-1}\epsilon^{2} )^{2 - \kappa} \sum_{i=1}^m k_i\ \E \left[\sup_{r \leq t}\sup_{y \in \Le} \left|  \Rg^{:k_1,\dots,k_i-1,\dots,k_m:}(\Pg{t-r}\Kg(\cdot-y)\varphi,r^{-}) \right|^{p + \kappa} \right]^{\frac{2}{p+\kappa}}
 \end{split}
\end{equation*}
and reiterating the above formula we obtain
\begin{equation}
\label{e:Rpowertested}
\begin{split}
&\left( \E \sup_{0 \leq s \leq t} | \Rg^{:\kb:}(\varphi,s) |^p \right)^{\frac{2}{p}}  \\
 &\leq C \int_{r_1 = 0}^t  \int_{r_2 = 0}^{r_1} \cdots \int_{r_{|\kb|}=0}^{r_{|\kb|-1}} \sum_{\bar{y} \in (\Le)^{|\kb|}} \epsilon^{2|\kb|} \ang{\varphi, F_{|\kb|}^t(\bar{y},\bar{r})}_{L^2(\Le)}^2 d\bar{r} + \mathbf{err}
 \end{split}
\end{equation}
where 
\begin{equation}
\label{e:Fdefin}
F_l^t(y_1,\dots,y_l,r_1,\cdots,r_l)(x) = \prod_{i=1}^l \Pg{t-{r_i}}\Kg(x-y_i)
\end{equation}
and the error term is given by
\begin{equation}
\label{e:Rpowererror}
\begin{split}
\mathbf{err} &= C (\delta^{-1}\epsilon^2)^{2 - \kappa} \sum_{l=1,\dots,| \kb|}  \int_{r_1 = 0}^t  \int_{r_2 = 0}^{r_1} \cdots \int_{r_{l-1}=0}^{r_{l-2}} dr_1 \cdots dr_{l-1} \sum_{y_1,\dots,y_{l-1} \in \Le} \epsilon^{2(l-1)} \times \\
& \times \sup_{\mathbf{a}\in \N^m : |\mathbf{a}|=l}\E \left[\sup_{\substack{ r_{l}< r_{l-1} \\ y_l \in \Le}}\left| \Rg^{:\kb  - \mathbf{a}:}\left(\varphi F_{l}^t(y_1,\dots,y_{l},r_1,\dots,r_{l}),r_{l}\right)\right|^{p + l\kappa} \right]^{\frac{2}{p + l\kappa}} 
\end{split}\:.
\end{equation}
\end{proposition}

\begin{proof}
It is easy to see that the above formula holds for $|\kb|=1$ and any $p>2$. We then use the Burkholder-Davis-Gundy inequality and the induction on $|\kb|$ to prove that it holds also for any $\kb \in \N^m$.\\
From the recursive formula \eqref{e:Rpowersdef} we compute the quadratic variation of
\begin{gather*}
\Rg^{:k_1,k_2,\dots,k_m:}(\varphi ,s) = \sum_{i=1}^m k_i \int_{[0,s)} \sum_{x \in \Le}\epsilon^2 \varphi(x)\Rg^{:k_1,\dots,k_i-1,\dots,k_m:}(x,r^{-}) d \Rg^{(i)}(x,r)\;.
\end{gather*}
In order to apply the Burkholder-Davis-Gundy inequality we have to estimate the quadratic variation of the process and the size of the jumps.\\
The quadratic variation of the process is then
\begin{gather*}
\ang{\sum_{x \in \Le} \epsilon^2\varphi(x) \Rg^{:k_1,k_2,\dots,k_m:}(x,\cdot)}_s\leq C(\kb)\sum_i k_i^2 \int_{[0,s)} \sum_{x,y \in \Le}\epsilon^4 \varphi(x)\varphi(y )\times \\
\times \Rg^{:k_1,\dots,k_i-1,\dots,k_m:}(x,r^{-})\Rg^{:k_1,\dots,k_i-1,\dots,k_m:}(y,r^{-}) d \ang{\Rg^{(i)}(x,\cdot),\Rg^{(i)}(y,\cdot)}_r\\
\lesssim \sum_i k_i^2 \int_{[0,s)} \sum_{z \in \Le} \epsilon^2 |\Rg^{:k_1,\dots,k_i-1,\dots,k_m:}(\varphi(\cdot) \Pg{t-r}\Kg(\cdot-z),r^{-})|^2 \Qm^{i,i}(r^-,z)dr\;.
\end{gather*}
We define the jump of a cadlag process at time $r \in \R$, as $\Delta_r  \Rg^{:k_1,k_2,\dots,k_m:}(\varphi,r)$ and it is given by 
\begin{equation}
\label{e:RroughBoundsjump}
 \epsilon^2 \delta^{-1}  \sum_i k_i \sup_{\substack{ z \in \Le\\0 \leq r \leq s}} \left|  \Rg^{:k_1,\dots,k_i-1,\dots,k_m:}(\varphi(\cdot)\Pg{t-r} \Kg(\cdot-z),r^{-})  \right| |\Delta_{r}\sigma_z(\alpha^{-1} r) 	|
\end{equation}
where $\Delta_{r}\sigma_z(\alpha^{-1} r) = \sigma_z(\alpha^{-1} r) -\sigma_z(\alpha^{-1} r^-) $. Therefore we have that
\begin{multline}
\label{e:RroughBoundssum}
\left(\E \sup_{0 \leq s \leq t}  |\Rg^{:k_1,k_2,\dots,k_m:}(\varphi,s)|^p\right)^{\frac{2}{p}} \\
\leq C(p) \left(\E\ang{ \Rg^{:k_1,k_2,\dots,k_m:}(\varphi,\cdot)}_t^{\frac{p}{2}}\right)^{\frac{2}{p}} \\
+ C(p) \left(\E \sup_{0 \leq s \leq t} \left|\Delta_r  \Rg^{:k_1,k_2,\dots,k_m:}(\varphi,r) \right|^p\right)^{\frac{2}{p}}
\end{multline}
Then use Minkowski's inequality with exponent $p/2 > 1$ 
\begin{multline*}
 \left(\E \ang{\Rg^{:k_1,k_2,\dots,k_m:}(\varphi,\cdot)}_s^{\frac{p}{2}} \right)^{\frac{2}{p}} \leq	C(\kb,p)\sum_i k_i^2 \int_{[0,s)}  \sum_{z \in \Le} \epsilon^2 \times \\
 \times  \E \left[ \left(|\Rg^{:k_1,\dots,k_i-1,\dots,k_m:}(\varphi(\cdot) \Pg{t-r}\Kg(\cdot-z),r^{-}) |^2 \Qm^{i,i}(r^-,z) \right)^{\frac{p}{2}} \right]^{\frac{2}{p}} dr 
\end{multline*}
and at this point we use the H\"older inequality to separate the term $\Qm^{i,i}(r^-,z)$
\begin{multline*}
 \E \left[ \left(|\Rg^{:k_1,\dots,k_i-1,\dots,k_m:}(\varphi(\cdot) \Pg{t-r}\Kg(\cdot-z),r^{-}) |^2 \Qm^{i,i}(r^-,z) \right)^{\frac{p}{2}} \right]^{\frac{2}{p}} \\
 \leq C(p,\kappa) \E\left[\left| \Rg^{:k_1,\dots,k_i-1,\dots,k_m:}(\varphi(\cdot) \Pg{t-r}\Kg(\cdot-z),r^{-})\right|^{p+\kappa} \right]^{\frac{2}{p+\kappa}}
\end{multline*}
Where in the last line we used the bounds over the moments of $\Qm^{i,i}(r^-,z)$ provided in Proposition~\ref{prop:boundonrate}. This is the only difference with the proof of \cite{MourratWeber}, where a uniform bound on $\Qm^{i,i}(r^-,z)$ is available. We can then use induction on the integrand, with the new test function $\varphi(\cdot) \Pg{t-r}\Kg(\cdot-z)$.\\
We now bound the jump part inside the summation in \eqref{e:RroughBoundsjump} with Lemma~\ref{lemma:supboundonPK}
\[
|\delta^{-1} \epsilon^2\Pg{t-r}\Kg(x-z)| \leq \delta^{-1} \gamma^2 \log(\gamma^{-1})
\]
and using the H\"older inequality (considering $\frac{\kappa}{\kappa + p} + \frac{p}{p+\kappa} = 1$) together with
\begin{equation}
\label{e:Rboundsjumps}
\left(\E\sup_{0 \leq r \leq t, z \in \Le}|\Delta_{r^-}\sigma(z)|^{p + \frac{p^2}{\kappa}}\right)^{\frac{\kappa}{\kappa + p}} \leq \left(\E\sum_{z \in \Le} \sum_{0\leq r \leq t} |\sigma_z(\alpha^{-1} r)|^{q(p + \frac{p^2}{\kappa})}\right)^{\frac{\kappa}{(\kappa + p)q}} 
\end{equation}
where the last sum is over all jumps that happened at site $z$ in $[0,t]$. Since the number of jumps is a Poisson process with intensity bounded by $\alpha^{-1}$, the last expectation can be replaced by
\[
\E\sum_{z \in \Le} \sum_{0\leq r \leq t} |\sigma_z(\alpha^{-1} r)|^{q(p + \frac{p^2}{\kappa})} = \alpha^{-1}\E\sum_{z \in \Le} \int_0^t |\sigma_z(\alpha^{-1} r)|^{q(p + \frac{p^2}{\kappa})} dr \leq C \epsilon^{-2}\alpha^{-1}
\]
and if we choose $q$ large enough we have that the last line in \eqref{e:RroughBoundssum} is bounded by
\begin{multline*}
C(q,p,\kappa,\nn)(\epsilon^2 \delta^{-1})^2(\epsilon^2\alpha^1)^{-\frac{2\kappa}{q(p+\kappa)}} \times \\
 \times\E \left[\sup_{\substack{ z \in \Le\\0 \leq r \leq s}} \left| \sum_{x \in \Le}\epsilon^2\varphi(x)\Pg{t-r} \Kg(x-z)\Rg^{:k_1,\dots,k_i-1,\dots,k_m:}(x,r^{-})  \right|^{p+\kappa} \right]^{\frac{2}{p+\kappa}} 
\end{multline*}
where for $q$ large and by \eqref{e:scaling} we can assume $(\epsilon^2\alpha )^{-\frac{\kappa}{q(p+\kappa)}} \ll(\epsilon^2 \delta^{-1})^{-\kappa}$. This proves the inductive step. The rest of the estimates follows directly from the proof in \cite{MourratWeber}.
\end{proof}
\begin{remark}
\label{remark:puntualbound}
In the above proposition the regularity of $\varphi$ is not entering into the proof, hence it is easy to see that one could take as $\varphi$ the discrete Dirac delta on the lattice and, using Lemma~\ref{lemma:est-cap2kernel3}, obtain the bound 
\[
\E[\sup_{0 \leq s \leq t} |\Rg^{:\kb:}(x,s)|^p] \lesssim \log^{p|\kb|}(\gamma^{-1})\;.
\]
\end{remark}

A result similar to the one in Proposition~\ref{prop:RroughBounds} can be proven also for
\[
 \E \sup_{0 \leq r < t} |  \Rg^{:\kb:}(\varphi,r) - \Rg^{:\kb:}(\varphi,s \wedge r) |^p \qquad \E \sup_{0 \leq r < t} |  \Rg^{:\kb:}(\varphi,r) - R_{\gamma,s}^{:\kb:}(\varphi,s \wedge r) |^p\;.
\]
Since the proof is exactly the same as in the case of Proposition~\ref{prop:RroughBounds}, we only state the result.

\begin{corollary}

Under the same assumptions as Proposition~\ref{prop:RroughBounds} and the definition of $F_l^t$ given in \eqref{e:Fdefin} we have

\begin{equation}
\label{e:Rpowertested2}
\begin{split}
&\left( \E \sup_{0 \leq r < t} |  \Rg^{:\kb:}(\varphi,r) - \Rg^{:\kb:}(\varphi,s \wedge r) |^p \right)^{\frac{2}{p}}  \\
 &\leq C \int_{r_1 = s}^t  \int_{r_2 = 0}^{r_1} \cdots \int_{r_{|\kb|}=0}^{r_{|\kb|-1}} \sum_{\bar{y} \in (\Le)^{|\kb|}} \epsilon^{2|\kb|} \ang{\varphi, F_{|\kb|}^t(\bar{y},\bar{r})}_{L^2(\Le)}^2 d\bar{r} + \mathbf{err}_1
 \end{split}
\end{equation}

\begin{equation}
\label{e:Rpowertested3}
\begin{split}
&\left( \E \sup_{0 \leq r < s} |  \Rg^{:\kb:}(\varphi,r) - R_{\gamma,s}^{:\kb:}(\varphi, r) |^p \right)^{\frac{2}{p}}  \\
 &\leq C \int_{r_1 = 0}^s  \int_{r_2 = 0}^{r_1} \cdots \int_{r_{|\kb|}=0}^{r_{|\kb|-1}} \sum_{\bar{y} \in (\Le)^{|\kb|}} \epsilon^{2|\kb|} \ang{\varphi, F_{|\kb|}^t(\bar{y},\bar{r})- F_{|\kb|}^s(\bar{y},\bar{r})}_{L^2(\Le)}^2 d\bar{r} + \mathbf{err}_2
 \end{split}
\end{equation}

and the error terms have the same form of \eqref{e:Rpowererror} with the replacement of the kernel $F^t_l$ precisely as done in \eqref{e:Rpowertested2} and \eqref{e:Rpowertested3}.

\end{corollary}

With the above considerations we are ready to state the bounds on the solution of the linear dynamic and its Wick powers.
\begin{proposition}
\label{prop:cap2boundsonR}
There exists $\gamma_0 > 0$ such that the following holds. For every multiindex $\kb\in \N^n$ , $p>1$ $\nu > 0$, $T \geq 0$, $0 \leq \lambda < \frac{1}{2}$ and $0 < \kappa \leq 1$, there exists a constant $C = C(\kb, , \nu, T,\lambda,\kappa)$ such that $0\leq s \leq t \leq T$ and $0 < \gamma < \gamma_0$
\begin{align*}
\E \sup_{0 \leq r \leq t} \Norm{\Cc^{-\nu-2\lambda}}{ R_{\gamma,t}^{:\kb:}(\cdot,r)}^p &\leq C t^{\lambda p} + C \gamma^{p(1-\kappa)}\\
\E \sup_{0 \leq r \leq t} \Norm{\Cc^{-\nu-2\lambda}}{ R_{\gamma,t}^{:\kb:}(\cdot,r ) - R_{\gamma,s}^{:\kb:}(\cdot,r \wedge s )}^p &\leq C |t-s|^{\lambda p} + C \gamma^{p(1-\kappa)}\\
\E \sup_{0 \leq r \leq t} \Norm{\Cc^{-\nu-2\lambda}}{ R_{\gamma,t}^{:\kb:}(\cdot,r ) - R_{\gamma,t}^{:\kb:}(\cdot,r \wedge s )}^p &\leq C |t-s|^{\lambda p} + C \gamma^{p(1-\kappa)}
\end{align*}
And, taking the limit $s \to t$ of the martingales, the same bounds are satisfied by $\Zg$.
\end{proposition}

\begin{proof}
The proof is equal to \cite[Prop~4.2]{MourratWeber}, with the use of the bounds \eqref{e:Rpowertested}, \eqref{e:Rpowertested2}, \eqref{e:Rpowertested3}. We will only need to apply Proposition~\ref{prop:RroughBounds} repeatedly with $\varphi$ equals to the kernel of every Paley-Littlewood projection. 
\end{proof}

\begin{remark}
For the above estimates we didn't use the fact that, for $i  \neq j$, the martingales $M^{(i)}_t$ and $M^{(j)}_t$ are orthogonal in the limit. The calculation of the covariation will be addressed in the next section.
\end{remark}

\begin{remark}
\label{cap2remarkboundsonZ}
As the bounds in Proposition~\ref{prop:cap2boundsonR} for $R_{\gamma,t}^{:\kb:}(s,\cdot)$ are uniform in $s$, the same bounds are available for the process $\Zg$ defined in \eqref{e:Zdef}. 
\end{remark}

We now state a lemma that gives a better control over the high frequencies of the fluctuation field, which will be used when we will extend the powers of the linear process to the continuous torus in Section~\ref{sec:nonlinear}. The next lemma correspond to \cite[Lemma~4.6]{MourratWeber}, and the proof follows exactly the same steps.
\begin{lemma}
\label{lemma:highfreqcontrol}
Recall the definition of $\ZgH$ in \eqref{e:highlowestensiondef}. For all $p \geq 1$, $\kappa > 0$ and $T > 0$, there exists a constant $C = C(p,\kappa,T,\nn)$ such that for all $\gamma < \gamma_0$ and $0 \leq t \leq T$
\begin{equation}
\label{e:highfreqcontrol}
\E \left[ \Norm{L^{\infty}(\T^2)}{\ZgH(\cdot,t)}^p\right]^{1/p} \leq C \gamma^{1- \kappa} \;.
\end{equation}
\end{lemma}

\section{Tightness and convergence for the linearized system }
\label{sec:tightness-linear}

In this section we state the tightness result for the powers of the linearized process $\Zg$ given by \eqref{e:Zpowerdef} and we will characterize the limit with a martingale problem in Subsection~\ref{subsec:convergence-linear}. This is the main reason for the introduction of the stopping time in Subsection~\ref{subsec:stopping time}. 

For a separable metric space $\mathcal{A}$, denote with $\D([0,T],\mathcal{A})$ the Skorokhod space of cadlag function taking value in $\mathcal{A}$ endowed with the Skorokhod topology: this makes $\D([0,T],\mathcal{A})$ a metric space as well with the distance
\[
\dist_{\D(\R^+,\mathcal{A})} = \sup_{\lambda \in \Lambda_{[0,T]}} \max \left\lbrace \norm{\infty}{\lambda - id} , \norm{L^{\infty}[0,T]}{f \circ \lambda - g} \right\rbrace
\]
for $f,g \in \D(\R^+,\mathcal{A})$ and for $\lambda \in \Lambda_{[0,T]}$ the space of continuous reparametrization of the interval $[0,T]$.\\

The following proposition corresponds to \cite[Prop.~5.4]{MourratWeber} and provide the tightness result for the laws of the Wick powers. We recall that $\Zg$ is a multivariate process with $m$ components and $\Zg^{:\kb:}(t,\cdot) \in \Cc^{-\nu}(\T^d)$ by Proposition~\ref{prop:cap2boundsonR} and Remark \ref{cap2remarkboundsonZ}.

\begin{proposition}
\label{prop:Ztightness}
Denote by $\gamma_0$ the constant in \cite[Lemma~8.2]{MourratWeber}. For any multiindex $\kb\in \N^n$ and $\nu > 0$, the family $\left\lbrace \Zg^{:\kb:};\gamma \in (0,\gamma_0)\right\rbrace$ is tight in $\D(\R^+, \Cc^{-\nu}(\T^d))$.\\
Any weak limit is supported on $\Cc(\R^+,\Cc^{-\nu}(\T^d))$ and
\begin{equation}
\label{cap2tightextimation}
\sup_{\gamma \in (0,\gamma_0)} \E \sup_{0 \leq t \leq T} \Norm{\Cc^{-\nu}}{\Zg^{:\kb:}(t,\cdot)}^p < \infty\;.
\end{equation}
\end{proposition}
\begin{proof}
The proof is the same as the proof of \cite[Prop.~5.4]{MourratWeber}, and it is a consequence of Proposition~\ref{prop:cap2boundsonR}.
\end{proof}

We will now formalize the fact that the iterated integrals, introduced in Section~\ref{sec:linear}, are a convenient approximation of the Wick power of the solution to the linear model. 
The proof of the next theorems are essentially the same as in \cite{MourratWeber}. Nonetheless the next subsection contains the main arguments for the proofs, and can be skipped on a first reading.\\

Let $H_{\kb}$ be the generalized Hermite polynomial defined in Section~\ref{subsec:hermite}, and $[R_{\gamma,t}(\cdot,x)]_s = \left([R_{\gamma,t}^{(i)},R_{\gamma,t}^{(j)}(\cdot,x)]\right)_{i,j = 1}^m$ the optional quadratic variation matrix. If we define the error
\begin{equation}
\label{e:cap2Eproc}
E^{:\kb:}_{\gamma,t}(s,x) \eqdef  H_{\kb}\left((R_{\gamma,t}(s,x) , [R_{\gamma,t}(\cdot,x)]_s \right) - R_{\gamma,t}^{:\kb:}(s,x)\;,
\end{equation}

then we can prove the following version of \cite[Prop.~5.3]{MourratWeber}.

\begin{proposition}
\label{prop:cap2boundonHermiteerror}
For any multiindex $\kb\in \N^m$, $\kappa > 0$, $t > 0$ and $1 \leq p < \infty$ there exists $C=C(\kb,p,t,\kappa,\nn)$ such that for all $\gamma \in (0,\gamma_0)$
\begin{equation}
\label{e:cap2boundonHermiteerror}
\E \sup_{x \in \Le} \sup_{0 \leq s \leq t} |E_{\gamma,t}^{:\kb:}(s,x)|^p \leq C \gamma^{p(1-\kappa)}\;.
\end{equation}
\end{proposition}

The form of the error considered in \eqref{e:cap2Eproc} is somehow unsatisfactory because of the presence of the optional quadratic variation in \eqref{e:cap2Eproc}. It is possible to prove, however, that the quadratic variation can be approximated by a diagonal matrix. This will be the content of Propositions~\ref{prop:boundcrossvariation} and \ref{prop:hermiteRbound}.\\

The next lemma shows that the quadratic variation $\ang{R_{t,\gamma}(\cdot,x)}_s$ approximates the bracket process $[R_{t,\gamma}(\cdot,x)]_s$ as in \cite[Lemma~5.1]{MourratWeber}. Its proof is postponed to Subsection~\ref{subsec-tightnessproof}.
\begin{lemma}
\label{lemma:Ubound}
Let $x \in \Lambda_N$, $s \in [0,t]$, and define the ($m \times m$) martingale $U_{\gamma,t}(s,x)$ as
\begin{equation}
\label{e:cap2quadvardiff}
U_{\gamma,t}^{(i,j)}(s,x) \eqdef \left[ R^{(i)}_{t,\gamma}(\cdot,x),R^{(j)}_{t,\gamma}(\cdot,x)\right]_s - \ang{R^{(i)}_{t,\gamma}(\cdot,x),R^{(j)}_{t,\gamma}(\cdot,x)}_s
\end{equation}
for $1 \leq i,j \leq m$.\\
For all $n \in \N^+,\ t > 0,\ \kappa > 0$ and $ p \in [1,\infty]$, there exists a constant $C=C(t,\kappa,p,m)$ such that for $\gamma \in (0,\gamma_0)$
\begin{equation}
\label{e:cap2quadvardiffbound}
\E \sup_{x \in \Le} \sup_{0 \leq s \leq t} \left| U_{\gamma,t}(s,x) \right|_{m \times m}^p \leq C \gamma^{p(1-\kappa)}
\end{equation}
where $|\cdot|_{m \times m}$ is the norm in $\R^{m \times m}$ and $\gamma_0$ is the constant in \cite[Lemma~8.2]{MourratWeber}.
\end{lemma}

It is clear that the bound in \eqref{e:cap2quadvardiffbound} can be computed componentwise, and the unidimensional case is proven in \cite{MourratWeber}, where their only assumption used is the boundedness of the rate function of the jumps. It is sufficient in particular to prove it for the diagonal elements.

As a difference with the main reference \cite{MourratWeber}, we now propose a bound on the quadratic variation
\[
[ R^{(i)}_{\gamma,t}(\cdot,x), R^{(j)}_{\gamma,t}(\cdot,x)]_s
\]
that shows that each component of $ R_{\gamma,t}$ is asymptotically uncorrelated with the others. Lemma~\ref{lemma:Ubound} shows that it is sufficient to bound the predictable quadratic variation. 

We first define a new approximation of the diverging constant
\begin{multline}
\label{e:cgammatsdef}
\mathfrak{c}_{\gamma,t}(s) 
= 2\int_0^s \Norm{L^2(\Le)}{\Pg{t-r}}^2 dr \\
=   \frac{s}{2} +  \sum_{\substack{\omega \in \Z^2\\0 <|\omega|\leq \epsilon^{-1}}} 
	\frac{|\hat{\Kg}(\omega)|^2 
		e^{-2 |t-s|\epsilon^{-2}\gamma^2 (1- \hat{\Kg}(\omega))}}
		{4 \epsilon^{-2}\gamma^2(1 - \hat{\Kg}(\omega))} 
	\left(1 - e^{-2 s\epsilon^{-2}\gamma^2 (1- \hat{\Kg}(\omega))}\right)	\;.
\end{multline}

The next proposition, whose proof is postponed to the following subsection, is the key estimate behind Proposition~\ref{prop:hermiteRbound}.

\begin{proposition}
\label{prop:boundcrossvariation}
For $1 \leq i , j \leq m$, $b \in [0,1]$, $\gamma \in (0,\gamma_0)$ and $p \geq 1$ we have
\begin{equation}
\label{e:boundcrossvariation}
\begin{split}
&\E \left[ \sup_{x \in \Le}\sup_{0 \leq s \leq t} \left| \ang{ R^{(i)}_{\gamma,t}(\cdot,x), R^{(j)}_{\gamma,t}(\cdot,x)}_s - \mathfrak{c}_{\gamma,t}(s)\delta_{i,j}\right|^p \right]^{1/p} \\
&\leq C(\kappa,T,\nn,\nu,p) \gamma^{1- \nu - \kappa} + C(p,b,\kappa)   (1 \wedge t^{-b} \alpha^{b}) \gamma^{-\kappa}
\end{split}
\end{equation}
\end{proposition}
\begin{remark}
The main difference between Proposition~\ref{prop:boundcrossvariation} and \cite[Prop.~3.4]{ShenWeber} is that in the latter the bounds on the error gets worse as $p$ grows, while in \eqref{e:boundcrossvariation} the power of $\gamma$ in the right-hand-side doesn't depend on $p$. Such a result is more convenient in our case when the degree of the Wick polynomial is arbitrary large, since the errors containing the renormalization constants diverge as a arbitrarily large power of $\log(\gamma^{-1})$.
\end{remark}

With the same proof it is possible to show the following lemma.

\begin{lemma}
\label{lemma:boundforlinearconv}
For $0 \leq i,j \leq m$, $\kappa > 0$, $\nu > 0$, and $p > 1$, $0 \leq t \leq T$ and for $\phi \in \Cc\left( \T^2 \times [0,T] , \R^m \right)$ there exists a constant $C=C(\nn,\kappa,\nu,p,T)$
\begin{multline}
\E \left[\left| \int_0^t \sum_{z \in \Le} \phi^{(i)}(z,s)\phi^{(j)}(z,s) \Qm^{i,j}(s^-,z) ds -   2 \delta_{i,j}  \int_0^t\ang{\phi^{(i)}(\cdot,s),\phi^{(j)}(\cdot,s)}_{\T^2}ds \right|^p \right] \\
\leq C\gamma^{1-\kappa - \nu} \int_0^t \Norm{L^p(\T^2)}{\phi^{(i)}(\cdot,s) \phi^{(j)}(\cdot,s) }^p ds +   C \alpha \E[\Norm{L^2(\Le)}{\sigma(0)}^2] \Norm{L^{\infty}(\T^2 \times [0,T])}{\phi^{(i)}\phi^{(j)}}^p\:.
\label{e:eqHermitecov}
\end{multline}
\end{lemma}

We conclude the section with a proposition that simplifies the expressions for the Hermite polynomial approximating the iterated integrals $\Rg$, with the replacement of the covariance of the process $\Rg$, with its limiting value. We recall the definition of the constant introduced in \eqref{e:cgammadef} and we define the values of $\CG,\CG(t)$, for future reference

\begin{align}
\label{e:cgammatdef}
\CG(t) &= \frac{t}{2} + \sum_{\substack{\omega \in \Z^2\\0 <|\omega|\leq \epsilon^{-1}}} \frac{|\hat{\Kg}(\omega)|^2}{4 \epsilon^{-2}\gamma^2(1 - \hat{\Kg}(\omega))} \left(1 - e^{-2 t\epsilon^{-2}\gamma^2 (1- \hat{\Kg}(\omega))}\right) \\
\label{e:cgammadef}
\CG &=   \sum_{\substack{\omega \in \Z^2\\0 < |\omega|\leq \epsilon^{-1}}} \frac{|\hat{\Kg}(\omega)|^2}{4 \epsilon^{-2}\gamma^2(1 - \hat{\Kg}(\omega))}
\end{align}

Together with Proposition~\ref{prop:cap2boundonHermiteerror}, we have the main result of the section:

\begin{proposition}
\label{prop:hermiteRbound}
For a multiindex $\kb \in \N^m$, $\kappa>0$, $\nu > 0, p \geq 1$ and $b \in [0,1]$, under the assumption \eqref{e:hp2bis}
\begin{equation}
\label{e:Hermitepolywithc}
\begin{split}
&\E \left[ \sup_{\substack{z \in \Le}} \sup_{0 \leq s \leq t} \left| H_{\kb}(\Rg(s,z),\mathfrak{c}_{\gamma,t}(s)I_m) - \Rg^{:\kb:}(s,z)\right|^p\right]^{\frac{1}{p}} \\
&\leq C(\kappa,T,\nn,\nu,\kb,b) \left(\gamma^{(1- \nu - \kappa)} + \gamma^{-\kappa }  (1 \wedge t^{-b } \alpha^{b })\right)
\end{split}\;.
\end{equation}
\end{proposition}
The proof of the above proposition is given at the end of Subsection~\ref{subsec-tightnessproof}.

\subsection{Convergence of the linearized dynamic}
\label{subsec:convergence-linear}

Recall the definition of $\Zg$ given in Subsection~\ref{subsec:limitingspde} and the tightness of their laws proved in Proposition~\ref{prop:Ztightness}. For $\nu > 0$, we assume in this subsection, that the limit $\gamma \to 0$ is taken along to a fixed converging subsequence of $\Zg$.\\
In this section we show that any limit law solves a martingale problem, more precisely we will use the fact, that the law of the stochastic heat equation is the only solution of a martingale problem (see \cite[App.~C]{MourratWeber}).\\
The next result has been proven in \cite[Theorem~6.1]{MourratWeber}, the extension to vector-valued processes being straightforward.
\begin{theorem}
\label{cap2linearconvergence}
Let $\nu >0$, The law of the processes $\Zg$ as $\gamma \to 0$, converge to the law of $Z$, the solution of the multivariate stochastic heat equation 
\begin{equation}
\label{Cap2Heateq}
\begin{cases}
\partial_t Z_t  &= \Delta Z_t + \sqrt{2} d W_t\\
Z_0 &\equiv 0
\end{cases}
\end{equation}
in the topology of $\D\left( [0,T]; \Cc^{-\nu}(\T^2,\R^m) \right)$.\\
Here $W$ is a $n$-component noise $(W^{(1)},\dots,W^{(n)})$ and each of the components is an independent space time white noise on $L^2([0,T] \times \T^d)$.
\end{theorem}

The proof is identical to \cite[Theorem~6.1]{MourratWeber}. The only new part is the estimation of the quadratic covariation via Lemma~\ref{lemma:boundforlinearconv} and the assumption \eqref{e:hp2bis}.

We are now ready to state the main result of the section. 

\begin{theorem}
\label{thm:Zsconvergence}
For any $\nn$ and $k \geq 1$, the processes $(\Zg^{\kb})_{|\kb| \leq k}$ converge jointly in law to $(Z^{\kb})_{|\kb| \leq k}$ in the topology of $\mathcal{D}(\R^+,(\Cc^{-\nu})^K)$ where $K = \binom{k + m -1}{m-1}$ 
\end{theorem}

The proof of the above theorem is essentially the same as the proof of theorem 6.2 in \cite{MourratWeber} or proposition 4.5 in \cite{ShenWeber}, and it is based on the approximation $\Rg$ of $\Zg$ and Proposition~\ref{prop:hermiteRbound}.

\subsection{Proofs of the statements}
\label{subsec-tightnessproof}
The aim of this section is to show that the iterated integrals of the process $\Rg(s,x)$ are a good approximation for the Hermite polynomial.

\begin{lemma}
\label{lemma:Rjumpsbound}
Recall that the Glauber dynamic is stopped as prescribed in Subsection~\ref{subsec:stopping time}. For any $p \geq 1$, $x \in \Le$ and $\kappa > 0$
\begin{equation}
\label{e:Rjumpbound}
\E \left[ \sup_{0 \leq r \leq t} \sup_{x \in \Le}|\Delta_r R^{(j)}_{\gamma,t}(\cdot,x)|^{p}\right]^{1/p} \leq C(p,\kappa,t) \gamma^{1 - \kappa}\;.
\end{equation}
\end{lemma}

\begin{proof}
By monotonicity of $L^p$ norms it is sufficient to prove the bound for high values of $p$. If at microscopic time $r$, a jump happens at macroscopic site $y \in \Le$, the size of the jump of $\Rg(r,x)$
\begin{equation}
|\Delta_r \Rg(r,x)| = \delta^{-1} \epsilon^2  |\Pg{t-r} \Kg(y-x)||\sigma_{r}(\epsilon^{-1} y) - \sigma_{r^-}(\epsilon^{-1} y)|\;.
\end{equation}
From the form of $\nug$ in Subsection~\ref{subsec:stopping time}, 
\[
\E|\sigma_{r}(\epsilon^{-1} y) - \sigma_{r^-}(\epsilon^{-1} y)|^p \leq  C(p)(\E|\sigma_{r}(\epsilon^{-1} y)|^p + \E|\sigma_{r^-}(\epsilon^{-1} y)|^p)< 2 C(p,\nn)
\]
bounded uniformly in the dynamic. Using the fact that, on $\Le$, $\Norm{L^1(\Le)}{\Pg{t}} = 1$ and $\sup_{x \in \Le} \Kg(x) \lesssim \epsilon^{-2}\gamma^2$ one has the following
\begin{gather*}
\E\Big[ \sum_{\substack{(r,x) \in [0,t]\times \Le\\ \text{jump at }(r,x)}}|\Delta_r \Rg(r,x)|^p\Big] \\
\leq \E\Big[ \sum_{i=0}^{N}\E\Big(|\Delta_{r_i} \Rg(r_i,x_i)|^p\Big| \text{jumps at }(r_i,x_i): i = 1\dots N \Big)\Big] \\
\leq C(p,\nn)\delta^{-p}\gamma^{2p} \E \left[ \text{\# of jumps in }[0,t]\right] \leq C(p,\nn,t)\gamma^{p}\epsilon^{-2}\alpha^{-1}\;.
\end{gather*}
And the proof is complete taking $p$ large enough. To go from the first line to the second we used the fact that the rate of the Poisson processes controlling the jumps is a constant, hence it is not dependent from the process.
\end{proof}

\begin{proof}[of Lemma~\ref{lemma:Ubound}] We will apply the Burkholder-Davis-Gundy inequality to $U_{\gamma,t}^{(i,i)}(s,x)$, defined in \eqref{e:cap2quadvardiff}. The bracket process of $s \mapsto \Rg(s,x)$ is given by
\begin{equation}
s \mapsto \sum_{r \leq s} (\Delta_r \Rg(r,x))^2
\end{equation}
and the jumps of $U_{\gamma,t}^{(i,i)}(s,x)$ are given by the jumps of $\Rg(r,x)$. Using Lemma~\ref{lemma:Rjumpsbound} 
\[
\begin{split}
\E\left[ \sup_{0 \leq r \leq t} |\Delta_r U_{\gamma,t}^{(i,i)}(s,x)|^p\right] 	\lesssim \E\left[ \sup_{0 \leq r \leq t} |\Delta_r \Rg(s,x)|^{2p}\right] \leq C(t,\nn,\kappa)\gamma^{2p(1  - \kappa)}\;.
\end{split}
\]
It remains to control the quadratic variation $\ang{U_{\gamma,t}^{(i,i)}(\cdot,x)}_s$. We can write it as
\begin{equation}
\label{e:Uquadvar}
 \sum_{z \in\Le}\int_{r=0}^s \left( \epsilon^2 \delta^{-1	}\Pg{t-r}\Kg(x-z) \right)^4 d \ang{(\sigma_z^{(i)}(\alpha^{-1} r) - \sigma_z^{(i)}(\alpha^{-1} r^-) )^2\mathscr{I}^{r,z} -  \alpha^{-1}\Qm^{i,i}(r^-,z)}_r
\end{equation}
where $\mathscr{I}^{r,z}$ is the Poisson process of rate $\alpha^{-1}$, that is responsible for the jumps.\\
The quantity in the angled brackets in \eqref{e:Uquadvar}, is bounded by 
\[
\alpha^{-1} C \int_S |\eta_z(\alpha^{-1} r) - \sigma_z(\alpha^{-1} r) |^4 \pgm{z,  r^- }(d\eta) \leq C(\nn) \alpha^{-1}\left(1 + |\sigma_z(\alpha^{-1} r^-)|^{4}\right)\;.
\]
Using the general H\"older inequality and Remark \ref{prop:boundonrate}
\[
\E[\prod_{j=1}^{p/2} 1 + |\sigma_{z_j}(\alpha^{-1} r_j)|^4] \leq \prod_{j=1}^{p/2} \left(\E[1 + |\sigma_{z_j}(\alpha^{-1} r_j)|^4]^{p/2}\right)^{2/p} \leq C(\nn,p)
\]
we find
\[
\E\left[\ang{U_{\gamma,t}^{(i,i)}(\cdot,x)}_s^{p/2}\right] \leq C(\nn,p) \left( \alpha^{-1}   \epsilon^8 \delta^{-4} \sum_{z \in\Le}\int_{r=0}^s \left(\Pg{t-r}\Kg(x-z) \right)^4 dr\right)^{p/2}
\]
and the conclusion follows from the fact  $|\Pg{t-s} \Kg(x)| \leq \epsilon^2\gamma^{-2}$ for $x \in \Le$ and Lemma~\ref{lemma:est-cap2kernel3}. By scaling \eqref{e:scaling}, $\alpha \sim \epsilon^2\gamma^{-2}$ and
\[
\E\ang{U_{\gamma,t}^{(i,i)}(\cdot,x)}_s^{p/2} \lesssim \left( \gamma^{	2} \sum_{z \in\Le}  \epsilon^2 \int_{r=0}^s \left(\Pg{t-r}\Kg(x-z) \right)^2 dr \right)^{p/2} \lesssim \gamma^{p(1-\kappa)}
\]
where the constants depends on $\nn,p,\kappa$ and the proof is completed.
\end{proof}

The next lemma correspond to \cite[Lemma~5.2]{MourratWeber}.
\begin{lemma}
\label{lemma:Rjumpsbound2}
For $j = 1,\dots,m$, any $t \geq 0$ and $1 \leq p < \infty$ and $\gamma$ small enough
\begin{equation}
\E \left[ \left|\sup_{x \in \Le }\sum_{r \leq t}  |\Delta_r R_{\gamma,t}^{(j)}(r,x)|^{2} \right|^p \right]^{1/p} \leq C(t,p) \log(\gamma^{-1})\;.
\end{equation}
\end{lemma}

We have the following proposition, which correspond to \cite[Prop.~5.3]{MourratWeber}.

\begin{proof}[of Proposition~\ref{prop:cap2boundonHermiteerror}]
The proof uses the generalized multidimensional It\^o formula for processes with finite first variation, that can be found in \cite[Chapter~II]{Protter1990}. Let $X_t = (X_{1,t},\dots,X_{n,t})$ a multidimensional process with finite first variation and let $[X_i,X_j]_t$ be its bracket process (find citation). Then
\begin{equation}
\begin{split}
f(X_t) &= f(X_0) +\sum_j \int_0^t \partial_j f(X_{s^-}) dX_{j,s} + \frac{1}{2} \sum_{i,j} \int_0^t \partial_{i}\partial_j f(X_{s^-}) d[X_i,X_j]_s \\
&+ \sum_{s \leq t} \left( \Delta f(X_s) - \sum_j \partial_j f(X_{s^-}) \Delta X_{j,s} - \sum_{i,j}\frac{\partial_j \partial_i f(X_{s^-})}{2} \Delta X_{i,s} \Delta X_{j,s}\right)\;.
\end{split}
\end{equation}
The key step in the proof uses the It\^o formula to prove  \eqref{e:cap2quadvardiffbound}  by induction. Indeed for $\kb= (0,\dots,0)$, the error \eqref{e:cap2Eproc} is zero and \eqref{e:cap2quadvardiffbound} is trivially true. Recall the definitions of the Hermite polynomials $H_{\kb}$ and its derivatives in Section~\ref{subsec:hermite} Using the It\^o formula on $H_{\kb}(\underline{R}_s) = H_{\kb}(R_{\gamma,t}(s,x),[R_{\gamma,t}(\cdot,x)]_s)$ 
\begin{align*}
H_{\kb}(\underline{R}_s) &= \sum_{i,j=1}^m \int_{r = 0}^s \partial_{T_{i,j}}H_{\kb}(\underline{R}_{r^-}) d[ R^{(i)}_{\gamma,t}(\cdot,x), R^{(j)}_{\gamma,t}(\cdot,x)]_r \\
 &+ \sum_{j=1}^m \int_{r = 0}^s \partial_{X_j}H_{\kb}(\underline{R}_{r^-}) d R^{(j)}_{\gamma,t}(r,x) \\
&+ \frac{1}{2}\sum_{i,j = 1}^m  \int_{r = 0}^s \partial_{X_j} \partial_{X_i} H_{\kb}(\underline{R}_{r^-}) d[ R^{(i)}_{\gamma,t}(\cdot,x), R^{(j)}_{\gamma,t}(\cdot,x)]_r + Err_{\kb}(s,x)
\end{align*}
where $Err_{\kb}(s,x)$ contains the jumps.
\begin{multline*}
Err_{\kb}(s,x) = \sum_{r \leq s} \left( \Delta H_{\kb}(\underline{R}_r) - \sum_{j=1}^m \partial_{X_j}H_{\kb}(\underline{R}_{r^-}) \Delta_r R^{(j)}_{\gamma,t}(\cdot,x) \right. \\
-\sum_{i,j=1}^m \partial_{T_{i,j}} H_{\kb}(\underline{R}_{r^-}) \Delta_r [ R^{(i)}_{\gamma,t}(\cdot,x), R^{(j)}_{\gamma,t}(\cdot,x)]_r \\
\left. - \frac{1}{2} \sum_{i,j=1}^m \partial_{X_j} \partial_{X_i} H_{\kb}(\underline{R}_{r^-}) \Delta_r R^{(j)}_{\gamma,t}(\cdot,x) \Delta_r R^{(i)}_{\gamma,t}(\cdot,x) \right)\;.
\end{multline*}
Using the properties of Hermite polynomials $\left(\frac{1}{2}\partial_{X_j} \partial_{X_i}  + \partial_{T_{i,j}}\right) H_{\kb}(\bar{X},\bar{T}) = 0$ the above can be rewritten as
\[
H_{\kb}(\underline{R}_s) = \sum_{j=1}^m k_j \int_{r = 0}^s H_{\kb^{j-}}(\underline{R}_{r^-}) d R^{(j)}_{\gamma,t}(r,x) +  Err_{\kb}(s,x)
\]
which has the same form as \eqref{e:Rpowersdef}. Subtracting the quantity $\Rg^{:\kb:}$ we thus obtain
\begin{equation}
\label{e:cap2equationerroquadvar}
E^{:\kb:}_{\gamma,t}(s,x)  = \sum_{j=1}^m k_j \int_{r = 0}^s E^{:\kb^{j-}:}_{\gamma,t}(r^-,x)  d R^{(j)}_{\gamma,t}(r,x) + Err_{\kb}(s,x)\;.
\end{equation}
We will use the induction over $|\kb|$ to prove \eqref{e:cap2boundonHermiteerror}. Clearly \eqref{e:cap2boundonHermiteerror} holds for every $\kb$ with $|\kb|=1$. Assume that \eqref{e:cap2boundonHermiteerror} holds for every multiindex $\mathbf{0} \leq \mathbf{a} < \kb$. We shall show that the conclusion of the proposition also holds for $\kb$.\\
The first step consists in applying the Burkholder-Davis-Gundy inequality to the integral in \eqref{e:cap2equationerroquadvar}. The quadratic variation is given by
\begin{gather*}
\ang{\int_{r = 0}^{\cdot} E^{:\kb^{j-}:}_{\gamma,t}(r^-,x)  d R^{(j)}_{\gamma,t}(r,x)}_s  \leq C \int_0^s \left| E^{:\kb^{j-}:}_{\gamma,t}(r^-,x) \right|^2 d \ang{R^{(j)}_{\gamma,t}(\cdot,x)}_{r} \\
\leq C  \sup_{0 \leq r \leq t} \left| E^{:\kb^{j-}:}_{\gamma,t}(r^-,x) \right|^2 \ang{R^{(j)}_{\gamma,t}(\cdot,x)}_{s} 
\end{gather*}
and using the Cauchy-Schwarz inequality, the expectation of the $\frac{p}{2}$-th power of the quantity above is bounded by
\begin{gather*}
\E\left[  \sup_{0 \leq r \leq t} \left| E^{:\kb^{j-}:}_{\gamma,t}(r^-,x) \right|^{2p } \right]^{\frac{1}{2}} \E\left[\ang{R^{(j)}_{\gamma,t}(\cdot,x)}_{s}^{p} \right]^{1/2}  \\
\leq C(\kappa,t,p,\nn) \gamma^{p(1-\kappa/2)} \E\left[\ang{R^{(j)}_{\gamma,t}(\cdot,x)}_{s}^{p} \right]^{1/2}  \leq C(\kappa,t,p,\nn) \gamma^{p(1-\kappa/2)} \gamma^{-p\kappa/2}\;,
\end{gather*}
where we used induction and \eqref{e:Rquadexpbound}. We bound the jump term in a similar way, using Lemma~\ref{lemma:Rjumpsbound}
\begin{gather*}
\E \left[ \sup_{0 \leq s \leq t} |\Delta_{s}E^{:\kb:}_{\gamma,t}(s,x) |^p\right] \leq \E \left[ \sup_{0 \leq r \leq t} |E^{:\kb:^{j-}}_{\gamma,t}(r,x) |^p |\Delta_r R^{(j)}_{\gamma,t}(\cdot,x)|^{p}\right] \\
\leq \E \left[ \sup_{0 \leq r \leq t} |E^{:\kb:^{j-}}_{\gamma,t}(r,x) |^{2p}\right]^{\frac{1}{2}}\E \left[ \sup_{0 \leq r \leq t} |\Delta_r R^{(j)}_{\gamma,t}(\cdot,x)|^{2p}\right]^{\frac{1}{2}} \leq  C(\kappa,t,p,\nn) \gamma^{p(1-\kappa)}\;.
\end{gather*}
It remains to bound the error $Err_{\kb}(s,x)$, that contains the errors from the application of the It\^o formula for processes with jumps, Taylor expanding up to second order the Hermite polynomials. For $\bar{x}=(x_1,\dots,x_m), t = (t_{i,j})_{i,j=1}^m$
\begin{gather*}
\Big| H_{\kb}(\bar{x} + \bar{y},\bar{t} + \bar{r}) - H_{\kb}(\bar{x},\bar{t}) \\
- \sum_{j=1}^n \partial_{X_j} H_{\kb}(\bar{x},\bar{t}) y_j - \frac{1}{2} \sum_{i,j=1}^n \partial_{X_j} \partial_{X_i} H_{\kb}(\bar{x},\bar{t}) y_j y_i  - \sum_{i,j=1}^n \partial_{T_{i,j}} H_{\kb}(\bar{x},\bar{t}) r_{i,j} \Big| \\
\leq C \left( \sum_{\mathbf{a}: |\mathbf{a}| = |\mathbf{k}| - 2}|\bar{x}|^{\mathbf{a}} + |\bar{t}|^{\mathbf{a}} + 1\right) (\sum_{\mathbf{b}:|\mathbf{b}|=3} |\bar{y}|^{\mathbf{b}} + \sum_{\mathbf{b}:|\mathbf{b}|=2} |\bar{r}|^{\mathbf{b}})
\end{gather*}
hence
\begin{gather*}
|Err_{\kb}(s,x)| \leq C \sum_{j=1}^m\left( \sup_{r \leq s} |R_{\gamma,t}^{(j)}(r,x)|^{|\kb|-2} + \sup_{r \leq s} [R^{(j)}_{\gamma,t}(\cdot,x)]_r^{(|\kb|-2)/2}  + 1 \right) \times\\
\times\sum_{r \leq s} \left( |\Delta_r R_{\gamma,t}^{(j)}(r,x)|^{3} + |\Delta_r [R_{\gamma,t}^{(j)}(\cdot,x)]_r|^{2}\right)
\end{gather*}
and using Lemma~\ref{lemma:Rjumpsbound} and Lemma~\ref{lemma:Rjumpsbound2}, and H\"older inequality for $q_1^{-1} + q_2^{-1} + q_3^{-1}= p^{-1}$
\begin{multline*}
\E \left[ \sup_{x \in \Le, s \in [0,t]}|Err_{\kb}(s,x)|^p \right]^{1/p} \leq  \E \left[ \left(\sup_{x \in \Le, r \leq s} |R_{\gamma,t}(r,x)|_{\R^n}^{|\kb| - 2} + 1 \right)^{q_1} \right]^{1/q_1} \times \\
\times\E\left[|\Delta_r R_{\gamma,t}(r,x)|^{q_2} \right]^{1/q_2}\E \left[ \sup_{x \in \Le }\left(\sum_{r \leq s} \norm{\R}{\Delta_r R_{\gamma,t}(r,x)}^{2} \right)^{q_3} \right]^{1/q_3}  \\
\leq C(\kappa,\nn,t,p) \gamma^{	1-\kappa} E \left[ \sup_{x \in \Le, r \leq s} |R_{\gamma,t}(r,x)|_{\R^n}^{q_1(|\kb| - 2)} + 1  \right]^{1/q_1} \\ \leq C(\kappa,\nn,t,p) \gamma^{1-\kappa}\gamma^{-\kappa}
\end{multline*}
where in the last line we used Remark~\ref{remark:puntualbound} and the induction is proven.
\end{proof}


\begin{proof}[of Proposition~\ref{prop:boundcrossvariation}]

We will prove the above theorem for $p$ large, the theorem for all $p > 1$ will follows from the monotonicity of $L^p$ norms. We start computing
\begin{multline}
\ang{R^{(i)}_{\gamma,t}(\cdot,x), R^{(j)}_{\gamma,t}(\cdot,x)}_s - \mathfrak{c}_{\gamma,t}(s)\delta_{i,j} \\
= \int_0^s \sum_{z_1,z_2 \in \Le} \epsilon^{4}\Pg{t-r}(x-z_1) \Pg{t-r}(x-z_2) \left( d \ang{\Mg^{(i)}(\cdot,z_1),\Mg^{(j)}(\cdot,z_2)}_r -2\delta_{i,j} dr\right)\\
= \int_0^s \sum_{z \in \Le} \epsilon^2 |\Pg{t-r}\Kg(x-z)|^2 \left( \Qm^{i,j}(r^-,z) - 2\delta_{i,j}\right) dr\;.
\label{e:covarfinalbounds1}
\end{multline}
The proof consists in evaluating the difference between $\Qm^{i,j}(r^-,z)$ and $2\delta_{i,j}$. In order to prove that the average is negligible in the limit we will need to exploit the time integral. From Proposition~\ref{prop:boundonrate}, and the form of the stopping time $\taun$ follows that we can prove the proposition for
\begin{equation}
\label{e:quadvarmarting}
\int_0^s \sum_{z \in \Le} \epsilon^2 |\Pg{t-r}\Kg(x-z)|^2 \left( \sigma_{\epsilon^{-1} z}^{(i)}(\alpha^{-1} s)\sigma_{\epsilon^{-1} z}^{(j)}(\alpha^{-1} s) - \delta_{i,j}\right) dr\,.
\end{equation}
Following \cite{ShenWeber}, we produce a coupling with the dynamic at infinite temperature $\beta = 0$.
Let
\[
Z_h = \int_S e^{\beta\ang{h,\eta}} \nug(d\eta) \qquad P^{h} = \int_S Z_h^{-1}e^{\beta \ang{h,\eta}} \wedge 1\ \nug(d\eta) 
\]
In particular 
\[
0 \leq 1 - P^{h} \leq 2\beta |h|\int_S |\eta| e^{\beta|h||\eta|} \nug(d\eta)\;.
\]
Let $ \tilde{\sigma}_{x}(t)$ be a process on $S^{\Lambda_N}\times \R^+$, starting from the configuration with all spins equal to $0$ following the Glauber dynamic with parameter $\beta = 0$. Recall the construction in Section~\ref{sec:model-theorems} together with stopping time in Subsection~\ref{subsec:stopping time}. We will now define the coupling between $\tilde{\sigma}_{x}(t)$ and $\sigma_{x}(t)$ as follows: since the Poisson times between each jumps have been chosen to be independent of the configuration and with constant mean, we can construct $\tilde{\sigma}_{x}(t)$ in such a way that it has jumps at the same time and at the same place as the original process. Assume a jump happens at $(x,t)$. If $t > \taun$, we chose $\sigma_x(t)=\tilde{\sigma}_x(t)$ since both are chosen according to $\nug$. If $t \leq \taun$ with probability $P^{\hg(x,\sigma_{t^-})}$ we choose $\sigma_x(t) = \tilde{\sigma}_x(t)$ distributed according to the density (here $\hg(x,t) = \hg(x,\sigma_{t^-})$)
\[
(P^{\hg(x,t)})^{-1} \left[ Z_{\hg(x,t)}^{-1}e^{\beta \ang{\hg(x,t),\eta}} \wedge 1 \right]\nug(d\eta) 
\]
and with probability $1-P^{\hg(x,t)}$ we will draw $\tilde{\sigma}_x(t)$ and $\sigma_x(t)$ independently with density, respectively proportional to
\[
 \left[ 1 - Z_{\hg(x,t)}^{-1}e^{\beta \ang{\hg(x,t),\eta}} \right]^+\nug(d\eta)\qquad\text{ and }\qquad  \left[ Z_{\hg(x,t)}^{-1}e^{\beta \ang{\hg(x,t),\eta}} - 1\right]^+\nug(d\eta)
\]
Thus for any function $f:S \to \R$, for $x \in \Lambda_N, t \in \R^+$ and $p \geq 1$
\[
 \left| f(\sigma_x(t)) - f(\tilde{\sigma}_x(t)) \right| \leq   \left| f(\sigma_x(t)) - f(\tilde{\sigma}_x(t)) \right|  +    1_{\{t \leq T_0\}}\left| f(\sigma_x(0))\right|
\]
where $T_0$ denotes the time of the first jump. For the inequality we used the fact that $\Norm{L^{\infty}}{\hg(\cdot,t)}\leq \gamma^{1-\nu} \nn$ for $t \leq \taun$ and the fact that $\nug$ has exponential moments.
\begin{multline*}
\sum_{x \in\Le} \E\left(\int_0^s \sum_{z \in \Le} \epsilon^2 |\Pg{t-r}\Kg(x-z)|^2 \left| f(\sigma_{\epsilon^{-1}z}(\alpha^{-1} r)) - f(\tilde{\sigma}_{\epsilon^{-1}z}(\alpha^{-1} r)) \right| dr\right)^p \\
\lesssim \epsilon^{-2}\left(\gamma^{ 1-\nu } + \E\left[ \alpha T_0 t^{-1} \wedge 1 \right]\right)^p \log^p(\gamma^{-1})\;.
\end{multline*}
The last expectation is estimated with $x\wedge 1 \leq x^{b}$ for any $b \in [0,1]$. This implies that it is sufficient to prove the proposition in the infinite temperature case, starting from the zero initial condition.
Let ${\tau_l(x)}_{x \in \Lambda_N, l \in \N}$ denote the collection of random times where $\tau_l(z)$ is the time at which the spin at site $x$ jumps for the $l$-th time, in macroscopic coordinate. When a jump occurs, the distribution of the new spin is drawn independently from the other, according to $\nug$. Let $M_s$ the quantity in \eqref{e:quadvarmarting}, calculated with the auxiliary process $\tilde{\sigma}$. We bound the supremum of \eqref{e:boundcrossvariation} in time with the supremum over a discretization of $[0,T]$ of mesh size $\gamma^R$ where $R$ is chosen later. The difference $|M_s - M_{\gamma^R\lfloor \gamma^{-R}s\rfloor} |$ is bounded by 
\[
2 \int_{\gamma^R\lfloor \gamma^{-R}s\rfloor}^{s} \Norm{L^2(\Lambda_N)}{\Pg{t-r}\Kg}^2 dr  \Norm{L^{\infty}(\Lambda_N \times [0,T])}{\tilde{\sigma}}^2 \leq \gamma^R \epsilon^{-2}\gamma^2	\Norm{L^{\infty}(\Lambda_N \times [0,T])}{\tilde{\sigma}}^2\;.
\]
Using
\[
\E\Big[\sup_{z \in \Lambda_N ;\ s \in [0,T])}{|\tilde{\sigma}_{z}(\alpha^{-1}s)|}^{2p} \Big] \lesssim \E\left[\text{\# of jumps in [0,T]}\right] = \epsilon^{-2} \alpha^{-1} T
\]
we deduce that $R$ has to satisfy $\gamma^{R}\epsilon^{-2}\gamma^2 \lesssim \gamma$. Bounding $\E[\sup_{s \in \gamma^R\Z \cap [0,t]} |M_s|^p]$ with $\E[\sum_{s \in \gamma^R\Z \cap [0,t]} |M_s|^p]$ it remains to estimate $\E[|M_s|^p]$. Let us expand the product
\[
\sum_{z_1,\dots,z_p \in \Lambda_N}\epsilon^{2p}\sum_{l_1,\dots,l_p \geq 1} \E\prod_{v=1}^p\left( \int_{\tau_{l_v}(z_v) \wedge s}^{\tau_{l_v+1}(z_v) \wedge s} |\Pg{t-r}\Kg(x-z)|^2 [\tilde{\sigma}^{(i)}\tilde{\sigma}^{(j)}_{z_v}(\alpha^{-1}\tau_{l_v})-\delta_{i,j}]dr\right)
\]
and notice that for different $z_v$ or $l_v$, the quantity inside the integrals are independent and with mean zero. We can thus perform the summation indexed over the possible partition of $\{1,\dots,p\}$ that don't contain singletons. Let $p$ be even and denote with $\mathcal{P}^*$ the set of such partitions, let $(q_1,\dots,q_m)$ be the sizes of the sets of a given partition $Q \in \mathcal{P}^*$ with $q_1+\dots + q_m = p$
\begin{multline*}
\epsilon^{2(p - m )}\sum_{z_1,\dots,z_{m} \in \Lambda_N}\epsilon^{2m}\sum_{l_1,\dots,l_{m} \geq 1} \prod_{v=1}^{m} \E\left[\left( \int_{\tau_{l_v}(z_v) \wedge s}^{\tau_{l_v+1}(z_v) \wedge s} |\Pg{t-r}\Kg(x-z)|^2 dr\right)^{ q_v}\right]\\
\leq\epsilon^{2(p - m )}\left(\int_0^s \Norm{L^2(\Le)}{\Pg{t-r}\Kg}^2 dr \right)^m \Norm{L^{\infty}(\Le)}{\Pg{t-r}\Kg}^{2(p-m)} \E\Big[\sup_{z, l}|\tau_l(z) -\tau_{l+1}(z) |^{p-m}\Big]\\
\lesssim \epsilon^{2(p-m)} \log^m(\gamma^{-1}) (\epsilon^{-2}\gamma^2)^{2(p-m)} (\epsilon^{-2}\alpha^{-1}T)\alpha^{p-m} \\
\lesssim \log^{m}(\gamma^{-1}) \gamma^{2(p-m)} \epsilon^{-2}\alpha^{-1}
\end{multline*}
where the supremum runs over $z \in \Lambda_N$ and $l \geq 1$ such that $\tau_l(z)\leq s$. Here in the second inequality we used lemma \ref{lemma:est-cap2kernel3}. The fact that $m \leq p/2$ and $p$ is large proves the proposition.
\end{proof}
\begin{proof}[of Proposition~\ref{prop:hermiteRbound}]
In virtue of Lemma~\ref{lemma:Ubound} and bound \eqref{e:cap2boundonHermiteerror}, it is sufficient to show the above inequality for the difference
\begin{align*}
&H_{\kb}(\Rg(s,z),\mathfrak{c}_{\gamma,t}(s)I_m) -H_{\kb}(\Rg(s,z),[\Rg(\cdot,z),\Rg(\cdot,z)]_s) \;.
\end{align*}
It is easy to see that the above difference can be written as a polynomial in the entries of the matrix $\mathfrak{c}_{\gamma,t}(s)I_m - [\Rg(\cdot,z),\Rg(\cdot,z)]_s$. The coefficient of the polynomial are of the form $\partial_{T_{i_1,i_2}}\partial_{T_{i_{m-1},i_m}} H_{\kb}(\Rg(s,z),[\Rg(\cdot,z),\Rg(\cdot,z)]_s)$.\\
Using the recursion formula for the Hermite polynomials in Subsection~\ref{subsec:hermite}, we can bound
the left-hand-side of \eqref{e:Hermitepolywithc} with
\begin{gather*}
\leq C(\kb,p) \sup_{\substack{\mathbf{0} \leq \mathbf{a}\leq \kb \\ 1 \leq i,j \leq m}} \E \left[ \sup_{x \in \Le}\sup_{0 \leq s \leq t} \left| [ R^{(i)}_{\gamma,t}(\cdot,x), R^{(j)}_{\gamma,t}(\cdot,x)]_s - \mathfrak{c}_{\gamma,t}(s)\delta_{i,j}\right|^{2p\left\lfloor\frac{|\kb - \mathbf{a}|}{2}\right\rfloor} \right]^{1/2p}\times\\
\times\E\left[\sup_{x \in \Le}\sup_{0 \leq s \leq t} \left| H_{\mathbf{a}}(\Rg(s,z),[\Rg(\cdot,z),\Rg(\cdot,z)]_s)  \right|^{2p} \right]^{1/2p} \\
\leq C(\kappa,T,p,\nn,\nu,\kb,b)  \gamma^{-\kappa }\left(\gamma^{(1- \nu - \kappa)} +  (1 \wedge t^{-b } \alpha^{b })  \left(\E\sup_x|\sigma_{\epsilon^{-1}x}(0)|^{2p |\kb|}\right)^{1/2p}\right)
\end{gather*}
where in the last line we used Propositions~\ref{prop:boundcrossvariation},~\ref{prop:cap2boundonHermiteerror}~and~\ref{prop:cap2boundsonR}  with the observation \ref{lemma:infinitybesovbound}.\\
The proof then follows from assumption \eqref{e:hp2bis} and \eqref{e:hpbound} for a suitably large power.

\end{proof}

\section{The nonlinear process}
\label{sec:nonlinear}
In this section we prove Theorem~\ref{thm:maintheorem} controlling the nonlinear dynamic. In Section~\ref{sec:tightness-linear}, we showed that the process $\Zg$, obtained from the dynamic stopped at random time $\taun$, as described in Subsection~\ref{subsec:beta-nug}, is convergent in law to the vector-valued stochastic heat equation. The random time guarantees a control over the $\Cc^{-\nu}$ norm of the process for a given $\nu > 0$, that we will assume to be fixed for this section. Following the strategy outlined in Subsection~\ref{subsec:limitingspde}, we use the linear dynamic to control the nonlinear one.\\
Recall from \eqref{e:Xequation1} in Section~\ref{sec:model-theorems} that the nonlinear process $\Xg$, started from $\Xng$ satisfies
\begin{equation}
\label{e:cap6equation}
 \Xg(z,t) = \Pg{t} \Xng(z) + \int_0^t \Pg{t-s} \Kg \aste \left( \pptg(\Xg(z,s)) + \Eg(z,s)\right) ds + \Zg(z,t)
\end{equation}
for $z \in \Le, t \in [0,T]$, and where the polynomial $\pptg$ is defined in \eqref{e:polydef}.\\

\subsection{Renormalization of the polynomial}

At some point it will be more convenient to renormalize the power of $\Xg$ with the time dependent $\CG(s)$ approximation of $\CG$ defined in \eqref{e:cgammatdef}. Consequently we will define
\begin{equation}
\label{e:coeffdefwiths}
\aag{2k+1}(s) \eqdef \left(e^{\frac{\CG(s)}{2}\Delta_X^*} \aatg{}\right)_{2k+1}
\end{equation}

and the corresponding decomposition
\begin{multline}
\label{e:polyrenormwiths}
\ppg^{(j)}(\Xg(z,s),s)  \eqdef \aag{1}(s)\Xg^{(j)} (z,s) + \aag{3}(s) \HXXgj{2}{\CG(s)} (z,s) + \dots \\
\dots + \aag{2n-1}(s) \HXXgj{2n-2}{\CG(s)} (z,s)
\end{multline}

The two similar decompositions \eqref{e:polyrenorm1} and \eqref{e:polyrenormwiths} will be useful for different purposes, in particular \eqref{e:polyrenormwiths} will be used when we will separate the linear part of the dynamic from the nonlinear one.\\
We now provide a description for the aforementioned polynomials as $\gamma$ goes to zero.\\
Assumption \eqref{e:hpmeasure} guarantees that the limit of $\aag{2k+1}$ is well defined. Moreover, from \eqref{e:polyrenormwiths} and \eqref{e:hpmeasure} we have that the following limit exists for every $s > 0$
\[
\aa{2k+1}(s) \eqdef \lim_{\gamma \to 0} \aag{2k+1}(s) = \lim_{\gamma \to 0} \left(e^{\frac{\CG(s)-\CG}{2}\Delta_X^*} \aag{}\right)_{2k+1} = \left(e^{\frac{A(s)}{2}\Delta_X^*} \aa{}\right)_{2k+1}
\]
and
\begin{multline*}
|\aa{2k+1}(s) - \aag{2k+1}(s)| \leq \left|\left(\Big( e^{\frac{A(s)}{2}\Delta_X^*}- e^{\frac{\CG(s)-\CG}{2}\Delta_X^*} \Big)\aag{}\right)_{2k+1}  \right| + \left|\left(e^{\frac{A(s)}{2}\Delta_X^*}  \left(\aa{} - \aag{} \right)\right)_{2k+1}  \right| \\
\lesssim s^{-\lambda} \alpha^{\lambda} |A(s)|^{|\kb|-1}+ |A(s)|^{|\kb|} c_0 \gamma^{\lambda_0}
\end{multline*}
where $A(s) \eqdef \lim_{\gamma \to 0} \CG(s)-\CG = \frac{s}{2} - \sum_{0 < |\omega|} \frac{e^{-2 \pi^2 s|\omega|^2}}{4\pi^2|\omega|^2}$ is a continuous function in $s$ on $(0,T]$ that diverges logarithmically as $s \to 0$. Here we used the bounds in \cite[Lemma~7.1]{MourratWeber} 
\[
 \left| A(s) -  \CG(s)+\CG \right| \lesssim  s^{-\lambda} \alpha^{\lambda}
\]
for $\lambda \in (0,1/2)$.
\subsection{Approximation and convergence of the nonlinear dynamic}

We will now introduce some approximations of the nonlinear part of the process.\\
We will first extend to the whole torus the relation \eqref{e:cap6equation}, in the same way as in \cite{MourratWeber}. This is not automatic since taking the power of the field do not commute with the trigonometric polynomial extension. In doing so recall the extension operator defined in \eqref{e:extentionoperator} and the definitions \eqref{e:highlowestensiondef}. Consider moreover the convolution
\[
F \star G (z) = \int_{[-1,1]^2} F(x-y)G(y) dy  
\]
for $x \in [-1,1]^2$.
\begin{proposition}
\label{prop:nonlineardynequation}
The multidimensional process $\Xg$, extended over the torus as in \eqref{e:extentionoperator}, started from $\Xng$ satisfies
\begin{equation}
\label{e:nonlineardynequation}
 \Xg(z,t) = \Pg{t} \Xng(z) + \int_0^t \Pg{t-s} \Kg \star \left( \pptg(\Xg(\cdot,s)) + \Err(\cdot,s)\right)(z) ds + \Zg(z,t)
\end{equation}
for $z \in \T^2, t \in [0,T]$, where the polynomial $\pptg$ is defined in \eqref{e:polydef}.\\
Moreover the error term satisfies
\begin{multline*}
\Norm{L^{\infty}(\T^2)}{\Err(\cdot,s)} \leq  C(T,\nu,\kappa)\left(1 + \Norm{\Cc^{-\nu}}{\Xg(\cdot,s)}\right)^{2n-2} \times\\
\times\left(\gamma^{-\kappa} \epsilon^{-2(n-1)\nu} \Norm{L^{\infty}(\T^2)}{\XgH(\cdot,s)}   + \gamma^{2} \epsilon^{-(2n+1)\nu} \Norm{\Cc^{-\nu}}{\Xg(\cdot,s)}^{3} \right)
\end{multline*}
\end{proposition}

\begin{proof}
The proof is the same as \cite[Lemma~7.1]{MourratWeber}, we only recall the bound on the error.  

From \eqref{e:errorterm} and Lemma~\ref{lemma:infinitybesovbound}, for $x \in \Le$
\begin{gather*}
|\Eg(x,s)| \leq C(\nn,\nu)\gamma^{2} \epsilon^{-(2n+1)\nu} \Norm{\Cc^{-\nu}}{\Xg(\cdot,s)}^{2n+1} 
\end{gather*}
and we can extend the previous inequality to $x \in \T^2$ at the expenses of an arbitrary small negative power of $\epsilon$.
\end{proof}

\begin{corollary}
\label{cor:nonlineardynequation}
Let $c_0 > 0$ and $\lambda_0 > 0$ as in  \eqref{e:hpmeasure}. Then the process $\Xg$ satisfies \eqref{e:nonlineardynequation} with the limiting polynomial $\pp$ (whose coefficient are independent of $\gamma$) defined in \eqref{e:polylimitdef} and the error term satisfying
\begin{multline*}
\Norm{L^{\infty}(\T^2)}{\Err(\cdot,s)} \leq  C(T,\nu,\kappa,\nn)\left(1 + \Norm{\Cc^{-\nu}}{\Xg(\cdot,s)}\right)^{2n-2} \\
\times\Big(\gamma^{-\kappa} \epsilon^{-2(n-1)\nu} \Norm{L^{\infty}(\T^2)}{\XgH(\cdot,s)} + \dots  \\
 \dots +  \gamma^{2} \epsilon^{-(2n+1)\nu} \Norm{\Cc^{-\nu}}{\Xg(\cdot,s)}^{3} + c_0 \gamma^{\lambda_0 - \kappa}\Norm{\Cc^{-\nu}}{\Xg(\cdot,s)}\Big)
\end{multline*}
\end{corollary}

\subsection{Da Prato - Debussche trick}

We are now ready to apply the idea of Da Prato and Debussche \cite{dPD} in our context, as it was applied in \cite{MourratWeber}. As described in Subsection~\ref{subsec:limitingspde}, the trick relies in the decomposition of the solution $\Xg$ into the linear term $\Zg$ approximation of the stochastic heat equation, and a reminder with finite quadratic variation, solving a PDE problem with random coefficients.\\
The treatment follows closely \cite{MourratWeber} and \cite{ShenWeber}, the only difference is given by the fact that in our case the process is multidimensional and an arbitrary quantity of Wick powers have to be controlled.\\
For $0 \leq t \leq T$ we will define the following approximation
\begin{equation}
\label{e:approximateX}
\bar{X}_{\gamma}(\cdot,t) \eqdef P_{t}\Xn(\cdot) + \Zg(\cdot,t) + \mathcal{S}_T\left(\left(\Zg^{:\kb:} \right)_{|\kb|\leq 2n-1}\right)(\cdot,t)
\end{equation}
where $\Xn$ is the initial condition for the continuous process (see also Assumption~\ref{e:hp1}) and $\mathcal{S}_T$ is the solution map described in Subsection~\ref{subsec:limitingspde}. Recall that, for any $\kappa>0$ and $\nu>0$, $\mathcal{S}_T$ is Lipschitz continuous from $L^{\infty}([0,T]; (\Cc^{-\nu})^{n^*})$ to $\Cc([0,T];\Cc^{2-\nu-\kappa})$ with $n^* = \binom{2n-2 + m}{m-1}$. In particular, by theorem \ref{Cap2Heateq} we have that the process $\bar{X}_{\gamma}$ converges in distribution to the solution $X$ of the SPDE \eqref{e:cap2limitingspde} as described in theorem \ref{thm:existuniqlimit}.\\
Since 
\[
\Norm{\Cc^{-\nu}}{ P_{t}\Xn  - \Pg{t}\Xng   } \leq \Norm{\Cc^{-\nu}}{ P_{t}(\Xn  - \Xng )  } + \Norm{\Cc^{-\nu}}{ (P_{t} - \Pg{t})\Xng   } \;.
\]
From \cite[Lemma~7.3]{MourratWeber} we have that
\[
\lim_{\gamma \to 0}\sup_{0 \leq t \leq T} \Norm{\Cc^{-\nu}}{ P_{t}\Xn  - \Pg{t}\Xng   } = 0\;.
\]
It remains to control the difference between
\begin{align}
\label{e:vdef}
v_{\gamma} (x,t) &= \Xg(x,t) - \Zg(x,t) - \Pg{t}\Xng(x,t)\\
\label{e:vbardef}
\bar{v}_{\gamma}(x,t) &= \bar{X}_{\gamma}(x,t) - \Zg(x,t) - P_{t}\Xn(x,t)\;.
\end{align}
For $t \in [0,T]$, $x \in \T^2$. In order to do so, it is more convenient to start the reminder processes $v_{\gamma}$ and $\bar{v}_{\gamma}$ from zero and add the initial condition to the martingales. This can be done rearranging the contribution of the initial condition and defining, for $\kb \in \N^m$
\begin{align}
\tilde{Z}_{\gamma} &\eqdef \Pg{t} \Xng+\Zg &
\bar{Z}_{\gamma}^{:\kb:} &\eqdef \sum_{\substack{ \mathbf{a}\in \N^m\\ \mathbf{a} \leq \kb}} (P_{t} \Xn)^{\mathbf{a}} \Zg^{:\kb-\mathbf{a}:} \;.
\label{e:barZpowerdefin}
\end{align}
The last relation is similar to \eqref{e:hermitesum} for the Hermite polynomial and
\[
H\left((\tilde{Z}_{\gamma} + v_{\gamma} )^{\kb}, \CG \right) = \sum_{\substack{ \mathbf{a}\in \N^m\\ \mathbf{a} \leq \kb}} v_{\gamma}^{\mathbf{a}} H\left(\tilde{Z}_{\gamma}^{\kb - \mathbf{a}}, \CG \right)\;.
\]
Recall the heat kernel regularization properties of \cite[Cor.~8.7]{MourratWeber}, for $\lambda > - \nu$
\[
\Norm{\Cc^{\lambda}}{(P_{t} {\Xn})^{(j)}} \leq C(\lambda) t^{-\frac{\lambda + \nu}{2}} \Norm{\Cc^{-\nu}}{{\Xn}^{(j)}} \leq C(\lambda) t^{-\frac{\lambda + \nu}{2}} \Norm{(\Cc^{-\nu})^m}{\Xn} 
\]
and the Besov multiplicative inequality 
\begin{multline}
\label{e:initialcondinequality}
\Norm{\Cc^{-\nu}}{\bar{Z}_{\gamma}^{:\kb:}(\cdot,t)} \leq C(\nu,\kappa)\sum_{\substack{ \mathbf{a}\in \N^m\\ \mathbf{a} \leq \kb}} \sup_{j = 1,\dots,m} \Norm{\Cc^{\nu + \kappa}}{(P_{t} {\Xn})^{(j)}}^{|\mathbf{a}|}\Norm{\Cc^{-\nu}}{\Zg^{:\kb-\mathbf{a}:}(\cdot,t) } \\
\leq C(\nu,\kappa)\sum_{\substack{ \mathbf{a}\in \N^m\\ \mathbf{a} \leq \kb}} t^{- \left(\nu +\frac{ \kappa}{2}\right)|\mathbf{a}|} \Norm{(\Cc^{-\nu})^m}{\Xn}^{|\mathbf{a}|} \Norm{\Cc^{-\nu}}{\Zg^{:\kb-\mathbf{a}:} (\cdot,t)} \\
\leq C\left(n,\nu,\kappa,T,\Norm{(\Cc^{-\nu})^m}{\Xn}\right)  t^{- \left(\nu +\frac{ \kappa}{2}\right) (2n-1)}   \sup_{\substack{ \mathbf{a}\in \N^m\\ |\mathbf{a}| \leq 2n-1}} \Norm{\Cc^{-\nu}}{\Zg^{:\mathbf{a}:} (\cdot,t)}
\end{multline}
for every $\kb \in \N^m$ with $|\kb| \leq 2n-1$, the degree of the polynomial $\pp$.\\
This allows us to write, using the definition of $\mathcal{S}_T$ in Section~\ref{subsec:limitingspde}, the relation for the $j$-th component of \eqref{e:vdef} and \eqref{e:vbardef}
\[
\bar{v}_{\gamma}^{(j)} (\cdot,t) = \int_0^t P_{t-s}  \overline{\Psi}^{(j)}\left(s, (\bar{Z}_{\gamma}^{:\kb:})_{|\kb| \leq 2n-1} \right)\left(\bar{v}_{\gamma} (\cdot,s)\right) ds
\]
while a similar expression holds for $v_{\gamma}^{(j)}$ in virtue of Proposition~\ref{prop:nonlineardynequation} and corollary \ref{cor:nonlineardynequation}
\[
v_{\gamma}^{(j)} (\cdot,t) =\int_0^t \Pg{t-s} \Kg \star \left( \pp^{(j)}(\tilde{Z}_{\gamma}(\cdot,s) + v_{\gamma} (\cdot,s)) + \Err_1^{(j)}(\cdot,s)\right) ds
\]
with $\Err_1$ satisfying the bound in corollary \ref{cor:nonlineardynequation}.
The next proposition allows a control over the nonlinear part in a space of functions rather than distributions. It corresponds to  in \cite[Lemma~7.5]{MourratWeber}  and \cite[Lemma~4.8]{ShenWeber}
\begin{proposition}
\label{prop:propvdiffstat}
There exists sufficiently small $\nu > 0$ and $\kappa >0$, such that for all $T >0$ the following inequality holds
\begin{multline}
\label{e:propvdiffstat}
\Norm{\Cc^{\frac{1}{2}}}{\bar{v}_{\gamma}^{(j)} (\cdot,t) - v_{\gamma}^{(j)} (\cdot,t)} \leq C_1 \int_0^t (t-s)^{-\frac{1}{3}} s^{-\frac{1}{6}} \Norm{(\Cc^{1/2})^m}{\bar{v}_{\gamma} (\cdot,s) - v_{\gamma} (\cdot,s)} ds \\
+ C_1\left( \Norm{\Cc^{-\nu}}{\Xn - \Xng}  + \gamma^{\frac{1}{2}}+ \gamma^{ \lambda_0 - 2\kappa}\right) + C_2 |\overline{\Err_2}(t)|
\end{multline}
and the error term satisfies
\begin{equation}
\label{e:propvdifferror}
\E\left[\sup_{t \leq T} |\overline{\Err_2}(t) |\right] \leq C(T,\nu,\nn,\kappa,n) (\gamma\epsilon^{-\kappa} + \gamma^{1 - \kappa} )
\end{equation}
\end{proposition}
with the constant $C_1$ depending on $\nu,\kappa,T ,n, \Norm{ (\Cc^{-\nu})^m  }{\Xng},\sup_{s \leq T}\Norm{(\Cc^{-\nu})^m }{\Zg^{:\mathbf{a}:}(\cdot,s)}$ with $|\mathbf{a}|\leq 2n-1$ and $\sup_{s \leq T} \Norm{(\Cc^{1/2})^m}{v_{\gamma}(\cdot,s)}$, while $C_2$ depends on $T,\nu,\kappa,n$.
\begin{proof}
We are going to give a complete proof of this bound since it is the central ingredient for the proof of the main theorem. To keep the formulas light, we will use $L^p$ in place of $L^p(\T^2)$.\\
Decompose the difference into
\begin{multline}
\label{e:nonlinearprocdiff}
\bar{v}_{\gamma}^{(j)} (\cdot,t) - v_{\gamma}^{(j)} (\cdot,t) =  \int_0^t \left(P_{t-s} - \Pg{t-s}\Kg \right)\star \overline{\Psi}^{(j)}\left(s, (\bar{Z}_{\gamma}^{:\kb:})_{|\kb| \leq 2n-1} \right)\left(\bar{v}_{\gamma} (\cdot,s)\right) ds \\
+ \int_0^t \Pg{t-s}\Kg \star \left(\overline{\Psi}^{(j)}\left(s, (\bar{Z}_{\gamma}^{:\kb:})_{|\kb| \leq 2n-1} \right)\left(\bar{v}_{\gamma} (\cdot,s)\right) -  \pp^{(j)}\left(\tilde{Z}_{\gamma}(\cdot,s) + v_{\gamma} (\cdot,s)\right) \right) ds \\
- \int_0^t \Pg{t-s} \Kg \star \Err_2^{(j)}(\cdot,s) ds\;.
\end{multline}
The  first term in \eqref{e:nonlinearprocdiff} is bounded in $\Cc^{\frac{1}{2}}$ using Lemma~\ref{lemma:boundsonsemigroup} with $\lambda$ and $\kappa$ satisfying $-\lambda - \frac{1}{4}-\frac{\nu}{2} -\kappa > -1$ and $\left(\nu + \frac{\kappa}{2}\right)(2n-1) < 1$
\begin{multline}
\label{e:vdifferr1}
\int_0^t\Norm{\Cc^{\frac{1}{2}}}{ \left(P_{t-s} - \Pg{t-s}\Kg \right)\star \overline{\Psi}^{(j)}\left(s, (\bar{Z}_{\gamma}^{:\kb:})_{|\kb| \leq 2n-1} \right)\left(\bar{v}_{\gamma} (\cdot,s)\right) }ds \\
\leq C(T,\lambda,\kappa) \int_0^t (t-s)^{-\lambda - \frac{1}{4}-\frac{\nu}{2} -\kappa} \gamma^{\lambda}  \Norm{L^{\infty}([0,T],\Cc^{-\nu})}{\overline{\Psi}^{(j)}\left(s, (\bar{Z}_{\gamma}^{:\kb:})_{|\kb| \leq 2n-1} \right)\left(\bar{v}_{\gamma}(\cdot,s)\right)}ds\\
\leq C\gamma^{\lambda}  \int_0^t (t-s)^{-\lambda - \frac{1}{4}-\frac{\nu}{2} -\kappa}  s^{-\left(\nu+ \frac{\kappa}{2}\right)(2n-1) } ds \leq C \gamma^{\lambda}
\end{multline}
where the last constant depends on $T,\lambda,\nu,\kappa,\Norm{L^{\infty}([0,T];\Cc^{\frac{1}{2}})}{\bar{v}_{\gamma}}, \Norm{L^{\infty}([0,T];\Cc^{-\nu})}{\Zg^{:\mathbf{a}:} }$.\\
The third part of \eqref{e:nonlinearprocdiff} 
\[
\sup_{0 \leq t \leq T} \Norm{\Cc^{\frac{1}{2}}}{ \int_0^t \Pg{t-s}\Kg \star \Err_1^{(j)}(\cdot,s) ds } \leq C(T) \int_0^T (T-s)^{-\frac{1}{4}- \kappa} \Norm{L^{\infty}}{\Err_2^{(j)}(\cdot,s)} ds
\]
is bounded with corollary \ref{cor:nonlineardynequation}: in particular we will need the following bounds provided by Lemma~\ref{lemma:boundsonsemigroup}
\[
\Norm{(\Cc^{-\nu})^m}{\Xg(\cdot,s)} \leq \Norm{(\Cc^{-\nu})^m}{\Zg(\cdot,s)} +\Norm{(\Cc^{-\nu})^m}{\Pg{s}\Xng} + \Norm{(\Cc^{1/2})^m}{v_{\gamma}(\cdot,s)}
\]
\begin{multline*}
\Norm{(L^{\infty})^m}{\XgH(\cdot,s)} \leq \Norm{(L^{\infty})^m}{\ZgH(\cdot,s)} +\Norm{(L^{\infty})^m}{(\Pg{s}\Xng)^{high}}+ \Norm{(L^{\infty})^m}{v_{\gamma}^{high}(\cdot,s)} \\
\lesssim  \Norm{(L^{\infty})^m}{\ZgH(\cdot,s)} + \left(\epsilon \gamma^{-1}\right)^{\lambda} t^{-\lambda -\frac{n \nu}{2(n-1)}} \Norm{\Cc^{-\nu}}{\Xng} +  \left(\epsilon \gamma^{-1}\right)^{\frac{1}{2}}\Norm{(\Cc^{1/2})^m}{v_{\gamma}(\cdot,s)}
\end{multline*}
that implies that the left hand side of \eqref{e:propvdifferror} is bounded by
\begin{gather*}
\leq C \left(\gamma^{-\kappa}\epsilon^{-2(n-1)\nu} \left(\Norm{(L^{\infty})^m}{\ZgH(\cdot,s)} + \left(\epsilon \gamma^{-1}\right)^{1/2} \right) + \gamma^{2} \epsilon^{-(2n+1)\nu} + \gamma^{\lambda_0 - \kappa}\right)
\end{gather*}
for $\lambda = 1/2$ and values of $\nu$ and $\kappa$ small enough. The value of the constant $C$ depends on $T,\nu,\kappa,n, \sup_{s \leq T} , \Norm{ (\Cc^{-\nu})^m  }{\Xng}$  and it is a polynomial function of the random quantities $\Norm{(\Cc^{-\nu})^m }{\Zg(\cdot,s)}, \sup_{s \leq T} \Norm{(\Cc^{1/2})^m}{v_{\gamma}(\cdot,s)}$. Using the inequality $Cab \leq \frac{1}{2}(C^2 \gamma a^2 + \gamma^{-1}b^2)$, we separate the constant and $\Norm{(L^{\infty})^m}{\ZgH(\cdot,s)}$. In particular we have that the error is controlled, changing the value of the constant if needed, by
\begin{multline}
\label{e:vdifferr2}
\sup_{t \leq T}\Norm{\Cc^{\frac{1}{2}}}{ \int_0^t \Pg{t-s}\Kg \star \Err_2^{(j)}(\cdot,s) ds } \\
\leq C \gamma^{-\kappa'}\left( \gamma^{1 \vee \lambda_0} + (\epsilon \gamma^{-1})^{1/2}\right) + \gamma^{-1}\sup_{s \leq T}\Norm{(L^{\infty})^m}{\ZgH(\cdot,s)}^2 ds
\end{multline}
for $\kappa'$ small enough dependent on $\kappa, \nu, n$. Lemma~\ref{lemma:highfreqcontrol} offers a control in expectation of the high frequencies of $\Zg$ and completes the treatment of the error term.\\
It remains to bound the second term of \eqref{e:nonlinearprocdiff}. In order to do this we make use of the expression for $\overline{\Psi}$ defined in \eqref{e:barPsidef}
\[
 \overline{\Psi}^{(j)}\left(s, (\bar{Z}_{\gamma}^{:\kb:})_{|\kb| \leq 2n-1} \right)\left(\bar{v}_{\gamma} (\cdot,s)\right) =  \sum_{ |\mathbf{b}| + |\mathbf{a}|\leq 2n-1} b_{\mathbf{a},\mathbf{b}}^{(j)}(s) \bar{v}_{\gamma}^{\mathbf{a}}(\cdot,s)\bar{Z}_{\gamma}^{:\mathbf{b}:}(\cdot,s)
\]
and using \eqref{e:barZpowerdefin}
\begin{multline}
\overline{\Psi}^{(j)}\left(s, (\bar{Z}_{\gamma}^{:\kb:})_{|\kb| \leq 2n-1} \right)\left(\bar{v}_{\gamma} (\cdot,s)\right) -  \pp^{(j)}\left(\tilde{Z}_{\gamma}(\cdot,s) + v_{\gamma} (\cdot,s)\right) \\
= \sum_{ |\mathbf{b}| + |\mathbf{a}|\leq 2n-1} b_{\mathbf{a},\mathbf{b}}^{(j)}(s) \left( \bar{v}_{\gamma}^{\mathbf{a}}(\cdot,s)\bar{Z}_{\gamma}^{:\mathbf{b}:}(\cdot,s) - v_{\gamma}^{\mathbf{a}}(\cdot,s)\tilde{Z}_{\gamma}^{:\mathbf{b}:}(\cdot,s) \right) \\
= \sum_{\mathbf{a},\mathbf{b},\mathbf{c}} b_{\mathbf{a},\mathbf{b},\mathbf{c}}^{(j)}(s) \Bigg( \bar{v}_{\gamma}^{\mathbf{a}}(\cdot,s)  (P_s\Xn)^{\mathbf{b}}(\cdot,s) \Zg^{:\mathbf{c}:}(\cdot,s) - v_{\gamma}^{\mathbf{a}}(\cdot,s)  (\Pg{s}\Xng)^{\mathbf{b}}(\cdot,s) H_{\mathbf{c}}(\Zg(s),\CG(s))\Bigg)
\label{e:polydecomposition}
\end{multline}
for some real $b_{\mathbf{a},\mathbf{b},\mathbf{c}}^{(j)}(s)$ growing like a power of $\log(s^{-1})$ as $s \to 0$, and satisfying $|b_{\mathbf{a},\mathbf{b},\mathbf{c}}^{(j)}(s)| \leq C(T,\kappa)s^{-\kappa}$ for $s \in [0,T]$. It is sufficient to bound each term in the sum of \eqref{e:polydecomposition}.
Using the Besov multiplicative inequality \ref{prop:productbesov},
\begin{multline*}
\Norm{\Cc^{-\nu}}{  \left(\bar{v}_{\gamma}^{\mathbf{a}}(\cdot,s)  -v_{\gamma}^{\mathbf{a}}(\cdot,s) \right) (P_s\Xn)^{\mathbf{b}}(\cdot,s) \Zg^{:\mathbf{c}:}(\cdot,s)   } \\
\leq \Norm{\Cc^{\frac{1}{2}}}{\bar{v}_{\gamma}^{\mathbf{a}}(\cdot,s)  -v_{\gamma}^{\mathbf{a}}(\cdot,s)}\Norm{\Cc^{-\nu}}{  (P_s\Xn)^{\mathbf{b}}(\cdot,s) \Zg^{:\mathbf{c}:}(\cdot,s)   } \\
\lesssim  \Norm{(\Cc^{\frac{1}{2}})^m}{\bar{v}_{\gamma} (s)  -v_{\gamma}(s) } \Norm{(\Cc^{\frac{1}{2}})^m}{|\bar{v}_{\gamma}(s)|+|v_{\gamma}(s)|}^{|\mathbf{a}|-1} \Norm{(\Cc^{\nu+\kappa})^m}{P_s\Xn(s) }^{|\mathbf{b}|}\Norm{\Cc^{-\nu}}{\Zg^{:\mathbf{c}:}(s)}  
\end{multline*}
where $\Norm{(\Cc^{\nu+\kappa})^m}{P_s\Xn(s) } \leq C(\nu,\kappa) s^{-2\nu - \kappa}\Norm{(\Cc^{-\nu})^m}{\Xn}$. And similarly
\begin{multline*}
\Norm{\Cc^{-\nu}}{  v_{\gamma}^{\mathbf{a}}(s) \left((P_s\Xn)^{\mathbf{b}}(s) - (\Pg{s}\Xng)^{\mathbf{b}}(s)\right)\Zg^{:\mathbf{c}:}(s)   } \\
\leq \Norm{(\Cc^{\frac{1}{2}})^m}{v_{\gamma} (s)}^{|\mathbf{a}|-1} s^{-(|\mathbf{b}|-1)(2\nu + \kappa)} \left(\Norm{(\Cc^{-\nu})^m}{ \Xn } + \Norm{(\Cc^{-\nu})^m}{ \Xng} \right)^{|\mathbf{b}|-1} \Norm{\Cc^{-\nu}}{\Zg^{:\mathbf{c}:}(s)} 
\end{multline*}
where we used \cite[Lemma~7.3]{MourratWeber}. We get
\begin{multline*}
\leq C s^{-(2n-1)(2\nu-\kappa) }\Bigg( \Norm{(\Cc^{\frac{1}{2}})^m}{\bar{v}_{\gamma} (s)  -v_{\gamma}(s) } + \Norm{(\Cc^{-\nu})^m}{\Xn - \Xng} \\
+ \epsilon^{-\kappa}\sup_{|\mathbf{a}|\leq 2n-1}\Norm{L^{\infty}(\Le)}{\Zg^{:\mathbf{a}:}(s) - H_{\mathbf{a}}(\Zg(s),\CG(s))}\Bigg)
\end{multline*}
where the constant depends on $\nu,\kappa,n,T,\Norm{\Cc^{-\nu}}{\Xn},\Norm{(\Cc^{-\nu})^m}{\Xng},\sup_{|\mathbf{a}| \leq 2n-1}\Norm{\Cc^{-\nu}}{\Zg^{:\mathbf{a}:}}$ as well as on $\Norm{L^{\infty}([0,T];(\Cc^{\frac{1}{2}})^m)}{v_{\gamma}},\Norm{L^{\infty}([0,T];(\Cc^{\frac{1}{2}})^m)}{\bar{v}_{\gamma}}$. The last term is estimated probabilistically with Proposition~\ref{prop:hermiteRbound}, where the supremum on the torus is bounded be the supremum on $\Le$ at a cost of an arbitrarily small negative power of $\epsilon$. Using Proposition~\ref{prop:boundsonP} we then bound the $ \Cc^{1/2}$ Besov norm of the second term in \eqref{e:nonlinearprocdiff} with the sum
\begin{multline}
\label{e:nonlindiffyoung}
C  \int_0^t (t-s)^{-\frac{1}{4}-\frac{\nu}{2}-\kappa}  s^{-\kappa} s^{-(2n-1)(2\nu-\kappa) } \Norm{(\Cc^{\frac{1}{2}})^m}{\bar{v}_{\gamma} (s)  -v_{\gamma}(s) } ds + C \Norm{(\Cc^{-\nu})^m}{\Xn - \Xng}  \\
+ C^2 \int_0^t (t-s)^{-\frac{1}{4}-\frac{\nu}{2}-\kappa}  s^{-\kappa} s^{-(2n-1)(2\nu-\kappa) } \epsilon^{-2\kappa}  \left(\gamma^{1 - \kappa - \nu} + \gamma^{-\kappa} s^{-\frac{1}{2}} \alpha^{\frac{1}{2}} \right)^2 \gamma^{-1} \\
+ \int_0^t (t-s)^{-\frac{1}{4}-\frac{\nu}{2}-\kappa}  s^{-\kappa} s^{-(2n-1)(2\nu-\kappa) } \epsilon^{-2\kappa} \gamma \sup_{|\mathbf{a}|\leq 2n-1}\frac{\Norm{L^{\infty}(\Le)}{\Zg^{:\mathbf{a}:}(s) - H_{\mathbf{a}}(\Zg(s),\CG(s))}^2}{ \left(\gamma^{1 - \kappa - \nu} + \gamma^{-\kappa} s^{-\frac{1}{2}} \alpha^{\frac{1}{2}} \right)^{2}} ds
\end{multline}
in the last line we used the inequality $CAB \leq \frac{1}{2}(C^2\gamma^{-1} A^2 + \gamma B^2)$. The last term is bounded in expectation with Proposition~\ref{prop:hermiteRbound} with $b = \frac{1}{2}$ (note the absence of any multiplicative constant in front of the last term). Using the fact 
\[
(t-s)^{-\frac{1}{4}-\frac{\nu}{2}-\kappa}  s^{-\kappa} s^{-(2n-1)(2\nu-\kappa) } \leq C(T,\kappa,\nu,n) (t-s)^{-\frac{1}{3}} s^{-\frac{1}{6}}
\]
for small enough $\kappa,\nu > 0$. Collecting together \eqref{e:nonlindiffyoung}, \eqref{e:vdifferr2} and \eqref{e:vdifferr1} with $\lambda = \frac{1}{2}$ we obtain the bound
\begin{multline*}
\Norm{\Cc^{\frac{1}{2}}}{\bar{v}_{\gamma}^{(j)} (\cdot,t) - v_{\gamma}^{(j)} (\cdot,t)} \leq C_1 \int_0^t (t-s)^{-\frac{1}{3}} s^{-\frac{1}{6}} \Norm{(\Cc^{1/2})^m}{\bar{v}_{\gamma} (\cdot,s) - v_{\gamma} (\cdot,s)} ds \\
+ C_1\left( \Norm{(\Cc^{-\nu})^m}{\Xn - \Xng}  + \gamma^{\frac{1}{2}}+ \gamma^{ \lambda_0 - 2\kappa}\right) + C_2 \overline{\Err_2}(t)
\end{multline*}
and the expectation of
\begin{multline*}
\sup_{t \leq T}|\overline{\Err_2}(t)| \leq C(T)\sup_{t \leq T}\gamma^{-1}\int_0^t (t-s)^{-\frac{1}{3}}\Norm{(L^{\infty})^m}{\ZgH(\cdot,s)}^2 ds \\
+ C(T)\gamma \epsilon^{-\kappa} \sup_{t \leq T}\int_0^t (t-s)^{-\frac{1}{3}} s^{-\frac{1}{6}}\sup_{|\mathbf{a}| \leq 2n-1}\frac{\Norm{L^{\infty}(\Le)}{\Zg^{:\mathbf{a}:}(s) - H_{\mathbf{a}}(\Zg(s),\CG(s))}^2}{ \left(\gamma^{1 - \kappa - \nu} + \gamma^{-\kappa} s^{-\frac{1}{2}} \alpha^{\frac{1}{2}}\right)^{2}} ds
\end{multline*}
is bounded by $C(T,\nu,\nn,\kappa,n) (\gamma\epsilon^{-\kappa} + \gamma^{1 - \kappa} )$ where we used the scaling \eqref{e:scaling}. In the above equation the constants are as after \eqref{e:propvdifferror}.
\end{proof}

\begin{appendix}
\appendix
\section{}
\subsection{Estimation on the kernels and the semigroup}
\label{subsec:kernel}
We collect, for reference, some of the estimates in \cite{MourratWeber}, some of which have been modified to adapt to our context. In this case we are providing a description of how the proof in \cite{MourratWeber} should be modified to accommodate our situation.
\begin{proposition}
\label{prop:Kbounds}
Consider the scaling in \eqref{e:scaling}. We have that, for $|\omega| \leq \gamma \epsilon^{-1}$ there exists a constant $C$
\begin{align*}
|\hat{\Kg}(\omega) |  &\leq 1\\
|\partial_j \hat{\Kg}(\omega) |  &\leq C\epsilon^2 \gamma^{-2}|\omega| \\
|\partial_j^2 \hat{\Kg}(\omega) |  &\leq C|\epsilon \gamma^{-1}|^2
\end{align*}
for $|\omega| \geq \gamma \epsilon^{-1}$
\begin{align*}
|\hat{\Kg}(\omega) |  &\leq C |\epsilon \gamma^{-1} \omega|^{-2}\\
|\partial_j \hat{\Kg}(\omega) |  &\leq C\epsilon \gamma^{-1}|\epsilon \gamma^{-1} \omega|^{-2} \\
|\partial_j^2 \hat{\Kg}(\omega) |  &\leq C\epsilon^2 \gamma^{-2}|\epsilon \gamma^{-1} \omega|^{-2}\;.
\end{align*}
Moreover for any $|\omega|\leq \epsilon^{-1}$ there exists a constant $c>0$
\begin{align*}
1 - \hat{\Kg}(\omega) & \geq c |\epsilon \gamma^{-1} \omega|^2\;.
\end{align*}
\end{proposition}
A proof of this proposition is given in lemma 8.1 and 8.2 of \cite{MourratWeber}.\\

The next lemma provides an estimate of the $L^{\infty}(\T^2)$ norm for $\Pg{t} \Kg$ that is basically lemma 8.3 in \cite{MourratWeber}
\begin{lemma}
\label{lemma:supboundonPK}
For $T > 0$, $x \in \T^2$ there exists a $C=C(T)$ we have
\begin{equation}
|\Pg{t} \Kg (x)|\leq C ( t^{-1} \wedge \epsilon^{-2}\gamma^2 )\log \gamma^{-1} \;.
\end{equation}
\end{lemma}

\begin{lemma}
\label{lemma:est-cap2kernel3}
For $\gamma$ small enough
\begin{equation}
\sup_{t \geq 0}\int_0^t \sum_{\omega:|\omega|\leq \epsilon^{-1}} |\widehat{\Pg{t-s} \Kg}(\omega)|^2 ds \leq\frac{1}{2} \sum_{0 < |\omega|\leq \epsilon^{-1}} \frac{|\hat{\Kg}(\omega)|^2}{\epsilon^{-2}\gamma^2 (1 - \hat{\Kg}(\omega))} \lesssim \log(\gamma^{-1})\;.
\end{equation}
\end{lemma}
The lemma follows immediately from the estimations in Proposition~\ref{prop:Kbounds}.

\begin{lemma}
\label{lemma:infinitybesovbound}
For any $\nu >0$ there exists constants $c,C(\nu)$ such that for all $X :\T^2 \to \R$ for which $\hat{X}(\omega)=0$ for all $|\omega| > \epsilon^{-2}$ we have
\[
\Norm{L^{\infty}(\T^2)}{X} \leq c\log(\epsilon^{-1}) \Norm{L^{\infty}(\Le)}{X}  \leq C(\nu)\epsilon^{-\kappa} \Norm{\Cc^{-\nu}}{X}\;.
\]
\end{lemma}
The proof of the above proposition is given in the Appendix of \cite{MourratWeber}.\\

We recall some bounds from section 8 in \cite{MourratWeber} regarding the semigroup associated to the diffusion. A minor difference is given by the following proposition, which depends on the scale of the model.\\
Recall that, for the heat semigroup $P_t$ and an element $X$ of $\Cc^{\nu}$ we have for $\beta > 0$
\begin{equation}
\label{e:heatsgregimpr}
\Norm{\Cc^{\nu + \beta}}{P_t X} \leq C(\nu,\beta)t^{-\frac{\beta}{2}}\Norm{\Cc^{\nu}}{X}\;.
\end{equation}
The next proposition will provide similar bounds for the approximate heat semigroup $\Pg{t}$.

\begin{proposition}
\label{prop:boundsonP}
For $\gamma$ sufficiently small, for $c_1,c_2 > 0$, $T> 0$, $\kappa>0$.\\
Then
\begin{itemize}
\item For $\beta >0$ and $0 \leq \lambda \leq 1$ there exists $C = C(c_1,T,\kappa,\beta,\lambda)$ such that for all functions $X: \T^2 \to \R$ with $\hat{X}(\omega) = 0$ for all $|\omega|\geq c_1 \epsilon^{-1}\gamma$ we have that for all $t \in [0,T]$ and $\nu \in \R$
\begin{align*}
\Norm{\Cc^{\nu+\beta -\kappa}}{P^{\gamma}_t X} &\leq C t^{-\frac{\beta}{2}} \Norm{\Cc^{\nu}}{X}\\
\Norm{\Cc^{\nu -\kappa}}{(P^{\gamma}_t -P_t)X} &\leq C \gamma^{\lambda} \left( t^{-\frac{\lambda}{2}} \Norm{\Cc^{\nu}}{X} \wedge \Norm{\Cc^{\nu+\lambda}}{X}\right)\\
\Norm{\Cc^{\nu -\kappa}}{\Kg \star X} &\leq C \Norm{\Cc^{\nu}}{X}\\
\Norm{\Cc^{\nu -\kappa}}{\Kg \star X - X} &\leq C \gamma^{2\lambda}\Norm{\Cc^{\nu+2 \lambda}}{X}\\
\end{align*}
\item For $\beta>0$ and $\lambda > 0$ there exists $C =C(c_2,T,\kappa,\beta,\lambda)$ such that for any distribution $X$ with  $\hat{X}(\omega) = 0$ for $|\omega|\leq c_2 \epsilon^{-1}\gamma$, for $t \in [0,T]$ and $\nu \in \R$
\begin{equation}
\label{e:heatkernelhigh}
\Norm{\Cc^{\nu+\beta -\kappa}}{P^{\gamma}_t X}   \leq C t^{-\beta \frac{n}{2(n-1)} - \lambda} (\epsilon \gamma^{-1})^{\lambda}  \Norm{\Cc^{\nu}}{X} \\
\end{equation}
and if $0 \leq \beta \leq 2$
\[
\Norm{\Cc^{\nu+\beta -\kappa}}{P^{\gamma}_t \Kg \star X}  \leq C t^{-\frac{\beta}{2}} \Norm{\Cc^{\nu}}{X}\;.
\]
\end{itemize}
\end{proposition}

This correspond to lemma 8.4 in \cite{MourratWeber}, we highlight a small difference in \eqref{e:heatkernelhigh}, due to the different scaling. The bound produced is actually better for high values of $n$.\\
A proof of \eqref{e:heatkernelhigh} is given using the inequality, valid for $\epsilon^{-1}\gamma \leq |\omega| \leq \epsilon^{-1}$
\[
e^{- t \epsilon^{-2}\gamma^2 (1 - \hKg(\omega))} \leq \exp \left( - \frac{t}{C_1} \gamma^{2 - 2n}\right) \lesssim t^{-\beta\frac{n}{2(n-1)}} \gamma^{2 \beta n} \lesssim t^{-\beta\frac{n}{2(n-1)}} |\omega|^{-\beta}\;.
\]

\begin{proposition}
\label{prop:differenceheatkernel}
For $\gamma$ small enough, for $T >0$, $\kappa>0$ and $0 \leq \lambda \leq 1$ there exists a constant $C=C(T,\kappa,\beta,\lambda)$ such that for $t \in [0,T]$, $\nu \in \R$ and any distribution $X$ on $\T^2$.
\begin{align*}
\Norm{\Cc^{\nu-\kappa}}{(P^{\gamma}_t - P_t)X} &\leq C(\epsilon \gamma^{-1})^{\lambda}\left( t^{- \frac{n}{2(n-1)} \lambda}\Norm{\Cc^{\nu}}{X} \wedge \Norm{\Cc^{\nu+\lambda}}{X}\right)\\
\Norm{\Cc^{\nu-\kappa}}{(P^{\gamma}_t \Kg - P_t)X} &\leq C(\epsilon \gamma^{-1})^{\lambda}\left( t^{-\frac{\lambda}{2}}\Norm{\Cc^{\nu}}{X} \wedge \Norm{\Cc^{\nu+\lambda}}{X}\right)\:,
\end{align*}
\end{proposition}

The next lemma will be used in Section~\ref{sec:nonlinear}. It is proven in the same way as the above propositions

\begin{lemma}
\label{lemma:boundsonsemigroup}
For $0 < \lambda$ and any $\kappa > 0$ there exists a constant $C=C(T,\lambda,\kappa)$ such that for $0 \leq t \leq T$ such that 
\begin{equation}
\Norm{\Cc^{-\nu} \to \Cc^{\beta}}{ P_t - \Pg{t} \Kg} \leq C (\epsilon\gamma^{-1})^{2 \lambda } t^{-\lambda  - \frac{\nu + \beta}{2}-\kappa}\;.
\end{equation}
\end{lemma}

\end{appendix}

\bibliographystyle{Martin}
\bibliography{references}

\begin{thebibliography}{BPRS93}
\expandafter\ifx\csname url\endcsname\relax
  \def\url#1{\texttt{#1}}\fi
\expandafter\ifx\csname urlprefix\endcsname\relax\def\urlprefix{URL }\fi

\bibitem[AK65]{aheizer1965classical}
\textsc{N.~I. Aheizer} and \textsc{N.~Kemmer}.
\newblock \emph{The classical moment problem and some related questions in
  analysis}.
\newblock Oliver \& Boyd, 1965.

\bibitem[BPRS93]{MR1317994}
\textsc{L.~Bertini}, \textsc{E.~Presutti}, \textsc{B.~R{\"u}diger}, and
  \textsc{E.~Saada}.
\newblock Dynamical fluctuations at the critical point: convergence to a
  nonlinear stochastic {PDE}.
\newblock \emph{Teor. Veroyatnost. i Primenen.} \textbf{38}, no.~4, (1993),
  689--741.
\newblock \ifx\href\undefined
  \texttt{doi:10.1137/1138062}\else\href{http://dx.doi.org/10.1137/1138062}{\texttt{doi:10.1137/1138062}}\fi.

\bibitem[BZ97]{bovier1997low}
\textsc{A.~Bovier} and \textsc{M.~Zahradn{\'\i}k}.
\newblock The low-temperature phase of kac-ising models.
\newblock \emph{Journal of statistical physics} \textbf{87}, no. 1-2, (1997),
  311--332.

\bibitem[CMP95]{cassandro1995corrections}
\textsc{M.~Cassandro}, \textsc{R.~Marra}, and \textsc{E.~Presutti}.
\newblock Corrections to the critical temperature in 2d ising systems with kac
  potentials.
\newblock \emph{Journal of statistical physics} \textbf{78}, no.~3, (1995),
  1131--1138.

\bibitem[DPD03]{dPD}
\textsc{G.~Da~Prato} and \textsc{A.~Debussche}.
\newblock Strong solutions to the stochastic quantization equations.
\newblock \emph{Ann. Probab.} \textbf{31}, no.~4, (2003), 1900--1916.

\bibitem[EH17]{erhard2017discretisation}
\textsc{D.~Erhard} and \textsc{M.~Hairer}.
\newblock Discretisation of regularity structures.
\newblock \emph{arXiv preprint arXiv:1705.02836} (2017).

\bibitem[Hai14]{Regularity}
\textsc{M.~Hairer}.
\newblock A theory of regularity structures.
\newblock \emph{Invent. Math.} \textbf{198}, no.~2, (2014), 269--504.
\newblock \ifx\href\undefined
  \texttt{arXiv:1303.5113}\else\href{http://arxiv.org/abs/1303.5113}{\texttt{arXiv:1303.5113}}\fi.
\newblock \ifx\href\undefined
  \texttt{doi:10.1007/s00222-014-0505-4}\else\href{http://dx.doi.org/10.1007/s00222-014-0505-4}{\texttt{doi:10.1007/s00222-014-0505-4}}\fi.

\bibitem[HM15]{hairer2015discretisations}
\textsc{M.~Hairer} and \textsc{K.~Matetski}.
\newblock Discretisations of rough stochastic pdes.
\newblock \emph{arXiv preprint arXiv:1511.06937} (2015).

\bibitem[JLM85]{JonaMitterquantization}
\textsc{G.~Jona-Lasinio} and \textsc{P.~K. Mitter}.
\newblock On the stochastic quantization of field theory.
\newblock \emph{Comm. Math. Phys.} \textbf{101}, no.~3, (1985), 409--436.

\bibitem[LP66]{1966JMP}
\textsc{J.~L. {Lebowitz}} and \textsc{O.~{Penrose}}.
\newblock {Rigorous Treatment of the Van Der Waals-Maxwell Theory of the
  Liquid-Vapor Transition}.
\newblock \emph{Journal of Mathematical Physics} \textbf{7}, (1966), 98--113.
\newblock \ifx\href\undefined
  \texttt{doi:10.1063/1.1704821}\else\href{http://dx.doi.org/10.1063/1.1704821}{\texttt{doi:10.1063/1.1704821}}\fi.

\bibitem[MW15]{MourratWeberGlobal}
\textsc{J.-C. Mourrat} and \textsc{H.~Weber}.
\newblock Global well-posedness of the dynamic $\phi^4$ model in the plane.
\newblock \emph{To appear in Ann. Probab.} (2015).
\newblock \ifx\href\undefined
  \texttt{arXiv:1501.06191}\else\href{http://arxiv.org/abs/1501.06191}{\texttt{arXiv:1501.06191}}\fi.

\bibitem[MW16]{MourratWeber}
\textsc{J.-C. Mourrat} and \textsc{H.~Weber}.
\newblock Convergence of the two-dimensional dynamic {I}sing-{K}ac model to
  $\phi^4_2$.
\newblock \emph{Comm. Pure Appl. Math. online first} (2016).
\newblock \ifx\href\undefined
  \texttt{arXiv:1410.1179}\else\href{http://arxiv.org/abs/1410.1179}{\texttt{arXiv:1410.1179}}\fi.

\bibitem[Pre08]{presutti2008scaling}
\textsc{E.~Presutti}.
\newblock \emph{Scaling Limits in Statistical Mechanics and Microstructures in
  Continuum Mechanics}.
\newblock Theoretical and Mathematical Physics. Springer Berlin Heidelberg,
  2008.

\bibitem[Pro90]{Protter1990}
\textsc{P.~Protter}.
\newblock \emph{Stochastic Differential Equations},  187--284.
\newblock Springer Berlin Heidelberg, Berlin, Heidelberg, 1990.
\newblock \urlprefix\url{http://dx.doi.org/10.1007/978-3-662-02619-9_6}.

\bibitem[SW16]{ShenWeber}
\textsc{H.~Shen} and \textsc{H.~Weber}.
\newblock Glauber dynamics of 2d kac-blume-capel model and their stochastic pde
  limits.
\newblock \emph{arXiv preprint arXiv:1608.06556} (2016).

\bibitem[TW16]{tsatsoulis2016spectral}
\textsc{P.~Tsatsoulis} and \textsc{H.~Weber}.
\newblock Spectral gap for the stochastic quantization equation on the
  2-dimensional torus.
\newblock \emph{arXiv preprint arXiv:1609.08447} (2016).

\end{thebibliography}

\end{document}